\documentclass[12pt,leqno]{article}

\usepackage{amsthm}
\usepackage{amsfonts}
\usepackage{amssymb}
\usepackage{verbatim}
\usepackage[all]{xy}
\usepackage{amsmath, amscd, amsthm}
\usepackage{longtable}
\usepackage{amscd}
\usepackage{mathrsfs}
\usepackage{graphicx}
\usepackage{latexsym}
\usepackage{empheq}
\usepackage{tikz-cd}
\usepackage{eucal}

\newtheorem{Lemma}{\sc Lemma}  
\newtheorem{Theorem}{\sc Theorem}  
\newtheorem{Proposition}{\sc Proposition}  
\newtheorem{Corollary}{\sc Corollary}

\theoremstyle{remark}
\newtheorem{Remark}{\sc Remark} 

\newtheorem{Example}{\sc Example} 
\newtheorem{Case}{\sc Case}

\usepackage{latexsym}

\newcommand{\vp}{\varphi}
\newcommand{\op}{\oplus}
\newcommand{\ld}{\ldots}

\newcommand{\ve}{\varepsilon}

\newcommand{\GL}{\textrm{GL}}
\newcommand{\PGL}{\textrm{PGL}}

\newcommand{\pr}{\textrm{pr\:}}
\newcommand{\id}{\textrm{id\,}}

\newcommand{\Stab}{\textrm{Stab}}
\newcommand{\Scal}{\textrm{Scal}}
\newcommand{\scal}{\textrm{scal}}
\newcommand{\Def}{\textrm{Def}}
\newcommand{\default}{\textrm{def}}

\newcommand{\Grass}{\textrm{Grass\,}}
\newcommand{\wA}{\widetilde{A}}
\newcommand{\wP}{\widetilde{P}}

\newcommand{\B}{{\mathcal B}}
\newcommand{\cG}{\mathcal{G}}

\newcommand{\cI}{\mathcal{I}}
\newcommand{\cJ}{\mathcal{J}}
\newcommand{\cS}{\mathcal{S}}

\newcommand{\cF}{\mathcal{F}}
\newcommand{\cO}{\mathcal{O}}
\newcommand{\cP}{\mathcal{P}}
\newcommand{\cQ}{\mathcal{Q}}
\newcommand{\T}{\mathcal{T}}
\newcommand{\cR}{\mathcal{R}}
\newcommand{\X}{\mathcal{X}}
\newcommand{\Y}{\mathcal{Y}}
\newcommand{\fZ}{\mathcal{Z}}
\newcommand{\A}{\mathcal{A}}

\newcommand{\M}{\mathcal{M}}
\newcommand{\cN}{\mathcal{N}}
\newcommand{\cU}{\mathcal{U}}
\newcommand{\cV}{\mathcal{V}}
\newcommand{\cC}{\mathcal{C}}

\newcommand{\ba}[1]{\overline{#1}}

\DeclareMathOperator{\Hom}{Hom}

\newcommand{\Sp}{\mathrm{Span}}

\newcommand{\f}{{\mathfrak f}}
\newcommand{\gl}{\mathfrak{gl}}
\newcommand{\W}{{\mathfrak W}}

\newcommand{\fN}{{\mathfrak N}}
\newcommand{\V}{{\mathfrak V}}

\newcommand{\fS}{\mathcal{S}}

\newcommand{\U}{{\mathfrak U}}
\newcommand{\PS}{{\mathbf{P}}}

\newcommand{\N}{{\mathbb N}}
\newcommand{\Z}{{\mathbb Z}}
\newcommand{\F}{{\mathbb F}}

\newcommand{\R}{{\mathbb R}}

\newcommand{\chr}[1]{\mathrm{char}\left( #1 \right)}
\newcommand{\Aut}[1]{\mathrm{Aut}(#1)}
\newcommand{\Der}[1]{\mathrm{Der}(#1)}
\newcommand{\Ker}[1]{\mathrm{Ker}\left( #1 \right)}

\newcommand{\var}[1]{\mathrm{var}\left( #1 \right)}

\newcommand{\rank}{\mathrm{rank}\,}

\theoremstyle{definition}
\newtheorem{Definition}{\sc Definition}  
\theoremstyle{plain}
\begin{document}
\title{On generic properties of nilpotent algebras}
\author{Yuri Bahturin\thanks{Supported by NSERC Discovery grant \# 227060-19}\\Department of Mathematics and Statistics\\Memorial University of Newfoundland\\St. John's, NL, A1C5S7
\and Alexander Olshanskii\thanks{Supported by NSF grant DMS-1901976}\\Department of Mathematics\\1326 Stevenson Center\\Vanderbilt University\\Nashville, TN 37240}
\date{}
\maketitle

\begin{abstract}
We study general nilpotent algebras. The results obtained are new even for the classical algebras, such as associative or Lie algebras. We single out  certain \textit{generic} properties of finite-dimensional algebras, mostly over infinite fields. The notion of being generic in the class of $n$-generated algebras of an arbitrary primitive class of $c$-nilpotent algebras appears naturally in the following way. On the set of the isomorphism classes of such algebras one can  introduce the structure of an algebraic variety. As a result, the subsets are endowed with the dimensions as algebraic varieties. A subset $\Y$ of a set $\X$ of lesser dimension can be viewed as \textit{negligible} in $\X$. For example, if $n\gg c$, we determine that an  automorphism group of a generic algebra $P$ consists of the automorphisms, which are scalar modulo $P^2$. Generic ideals are in $I(P)$, the annihilator of $P$. In the case of classical nilpotent algebras as above, the generic algebras are graded by the degrees with respect to some generating sets. 
 \end{abstract}
\tableofcontents 
 \section{Introduction}\label{sI}

In this paper we are interested in the question: what is a ``generic'' group of automorphisms of a  finite-dimensional nilpotent algebra? Or: what is a ``generic''  Lie algebras of derivations of a finite-dimensional nilpotent algebra? One says that a property of mathematical objects of a class $\cC$ is \textit{generic} if ``almost every'' object of $\cC$ enjoys this property. The words ``almost every'' are understood differently, depending on the class $\cC$.

At the same time, the answer does not depend on the particular class of algebras: associative, Lie, Jordan, etc. One of the early relevant papers on nilpotent algebras in such generality was the paper of A. I. Malcev \cite{AIM}.

 In particular,  if several properties are generic then their (finite) intersection is also generic. Moreover, sometimes,  it is easier to prove that some property is generic than to provide examples. For instance, in the papers \cite{Be,Mil,OAY}, the estimate of the cardinality or the dimension of the set of counterexamples allows one to disprove certain conjectures without explicitly providing the counterexamples. 

In the classes of finite groups (rings, graphs, etc.), it is natural to say that almost all, up to isomorphism, groups have certain property $\cP$ if $\frac{f_\cP (n)}{f(n)}\to 1$ as $n\to\infty$ where $f_\cP (n)$ is the number of groups of order $n$ with property $\cP$ and $f(n)$ is the total number of groups of order $n$ in the class of all groups in question. For example, a known conjecture is that a generic finite group is a 2-group.

If we consider the properties of pairs of elements in the symmetric groups $S_n$, then   the classical Dixon's
theorem \cite{Dix}
says that a pair of random permutations of a set of $n$ elements 
almost surely generates either the symmetric group or the 
alternating group of degree $n$.

Infinite groups  (rings, etc.) often have finite description, say, in terms of generators and defining relations. However here the number $n$ restricting the order of the group above now  restricts the length of defining relation when the number of relations is fixed. So the following statements make sense. ``Almost any finitely presented group is hyperbolic'' \cite{OA} and furthermore ``almost every subgroup with lesser number of generators in such a group is free'' \cite{GVS,AO}.

One of the most common ways to make sense of the words ``almost all'' or ``generic'' is to introduce a probabilistic measure on the mathematical objects in question. An example : ``almost every symmetric random walk on the set of integers is recurrent''.

But how can one handle the notion of being ``generic'' in the case of finite-dimensional algebras over an infinite field?
To reach our goals in this paper, we start with a ``primitive'' class of algebras $\V$ over a field $\F$, defined by a set of identical relations, say the the class  of associative, or Lie algebras, or commutative,  etc. algebras.  

Then we restrict the number of generators in algebras from $\V$ by some number $n$ and obtain a subclass of $n$-generated algebras $\V_n$. The class $\V_n$ is too large, unless we impose further restrictions on $\V$. So we  restrict ourselves to \textit{nilpotent} primitive classes, that is, we assume that there is $c\ge 2$ such that $P^{c+1}=\{ 0\}$ for any algebra $P\in\V$. Now $\V_n$ consists of finite-dimensional algebras and all algebras in $\V_{n}$ are homomorphic images of one algebra $F_n= F_n(\V)$, called the \textit{free algebra of rank $n$} in $\V$, or \textit{$\V$-free algebra of rank $n$}. As a matter of fact, if $\F$ is infinite, as we assume throughout this paper except in Chapter \ref{sGAFF}, then $F_n(\V)$ is $\Z$-graded, hence finite-dimensional if and only if $\V$ is nilpotent. 

Every algebra $P\in \V_n$ has the form of $F_n/K$ where $K$ is an ideal of $F_n$. The set of all ideals is an algebraic variety, a subvariety in the abstract disjoint union of Grassmann varieties of subspaces of $F_n$. To study such sets, we can  use the notion of the dimension of an algebraic variety. Then the statement that some set $\A$ of ideals is ``almost the same'' as a bigger variety $\B$ simply means that the complement $\B\setminus \A$ has strictly smaller dimension than the dimension of $\B$. (According to this understanding, for instance, almost every point on the plane has its first coordinate nonzero.) Later, this comparison of the dimensions of the varieties of ideals is transferred to the factor-space of quotient-algebras, where we identify ideals with isomorphic quotient-algebras. For the rigorous definition of the notion ``generic'' see Definition \ref{dGeneric}.\footnote{Another approach to the structures of algebraic varieties, specifically, on the sets of the isomorphism classes of nilpotent algebras of fixed dimension is suggested in \cite{KN, NYA}. It could be interesting to study the generic properties from this point of view.}

Our main results in this paper are proven in Section \ref{sMR}, where we assume that the base field $\F$ is algebraically closed. Starting from Section \ref{ssARFNA}, we also restrict ourselves to \textit{central} primitive classes. Many important nilpotent primitive classes are central, including all $c$-step nilpotent associative, or all $c$-step nilpotent Lie, and so on, provided that $c\ge 2$.   

Below we list some generic properties. We fix a $c$-step nilpotent central primitive class $\V$, with $c\ge 2$. Then the following properties of  $n$-generated algebras in $\V$ are generic, provided that the number $n$  is greater than some $n_0=n_0(c)$. In the last claim the primitive class does not need to be central.
\begin{itemize}
\item Every automorphism of a generic algebra $P$ is scalar modulo $P^2$, that is, given an automorphism $\alpha$ of $P$, there exists $\lambda\in\F$ such that $\alpha(a)-\lambda a\in P^2$, for any $a\in P$. 
\item Every derivation of a generic algebra $P$ is scalar modulo $P^2$.
\item For every $c\ge 2$, almost all ideals in the relatively free $c$-step nilpotent algebra $F_n$ of rank $n>n_0$ are contained in the annihilator ideal of $F_n$.
\end{itemize}

In order to prove these results, we review and prove, where necessary, some results about the primitive classes of nonassociative algebras, in Section \ref{sPCA}. In a more detailed way, we treat the primitive classes of nilpotent algebras, in Section \ref{sNAV}. Section \ref{sLA} contains some rather general results from Linear algebra. We use them in Section \ref{sGRFNA} to obtain some crucial estimates on the dimension of ideals in relatively free nilpotent algebras. The results in the above five sections are true without stringent restrictions on the base field. We use them to derive generic properties of nilpotent algebras over finite fields in Section \ref{sGAFF}.

\section{Primitive classes of algebras}\label{sPCA} 

In the study of automorphisms and derivations of nilpotent algebras, our basic objects are \textit{relatively free nilpotent algebras}. To deal with them, in  Sections \ref{sPCA} and  \ref{sNAV} we quickly review some standard facts from the theory of primitive classes of general nonassociative algebras. Note that one of the first papers on the identical relations in general algebras  was \cite{AIM}. 

For the reader's convenience we give references to the respective theorems in the theory of primitive classes (varieties) of Lie algebras in the recent monograph \cite{B}.

\subsection{Absolutely free algebras}\label{ssAFA} For any field $\F$ and a nonempty set $X$, one defines an (absolutely) free algebra $\cF(X)$ as follows. We first define the \textit{nonassociative monomials} in $X$ by induction on their \textit{degree}. The monomials of degree 1 are the elements of $X$. We write $\deg x=1$. Once the monomials of all degrees less than $n$ have been defined, the monomials of degree $n$ are all possible expressions of the form $(u)(v)$ where $\deg u=k\ge 1$ and $\deg v=n-k\ge 1$. The set of all nonassociative monomials in $X$ is denoted by $\Gamma(X)$. One has $\Gamma(X)=\bigcup_{n=1}^\infty \Gamma_n$ where $\Gamma_n$ is the set of monomials of degree $n$. 

The set $\Gamma(X)$ is endowed with an operation $\circ$, as follows. If $u\in \Gamma_k$ and $v\in \Gamma_\ell$ then $u\circ v=(u)(v)\in\Gamma_{k+\ell}$. With respect to this operation, $\Gamma(X)$ is generated by $X$. In practice, while using this operation, one omits the symbol $\circ$ and all unnecessary parentheses. For instance, one writes $(x_1x_2)x_2$ instead of  
$((x_1)\circ(x_2))\circ(x_2)$.
 
Finally, one defines the free algebra $\cF(X)$ as the linear span of $\Gamma(X)$ with coefficients in $\F$. The elements of $\cF(X)$ are often called the \textit{nonassociative polynomials}. The product in $\cF(X)$ is the extension by linearity of the product in $\Gamma(X)$. 

If $\cF_n$ is the linear span of $\Gamma_n$ then $\cF_k\cF_{\ell}\subset \cF_{k+\ell}$ and so we have the canonical $\N$-grading of the free algebra, where $\N$ is the additive semigroup of natural numbers, as follows 
\begin{equation}\label{eGrading}
\cF(X)=\bigoplus_{n=1}^\infty\cF_n
\end{equation}

The elements of $\cF_n$ are called \textit{homogeneous} of degree $n$. The \textit{ideals} $\cF^n=\bigoplus_{k=n}^\infty \cF_k$ satisfy $\cF^n\cF^m\subset \cF^{m+n}$ and form a descending filtration in $\cF(X)$: 
\begin{equation}\label{ePSF}
\cF(X)=\cF^1\supset \cF^2\supset\cdots\supset \cF^n\supset\cdots\mbox{ such that }\bigcap_{n=1}^\infty \cF^n=\{ 0\}.
\end{equation}
 One also calls the terms of (\ref{ePSF}) the \textit{powers of $\cF(X)$} and the filtration itself the \textit{power} series of $\cF(X)$.

Now let $\Phi\subset\Z^X$ be the additive subsemigroup of the group of functions  $\alpha:X\to \Z$  consisiting of the functions with nonnegative values and finite positive norm $|\alpha|=\sum_X\alpha(x)$. By definition, a nonassociative monomial $u=u(x_1,\ld,x_m)$ is contained in $\Gamma_\alpha$ if the degree $d_i$ of $u$ with respect to each $x_i$ equals $\alpha(x_i)$ and $\deg u=|\alpha|$. One then can partition $\Gamma_n$ as the union of subsets $\Gamma_\alpha$ such that $|\alpha|=n$. One has $\Gamma=\bigcup_{\alpha\in\Phi}\Gamma_\alpha$ and $\Gamma_\alpha\Gamma_\beta\subset\Gamma_{\alpha+\beta}$. We denote by $\cF_\alpha$ the linear span of $\Gamma_\alpha$. Then $\cF_\alpha\cF_\beta\subset\cF_{\alpha+\beta}$ and one then has a $\Phi$-grading of $\cF(X)$: 
\begin{equation}\label{enGrading}
\cF(X)=\bigoplus_{\alpha\in\Phi}\cF_\alpha.
\end{equation}
If $|X|=n$ then we can identify $\Phi$ with a subsemigroup in $\Z^n$ and so (\ref{enGrading}) becomes a $\Z^n$-grading, or more precisely a $\N_0^n$-grading where $\N_0=\N\cup\{ 0\}$. The elements of $\cF_\alpha$ are called \textit{multihomogeneous} of \textit{multidegree} $\alpha$. If $\alpha(x)$ is either $0$ or $1$, for every $x\in X$, then the elements of $\cF_\alpha$ are called \textit{multilinear} with respect to the set of the variables $x$ for which $\alpha(x)=1$.

\subsection{Primitive classes and relatively free algebras}\label{ssPC}

Given a nonassociative monomial $u(x_1,\ld,x_n)$ and the elements $a_1,\ld,a_n$ of  any algebra $P$ over any field $\F$, one can use the induction on $\deg u$  and the operation of $P$ to define the \textit{value} $u(a_1,\ld,a_n)\in P$.  The value of a monomial $x_i$ is just $a_i$, while $(uv)(a_1,\ld,a_n)=u(a_1,\ld,a_n)v(a_1,\ld,a_n)$. Now the value of a nonassociative polynomial $w(x_1,\ld,x_n)$ which is a linear combination of several monomials with coefficients in $\F$,  is just the linear combination of the values of these monomials with the same coefficients. 

If $\vp:X\to P$ is any mapping then the evaluation of nonassociative polynomials which replaces each $x\in X$ by $\vp(x)\in P$ is a unique homomorphism of algebras $\bar\vp:\cF(X)\to P$, extending $\vp$. This property is called the \textit{universal property} of $\cF(X)$. An immediate consequence of this property is that given any algebra $P$  there exists $X$ such that $P$ is a homomorphic image of $\cF(X)$ under a homomorphism $\ve$ mapping $X$ onto a generating set of $P$. Taking $K=\Ker\ve$, we get $P\cong \cF(X)/K$.

If $w(a_1,\ld,a_n)=0$, for all $a_1,\ld,a_n\in P$, we say that $w(x_1,\ld,x_n)=0$ is an \textit{identical relation} or shortly an \textit{identity} in $P$. 

Now let us set $\cF_\infty=\cF(x_1,\ld,x_n,\ld)$ over a field $\F$. Given a nonempty subset $V\subset\cF_\infty$, the class $\V$ of all algebras (over fixed field $\F$) satisfying all identities $v(x_1,\ld,x_n)=0$ where $v(x_1,\ld,x_n)\in V$ is called the \textit{primitive class defined by $V$}. By Birkhoff's Theorem \cite[Theorem 4.3]{B} a nonempty class $\V$ of algebras is primitive if and only if $\V$ is closed under subalgebras, quotient-algebras and Cartesian products of algebras.

Given a primitive class $\V$ and algebra $P$, one considers the ideal $\V(P)$ (more often denoted by a matching notation $V(P)$) of all values $v(a_1,\ld,a_n)$, where $a_1,\ld,a_n\in P$ and $v(x_1,\ld,x_n)=0$ is an identity in $\V$. By analogy with the group-theoretic notation, one calls $V(P)$ \textit{the verbal ideal corresponding to the primitive class $\V$}. $V(P)$ is an least ideal of $P$ such that $P/V(P)\in\V$. The ideal $V(P)$ is closed under all endomorphisms of $P$. Such ideals are called \textit{fully invariant}.

If $\F$ is a fixed field, then there is one-one Galois-type correspondence between the primitive classes over $\F$ and fully invariant ideals of $\cF_\infty$:
 
\begin{Proposition}\label{pGC}
Given a primitive class $\V$, the corresponding verbal ideal $V(\cF_\infty)$ is fully invariant. Conversely, every fully invariant ideal is a verbal ideal for an appropriate primitive class $\V$.$\hfill\Box$
\end{Proposition}  

If $U\subset\cF_\infty$ and $P$ is an algebra, one denotes by $U(P)$ the ideal of $P$ generated by all $u(a_1,\ld,a_n)$, where $a_1,\ld,a_n\in P$ and $u(x_1,\ld,x_n)\in U$. If $U=\{ w\}$ then it will be convenient to write $U(P)=w(P)$. This is the \textit{verbal ideal of $P$ defined by the set of nonassociative polynomials $U$.} If $U(\cF_\infty)=\V(\cF_\infty)$ for some variety $\V$ then one says that the set of identities $\{ u=0\,|\,u\in U\}$ is a \textit{basis of identities} of $\V$. Any algebra $P$ in which  $u(x_1,\ld,x_n)=0$ for any $u\in U$ also satisfies all identities of $\V$.

Given a nonempty set $X$, the algebra $F(\V,X)=\cF(X)/V(\cF(X))$ is called a \textit{$\V$-free algebra with the free generating set $X$} or a \textit{ free algebra in $\V$ with the free generating set $X$}. If $\V$ is nonzero, $X$ is injectively embedded in $F(\V,X)$. Each of these algebras possesses the \textit{universal property within the primitive class $\V$}, namely, any map $\vp:X\to P\in\V$ extends to a unique homomorphism $\overline{\vp}:F(\V,X)\to P$. It follows that every algebra $P\in\V$ is isomorphic to a quotient-algebra of an appropriate $F(\V,X)$. If we do not specify parameters $\V$ or $X$, then we simply say that $P$ is a \textit{relatively free algebra}.

Examples of relatively free algebras, that is free in some varieties, include free associative algebras, free Lie algebras, free commutative algebras, etc. In this paper we will be mostly interested in \textit{relatively free nilpotent algebras}. 

A useful characterization of relatively free algebras is the following. 

\begin{Lemma}\label{lRFA}An algebra $P$ generated by a set $X$ is relatively free with $X$ a free generating if and only if any map $\vp:X\to P$ extends to a  homomorphism $\bar{\vp}:P\to P$.
$\hfill\Box$
\end{Lemma} 

If $P$ satisfies the extension property in the above Lemma, then $P$ is free in the primitive class $\V$ \textit{generated by} $P$  that is, the least primitive class of algebras over $\F$ containing $P$. In this case, we write $\V=\var P$.  It follows by Birkhoff's Theorem (see \cite[Theorem 4.3]{B}) that $\V$ consists of the quotient-algebras of the subalgebras of the Cartesian powers of $P$.

One more useful property of relatively free algebra, following from the universal property, is given below.

\begin{Lemma}\label{lRelX} Let $0\ne v(x_1,\ld,x_n)\in\cF_\infty$ be a nonassociative polynomial and $F=F(\V,Y)$ a  free algebra in a primitive class $\V$ of algebras, with a free generating set $Y=\{ y_1,y_2,\ld\}$. Suppose $v(y_1,\ld,y_n)=0$ is a \emph{relation} among $y_1,\ld,y_n$ that holds in $F$. Then $v(x_1,\ld,x_n)=0$ is an \emph{identical relation} in $\V$. 
$\hfill\Box$
\end{Lemma}

Next we mention the following.

Any map of a set $X$ to a set $Y$ extends to a unique homomorphism $F(\V,X)\to F(\V,Y)$. In particular, the embedding  $X\subset Y$ extends to the embedding  $F(\V,X)\subset F(\V,Y)$ and  $F(\V,X)$ is a retract of $F(\V,Y)$ under the projection extending $x\mapsto x$ for $x\in X$ and $x\mapsto 0$ for $x\in Y\setminus X$.

The cardinality of $X$ is called the \textit{rank} of $F(\V,X)$. Two relatively free algebras in a non-trivial primitive class $\V$, that is, $\V$ contains a nonzero algebra, are isomorphic is and only if they have the same rank. So, with $\V$ fixed, we can simply write $F_n=F_n(\V)$ for the relatively free algebra of rank $n$ in the variety $\V$. In this case, we can choose $X=\{ x_1,\ld,x_n\}$.

\medskip

We conclude this section by a rarely mentioned property of the homomorphisms of algebras in a primitive class $\V$ of algebras over a field $\F$. It is true not only for the primitive classes of algebras over a field but also for the primitive classes of universal $\Omega$-algebras, with arbitrary signature $\Omega$.

\begin{Proposition}\label{Scrap8}Let $F=F(\V,X)$ be a free algebra in a primitive class $\V$ of algebras. Let $P$ and $Q$ be two algebras in $\V$. Suppose we are given three homomorphisms $\alpha: F \to P,\, \beta: F\to Q,\, \gamma: P\to Q$ such that $\beta$ is surjective. Then there exists a homomorphism $\delta:F\to F$ such that $\gamma\alpha = \beta\delta$. 
\end{Proposition}

\begin{proof}
 \[
  \begin{tikzcd}
    F \arrow[dashrightarrow]{r}{\delta} \arrow{d}{\alpha} & F\arrow{d}{\beta}\\ P\arrow{r}{\gamma} & Q
  \end{tikzcd}
\]
For each $x\in X$ we choose $a_x\in F$ such that $\beta(a_x)=\gamma\alpha(x)$. Since $X$ is a free generating set for $F$, the map from $X$ to $F$ given by $x\mapsto a_x$ extends to a homomorphism $\delta:F\to F$. Now two homomorphisms $\beta\delta$ and $\gamma\alpha$ from $F$ to $Q$ coincide on $X$. As a result, $\gamma\alpha = \beta\delta$. 
\end{proof}

\subsection{Multihomogeneous primitive classes}\label{ssMHPC}

In Section \ref{ssAFA} we noticed that the free algebra $\cF(X)$ has natural gradings by degrees and multidegrees with respect to the free generating set $X$. This important property does not hold for general relatively free algebras. For instance, this is not true in the case of any primitive class $\var P$ generated by a finite non-nilpotent algebra $P$. So one defines \textit{multihomogeneous primitive classes}, as follows.
\begin{Definition}\label{dMV} A primitive class $\V$ is called \textit{multihomogeneous} if given an identity $v(x_1,\ld,x_n)=0$ holding in $\V$, it is true that also $v_\alpha(x_1,\ld,x_n)=0$ is an identity in $\V$, for any multihomogeneous component $v_\alpha(x_1,\ld,x_n)$ of $v(x_1,\ld,x_n)$. In other words, $V(\cF(X))$ is graded with respect to the multihomogeneous grading (\ref{enGrading}) of $\cF(X)$, where $X=\{x_1, x_2,...,x_n,\ld\}$. 
\end{Definition}
For any natural $n$ or any $\alpha\in\Phi$ one can define $F(\V,X)_n$ or $F(\V,X)_\alpha$ as the image of $\cF(X)_n$ or $\cF(X)_\alpha$ under the natural homomorphism $\ve$ extending the identity map $X\to X$. If $\V$ is multihomogneous, then the kernel $V(\cF(X))$ of $\ve$ is graded. This allows us to apply (\ref{eGrading}), (\ref{enGrading}) and (\ref{ePSF}). As a result, we have the following: 
\begin{equation}\label{eRelGrad}
F(\V,X)=\bigoplus_{n=1}^\infty F(\V,X)_n,\;F(\V,X)=\bigoplus_{\alpha\in\Phi}(F(\V,X))_\alpha,
\end{equation}
and
\begin{equation}\label{eApprox}
\bigcap_{i=1}^\infty F(\V,X)^n=\{0\}.
\end{equation}
Now the following is true.
\begin{Lemma}\label{L1.3} If $\F$ is an infinite field then any primitive class $\V$ of algebras over $\F$ is multihomogeneous. For a $\V$-free algebra $F(\V,X)$ the equations (\ref{eRelGrad}) and (\ref{eApprox}) hold valid. 
$\hfill\Box$
\end{Lemma}

See the proof in \cite[\S 5]{AIM}. 

Of course, there are multihomogeneous varieties even in the case of finite fields. These are, for example, the primitive classes defined by the identities  $v(x_1,\ld,x_n)=0$, where $v(x_1,\ld,x_n)$ is multilinear, as defined earlier, at the end of Section \ref{ssAFA}. Such identities are called \textit{multilinear}. If one deals only with the fields of characteristic 0, then the following is true. 

\begin{Proposition}\label{pMultilin} If $\chr\F=0$ then any primitive class $\V$ of algebras over $\F$ can be defined by a set of multilinear identities. 
$\hfill\Box$
\end{Proposition}

See the proof in \cite[\S 5]{AIM}.

\subsection{Derivations of relatively free algebras}\label{ssDRFA}

Recall that a linear map $d:P\to P$ of an algebra $P$ is called a \textit{derivation} if for any $a,b\in P$ one has $d(ab)=d(a)b+ad(b)$. 
 
We noticed earlier, see Proposition \ref{pGC}, that the verbal ideals of algebras are closed under their endomorphisms. This  does not need to be true for their derivations, even in the case of $\cF(X)$.  

Weare going to show now  that this is still true in the case where $\V$ is multihomogeneous. We start with a simple fact about the absolutely free algebras $\cF(X)$.
\begin{Lemma}\label{lDAFA}
Any map $\vp:X\to\cF(X)$ extends to a derivation $d_\vp:\cF(X)\to\cF(X)$.
\end{Lemma}
 \begin{proof}
  The desired extension of $\vp$ to a derivation $d_\vp:\cF(X)\to\cF(X)$ can be obtained, as follows. For $x\in X$ one has to set $d_\vp(x)=\vp(x)$ and for $w=uv$ of degree greater than 1, using induction, set $d_\vp(w)=d_\vp(u)v+ud_\vp(v)$. Using linearity, one expands $d_\vp$ to the whole of $\cF(X)$. The defining property of the derivation easily follows. 
\end{proof} 

Now if $F=F(\V,X)$ is a free algebra in the primitive class $\V$, one gets the same result provided that $V(\cF(X))$ is invariant under $d_\vp$. In Lemma \ref{lDer} we prove that the latter property holds valid in the case where $\V$ is a multihomogeneous primitive class. 

\begin{Remark}\label{rQ} Note that in the case of a finite field $\F$ the extendibility as in Lemma \ref{lDAFA} does not take place. Indeed, choose $\F$ with $|\F|=q$ and $\V$ the primitive class of associative and commutative algebras with an additional identity $x^q-x=0$. Let $F$ be $\V$-free algebra with one free generator $x$. Then the map $x\mapsto x$  does not extend to a derivation of $F$.
\end{Remark}

\begin{Lemma}\label{lDer} If $X$ is an arbitrary nonempty set, $\V$ a multihomogeneous primitive class over a field $\F$ then $V(\cF(X))$ is invariant under all derivations of $\cF(X)$. 
\end{Lemma}
\begin{proof} Let $D$ be a derivation of $\cF(X)$ and $u(x_1,\ld,x_n)\in V(\cF(X))$. We want to show that $D(u(x_1,\ld,x_n))\in V(\cF(X)$. In the language of identities, we need to show that if $u(x_1,\ld,x_n)=0$ is an identity in $\V$ then also $D(u(x_1,\ld,x_n))=0$ is an identity in $\V$. This reformulation allows one to restrict oneself to the case where $X$ is finite set: $X=\{x_1,\ld,x_n\}$. If we embed $X$ in a countable set $X_\infty=\{x_1,\ld,x_n,\ld\}$ then $\cF(X)$ is a retract in $\cF(X_\infty)$. In particular, this means that $V(\cF(X))=V(\cF(X_\infty))\cap\cF(X)$. So it is enough to prove $D(u(x_1,\ld,x_n))\in\cF(X_\infty)$. Since $\V$ is multihomogeneous, if $u(x_1,\ld,x_n)=0$ is an identity in $\V$, then also its multihomogeneous components are identities of $\V$.  So  we only need to show $D(u(x_1,\ld,x_n))\in\cF(X_\infty)$ in the case where $u(x_1,\ld,x_n)$ is multihomogeneous. 

Since every derivation of $\cF(x_1,\ld,x_n)$ is completely defined on the set of the free generators, we can write $D$ as the sum of derivations $D= D_1+\cdots+D_n$ where each derivation $D_i$ is given by $D_i(x_j)=\delta_{ij}D(x_i)$. So it is enough to prove $D(u(x_1,\ld,x_n))\in V(\cF(x_1,\ld,x_n))$ when $D=D_i$, for some $i=1,\ld,n$, say for $i=1$. 

We consider the endomorphism $\psi:\cF_\infty\to \cF_\infty$ given on the free generators by $\psi(x_{n+1})=D(x_1)$ and $\psi(x_i)=x_i$, for all the other $x_i$. Then we consider the derivation $\delta:\cF_\infty\to \cF_\infty$ given on the free generators by $\delta(x_1)=x_{n+1}$ and $\delta(x_i)=0$,  for all the other $x_i$. By induction with trivial basis, on all elements $v=v(x_1,\ld,x_n)$ of $\cF$, we have $(\psi\delta)(v)=D(v)$. Indeed, if $v=v_1v_2$ then 
\[
\psi(\delta(v_1v_2)=(\psi\delta)(v_1)\psi(v_2)+\psi(v_1)(\psi\delta)(v_2)=\delta(v_1)v_2+v_1\delta(v_2)=\delta(v_1v_2).
\]
By induction with trivial basis, let us show that $\delta(u(x_1,\ld,x_n))$ is the multihomogeneous component $\hat{u}$ of $u(x_1+x_{n+1},x_2,\ld,x_n)$ of degree 1 in $x_{n+1}$. This is enough to show when $u=u(x_1,\ld,x_n)$ is a monomial. So if $u=u_1u_2$ then the monomials of degree 1 in $u_1(x_1+x_{n+1},x_2,\ld,x_n)u_2(x_1+x_{n+1},x_2,\ld,x_n)$ come in two ways: choosing monomials of degree 1 in $x_{n+1}$ from the first factor, that is replacing $u_1$ by $\widehat{u_1}$ and leaving $u_2$ as  is, or leaving $u_1$ as  is and replacing $u_2$ by $\widehat{u_2}$. So, 
\[
\widehat{u}=\widehat{u_1}u_2+u_1\widehat{u_2}=\delta(u_1)u_2+u_1\delta(u_2)=\delta(u).
\]
Now if $u(x_1,\ld,x_n)\in V(\cF_\infty)$ then $u(x_1+x_{n+1},\ld,x_n)\in V(\cF_\infty)$. Since $\V$ is multihomogeneous, $\widehat{u}\in V(\cF_\infty)$, hence $\delta(u(x_1,\ld,x_n))\in V(\cF_\infty)$, hence $\psi(\delta(u(x_1,\ld,x_n)))\in V(\cF_\infty)$ and finally 
\[
D(u(x_1,\ld,x_n))\in V(\cF_\infty))\cap\cF(x_1,\ld,x_n)=V(\cF(x_1,\ld,x_n),
\]
 as claimed.
\end{proof}

\begin{Corollary}\label{cDer}
If $F=F(\V,X)$ is a  $\V$-free algebra in a multihomogeneous primitive class $\V$ then any map $\vp:X\to F$ extends to a unique derivation $d_\vp$ of $F$.
$\hfill\Box$
\end{Corollary} 

Using the argument of Proposition \ref{Scrap8} allows us to prove the following result about the derivations of \textit{arbitrary} algebras in the multihomogeneous primitive classes.

\begin{Proposition}\label{pID}
Let $P$ be an algebra over a field $\F$ in a multihomogeneous primitive class $\V$. If $P=F/I$, where $F$ is a $\V$-free algebra and $\ve$ is a natural epimorphism $F\to P$ with kernel $I$ then for any derivation $d$ of $P$ there is a derivation $\ba{d}$ of $F$ such that the following diagram is commutative
\begin{equation}\label{diag1}
  \begin{tikzcd}
    F \arrow[dashrightarrow]{r}{\ba{d}} \arrow{d}{\ve} & F\arrow{d}{\ve}\\ P\arrow{r}{d} & P
  \end{tikzcd}
\end{equation} 
In particular, $I$ is invariant under $\ba{d}$.  
\end{Proposition}

\begin{proof} For each $x\in X$ we choose $a_x\in F$ such that $\ve(a_x)=d\ve(x)$. Since $X$ is a free generating set for $F$ and $F$ is free in a multihomogeneous primitive class $\V$, by Corllary \ref{cDer}, there exist a derivation $\ba{d}:F\to F$ extending the map $\vp:X \to F$ given by $\vp(x)=a_x$, for any $x\in X$. Now two maps $\ve \ba{d}$ and $d\ve$ from $F$ to $B$ coincide on $X$. Since $\V$ is multihomogeneous, $F$ is graded by total degrees in $X$. This allows us to use induction with an obvious basis by the degrees of the monomials in $X$ to prove that $\ve \ba{d}=d\ve$ on the whole of $F$. Indeed, if we already knew that these two maps coincide on all monomials of degree less than $n$ and $w=uv$ is a monomial of degree $n>1$ then
\begin{eqnarray*}
\ve \ba{d}(uv)&=&\ve(\ba{d}(u)v+u\ba{d}(v))=\ve(\ba{d}(u))\ve(v)+\ve(u)\ve(\ba{d}(v))\\&=&d(\ve(u))\ve(v)+\ve(u)d(\ve(v))=d(\ve(u)\ve(v))=d\ve(uv),
\end{eqnarray*}
as needed. 
\end{proof}
  
\section{Primitive classes of nilpotent algebras}\label{sNAV}

If there is $m$ such that the value of every nonassociative monomial of degree $m$ in an algebra $P$ is zero, then
$P$ is called \textit{nilpotent}. The linear span of
the values of monomials of degree $i$ in $P$ is denoted by $P^i$. One says that $P$ is $c$-step nilpotent if $P^{c+1}=\{0\}$, but $P^c\ne \{0\}$. For each $c\ge 0$, the class $\mathfrak{N}_c$ of all nilpotent algebras $P$ satisfying $P^{c+1}=\{ 0\}$ is primitive. The relatively free algebras in $\mathfrak{N}_c$ are called \textit{free $c$-step nilpotent algebras}. One has $F(\mathfrak{N}_c,X)\cong\cF(X)/(\cF(X))^{c+1}$. The basis of $F(\mathfrak{N}_c,X)$ is formed by the (images) of all nonassociative monomials of degree at most $c$.

In any algebra $P$ we have two series of ideals: the descending power series  
\begin{equation}\label{ePS}
P=P^1\supset P^2\supset\cdots P^n\supset\cdots
\end{equation}
 and the ascending \textit{annihilator series}, defined as follows. For any algebra $P$, we define the ideal $I(P)$ called the \textit{annihilator} of $P$ by setting 
\[
I(P)=\{ x\in P\,|\,(\forall a\in P)(ax=xa=0)\}.
\]
Then we define $I_n(P)$ by induction on $n$ if we set $I_1=I(P)$ and define $I_n(P)$ as the full preimage of $I(P/I_{n-1}(P))$ in $P$, $n>1$. As a result, we have an ascending annihilator series of ideals of $P$:
\begin{equation}\label{eAS}
I_1(P)\subset I_2(P)\subset\cdots\subset I_n(P)\subset\cdots
\end{equation} 
In the case of a $c$-step nilpotent algebra  $P$, both series have finite length and $P^k\subset I_{c-k+1}(P)$, for any $k=1,2,\ld,c$.

\begin{Remark}\label{rPowvsAnn}
In distinction  with the lower and upper central series in groups or Lie algebras, the length of the annihilator series can be shorter than the length of the power series. A quick example is the quotient-algebra $P$ of the free $4$-step nilpotent algebra $F(x,y,z,t)$ by the ideal generated by all monomials of degree 4 except $(xy)(zt)$. One has $I_1(P)=P^3$, $I_2(P)=P^2$, $I_3(P)=P$. At the same time, $P^4\ne\{0\}$. 
\end{Remark}

 \subsection{Homomorphisms of relatively free nilpotent algebras}\label{ssHA}
 
We start with two lemmas, whose proofs copy can be found in \cite[\S 3]{AIM}.
 
 \begin{Lemma}\label{lBPNA}
If a nilpotent algebra $P$ is generated by $S\cup P^2$, where $S$ a subset of $P$, then $P$ is generated by $S$ alone. 
 $\hfill\Box$
\end{Lemma}
 \begin{Lemma}\label{lFSRFL}
Let $F=F(\V,X)$ be a free algebra of a  primitive class $\V$ over a field $\F$, $Y$ a subset of $F$ linealry independent modulo $F^2$ and $G$ be a subalgebra of $F$ generated by $Y$. Then $G\cong F(\V,Y)$. 
 $\hfill\Box$
\end{Lemma} 
 \begin{Corollary}\label{cGFS} 
 \begin{enumerate} Let $F=F(\V,X)$ be a free algebra of a  nilpotent  primitive class $\V$
 \item If a subset $S\subset F$   is a basis of $F$ modulo $F^2$ then $S$ is a free generating set of $F$
 \item Let $\alpha$ be an endomorphism of $F$. Let $\gamma$ be the endomorphism of $F/F^2$ induced by $\alpha$. If $\gamma$ is an automorphism then the same is true for $\alpha$.$\hfill\Box$
\end{enumerate}  
 \end{Corollary}
 
 \begin{Remark}\label{rScrap29.1} If $F_n$, $n\ge 2$, is a $\V$-free algebra in a $c$-step nilpotent primitive class $\V$, $c\ge 2$, then $I(F_n)\subset F_n^2$. Indeed if $x\in I(F_n)\setminus F^2$ then by Corollary \ref{lFSRFL},  $x$ can be included in a  free generating set of $F_n$. Since $xy=0$ for all $y\in F_n$, including a free generator $y$. By Lemma \ref{lRelX} $xy=0$ is an identity in $\V$, contradicting to $c\ge 2$.
 \end{Remark}
 
One of the main tools of dealing with the automorphisms of nilpotent algebras is the following result of A. I. Malcev \cite[\S 3]{AIM}.

\begin{Lemma}[Malcev]\label{tScrap8.4} Let $c\ge 1$ and $\V$ be a $c$-step nilpotent primitive class of  algebras over a field $\F$. Let $F=F(\V,X)$ be a $\V$-free algebra with the free generating set $X$. Then the following are true.
\begin{enumerate}
\item[(1)] Let $P$ be an image of $F$ under  an epimorphism $\ve:F\twoheadrightarrow P$. Then any automorphism $\gamma$ of $P$ is induced by an automorphism $\ba{\gamma}$ of $F$, that is $\ve\ba{\gamma}=\gamma\ve$.
\item[(2)] If $K,L$ are two ideals of a $c$-step nilpotent relatively free algebra $F$ of finite rank then $F/K\cong F/L$ if and only if there is an  automorphism  $\delta$ of $F$ such that $\delta(K)=L$.$\hfill\Box$
\end{enumerate}
\end{Lemma}

\begin{Remark}\label{rScrap8.5} The same argument applies to any group $G$ in any primitive class $\mathfrak{N}$ of nilpotent groups or locally finite $p$-groups. However, we must assume one of $N\subset F^2$ or $N\subset\Phi(F)$, the Frattini subgroup of $F$.
\end{Remark}

In our future sections we will be viewing the last term $F^c_n$ of a relatively free algebra $F_n(\V)$ of rank $n$ as a module for the group $\GL(n)$ and Lie algebra $\gl(n)$. 

Given a $\V$-free algebra $F_n=F_n(\V)$ of a $c$-step nilpotent primitive class $\V$, all terms $F_n^k$, $k=1,\ld,c$ are invariant under all automorphisms and all derivations.  Now $\Aut{F_n/F_n^2}\cong\GL(n)$, as groups, while $\Der{F_n/F_n^2}\cong \gl(n)$, as Lie algebras. By Lemma \ref{lRFA}, each $A\in\GL(n)$ extends to an automorphism $\alpha$ of $F_n$. Also, if $\V$ is multihomogeneous, by Corollary \ref{cDer}, each $D\in\gl(n)$ extends to a derivation $d$ of $F_n$. If an automorphism is trivial on $F_n/F_n^2$ then it is also trivial on $F_n^c$. If a derivation is zero on $F_n/F_n^2$ then it is zero on $F_n^c$. This allows us to define the action of  $A\in\GL(n)$ on $F_n^c$ as $\wA$ in the following: 
\begin{equation}\label{ewA}
\wA w=\alpha(w),\mbox{ where }w\in F^c.
\end{equation}
If $\V$ is mulihomogeneous, then $D\in\gl(n)$ acts on $F_n^c$ as $\widetilde{D}$ in the following:
\begin{equation}\label{ewAD}
\widetilde{D} w=d(w),\mbox{ where }w\in F^c.
\end{equation}

We summarize the above in the following.
  
\begin{Proposition}\label{lGLFc} Let $F_n$ be a $\V$-free $c$-step nilpotent algebra of rank $n$. Then (\ref{ewA}) gives a well-defined group action of $\GL(n)$ on $F_n^c$. Also, if $\V$ is multihomogneous, (\ref{ewAD}) gives a well-defined  Lie algebra action of $\gl(n)$ on $F^c$.$\hfill\Box$
\end{Proposition}

\begin{Remark}\label{rvarFc} In conclusion of this section, we mention the following property of the $c$-step nilpotent primitive class $\V$: $\V=\var {F_c(\V)}$. Indeed, obviously $\var {F_c(\V)}\subset \V$. Conversely, let $w=0$, $w\in\cF_\infty$, be an identity that is satisfied in $F_c(\V)$ but not in $\V$. Suppose that the number of of monomials in the expression of $w$ through the free generating set of $\cF_\infty$ is the least possible. Each monomial depends on at most $c$ generators. Let us fix the sum of monomials depending on some free generators $x_1,\ld,x_c$. If we replace all other free generators by 0, we will obtain a consequence $u=0$ which is satisfied in $F_c(\V)$, and also in $\V$. The difference $w-u$ has shorter expression through the monomials and so $w-u=0$ must be satisfied in $\V$. Since $u=0$ is satisfied in $\V$, we must have $w=0$ satisfied in $\V$, which is a contradiction.
\end{Remark} 
\subsection{Some special automorphisms and derivations of nilpotent algebras}\label{ssDefault} The group $\Aut{P}$ naturally acts on the set of ideals of an algebra $P$. Here we present special automorphisms and derivations that naturally appear in the context of  this paper. 

\begin{Definition}\label{dScal}  Given $K\in\cJ_m(P)$, we denote by $\Stab (K)=\Stab_{\Aut{P}}K$ the stabilizer of $K$ in $\Aut{F_n}$, that is, the subgroup of all $\delta\in\Aut{F_n}$ such that $\delta(K)=K$. Next, given an algebra $P$ over a field $\F$, we denote by $\Scal(P)$ the normal subgroup  of  $\Aut{P}$ consisting of the automorphisms $\delta$, which are scalar modulo $P^2$. This condition means that there exists $\lambda\in \F$ such that $\delta(a)-\lambda a\in P^2$, for all $a\in P$. We call the automorphisms in $\Scal(P)$ the \textit{$S$-automorphisms} while those outside $\Scal(P)$ the \textit{$NS$-automorphisms}. In a similar manner one defines the Lie ideal $\scal(P)\subset\Der{P}$ as consisting of derivations which are scalar mod $P^2$.
\end{Definition}

The importance of Definition \ref{dScal} will be seen in Section \ref{ssARFNA} where it will be shown that $\Stab(K)\subset\Scal(F_n)$ is a generic property for the ideals $K$ of $\V$-free nilpotent algebras $F_n$ in a central  primitive class $\V$ (see Definition \ref{dCentral}).

In the next example we present the default automorphisms and derivations of $2$-step nilpotent algebras. 

\begin{Example}\label{rtwostep}
Let $\alpha$ be an $S$-automorphism of a $2$-step nilpotent algebra $P$. Then there exist a vector space decomposition $P=Q\oplus P^2$, $\lambda\ne 0$ and a linear map $f:Q\to P^2$ such that $\alpha(x)=\lambda x +f(x)$, for any $x\not\in P^2$. Also, $\alpha(y)=\lambda^2 y$, for any $y\in P^2$. Conversely, given  a vector space decomposition $P=Q\oplus P^2$, a linear map $f:Q\to P^2$ and $0\ne\lambda\in\F$,  map $\alpha_{f,\lambda}:P\to P$ given by
\begin{equation}\label{edefaultaut}
\alpha_{f,\lambda}(x)=\lambda x+f(x),\mbox{ for }x\in Q,\mbox{ and }\alpha_{f,\lambda}(y)=\lambda^2 y\mbox{ if }y\in P^2
\end{equation} 
 is an automorphism of $P$. Since this is an automorphism of any $2$-step nilpotent algebra, we call $\alpha_{f,\lambda}$ a \textit{default automorphism} of $P$.  The set   of all default automorphisms is a normal subgroup $\Def(P)=\Scal(P)$ in $\Aut{P}$. We will show in Theorem \ref{36}, Claim (ix) that for a generic $2$-step nilpotent algebra $P$, one has $\Aut{P}=\Def(P)$. 
 Note here that $\Def(P)=\{\alpha_{f,\lambda}\,|\,f\in\Hom(Q,P^2),\lambda\in\F\}$ can be written in terms of the matrices of order $m=\dim P$, once we choose a basis in $Q$ and complement it to the basis of $P$:
\begin{equation}\label{edefmat}
A_{f,\lambda}=\begin{pmatrix}
\lambda I_n&0\\R&\lambda^2 I_{m-n}
\end{pmatrix}.
\end{equation}
Here $R$ is a matrix of $f$ with respect to the chosen bases of $Q$ and $P^2$.

In a similar way, one treats the derivation of a $2$-step nilpotent algebra $P$. If $\delta$ is an $S$-derivation of $P$ then there is a vector space decomposition $P=Q\oplus P^2$, $\lambda\in\F$ and a linear map $f:Q\to P^2$ such that $\delta(x)=\lambda x+f(x)$, $\delta(y)=2\lambda y$, if $y\in P^2$. Conversely, given a vector space decomposition $P=Q\oplus P^2$,  $\lambda\in\F$, and a linear map $f:Q\to P^2$, one can define a derivation $\delta_{f,\lambda}:P\to P$, given by
 \begin{equation}\label{edefaultdef}
\delta_{f,\lambda}(x)=\lambda x+f(x),\mbox{ for }x\in Q,\mbox{ and }\delta_{f,\lambda}(y)=2\lambda y\mbox{ if }y\in P^2.
\end{equation} 
One calls $\delta_{f,\lambda}$ a  \textit{default derivation} of $P$. The set of all default derivations forms a Lie ideal $\default(P)=\scal(P)$ in $\Der{P}$.    

We will show in Theorem \ref{36}, Claim (x) that for a generic $2$-step nilpotent algebra $P$, one has $\Der{P}=\default(P)$. As earlier for the automorphisms, note that $\default(P)=\{\alpha_{f,\lambda}\,|\,f\in\Hom(Q,P^2)\}$ can be written in terms of the matrices of order $m=\dim P$, once we choose a basis in $Q$ and complement it to the basis of $P$:
\begin{equation}\label{edefmatder}
A_{f,\lambda}=\begin{pmatrix}
\lambda I_n&0\\R&2\lambda I_{m-n}
\end{pmatrix}.
\end{equation}
Again, $R$ is a matrix of $f$ with respect to the chosen bases of $Q$ and $P^2$.
\end{Example}

\subsection{Verbal ideals of relatively free nilpotent algebras}\label{ssTNA}

The following result is well known in the classical cases, such as associative or Lie algebras. 

\begin{Lemma}\label{27.9} Let $F=F(\V,X)$ be a relatively free algebra of a mutihomogeneous $c$-step nilpotent primitive class $\V$, with a free generating set $X$. Suppose $|X|> c$. Then the following are true.
\begin{enumerate}
\item[(a)] Consider  $Y\subset X$ with $|Y|=c+1$ and let $u\in F$ be such that $ux=0$, for any $x\in Y$. Then also $uF=\{ 0\}$. Similarly, $xu=0$ for all $x\in Y$ implies $Fu=\{0\}$.
\item[(b)] Each term $I_k(F)$ of the annihilator series (\ref{eAS}) of $F$ is a multihomogeneous verbal ideal of $F$.
\end{enumerate}  
\end{Lemma}

\begin{proof}(a) Since each $ux=0$, $x\in Y$, is a relation among the free generators of $F$, by Lemma \ref{lRelX}, $ux=0$ is an identical relation in $\V$. Let us write $u=u_1+\cdots+u_m$ where $u_i$ is a multihomogeneous component of $u$ as a polynomial in $X$. Now for each $x\in Y$, $u_ix$ is a multihomogeneous component of $ux$, hence by our condition on $\V$, each $u_ix=0$ is a relation among the generators of $F$, so  an identity in $\V$.  Since $F^{c+1}=\{ 0\}$,  being multihomogeneous, $u_i$ depends only on the set $Z$ of at most $c$ variables. Thus there is $x\in Y$ such that $u_i$ does not depend on $x$. Now for any $v\in F$, let us consider a homomorphism $\vp:F\to F$ that maps the variables of $Z$ to themselves and $x$ to $v$. Then we will have $uv=0$, hence $uF=0$, as needed.

(b) If we take $u\in I$ in the argument in the proof of (a), we can see that $I$ is a multihomogeneus ideal. Now let us naturally embed $F=F(x_1,\ld,x_n)$, with $n>c$, in $F_\infty=F(x_1,\ld,x_n,\ld)$ and consider $I=I(F_n)$ and $J=I(F_\infty)$. If $J$ is a verbal ideal of $F_\infty$ then there is a primitive subclass $\U\subset\V$ such that $J=U(F_\infty)$. Since $F_n$ is a retract in $F_\infty$, we will get that $I=U(F_n)$, which follows because by Part (a), $I\subset J$.

So it remains only to show that $J$ is a verbal ideal of $F_\infty$. Using Proposition \ref{pGC}, it is sufficient to show that for any endomorphism $\vp$ of $F_\infty$ we have $\vp(J)\subset J$. Take $v\in J$ and assume $v\in F(x_1,\ld,x_m)$. Set $y_i=\vp(x_i)$ for $i=1,\ld,m$. Suppose $y_1,\ld,y_m\in F(x_1,\ld,x_k)$ where $k\ge m$. Consider $vx_{k+1}=0$ and apply to this equation an endomorphism $\psi$ such that $\psi(x_i)=y_i$ if $i=1,\ld,m$ and $\psi(x_{k+1})=x_{k+1}$. Then $0=\psi(vx_{k+1})=\vp(v)x_{k+1}$. So if we apply to this equation an endomorphism $\tau$ such that $\tau(x_i)=x_i$ for $i=1,\ld,k$ and $\tau(x_{k+1})=u$, where $u$ is any element of $F_\infty$, we will obtain $\vp(v)u=0$ or $\vp(v)F_\infty=\{ 0\}$. Similarly, we can obtain $F_\infty\vp(v)=\{ 0\}$. Thus $\vp(J)\subset J$, for any endomorphism $\vp$ of $F_\infty$, proving that $J=I(F_\infty)$ is a verbal ideal. It follows that $I=I(F_n)$ is indeed a verbal ideal.

It is true that $I=I_1(F)$ is a multihomogeneous verbal ideal in $F$. Indeed, assume $u=u_1+...+u_k\in I$, where each $u_i$ is homogeneous of degree $i$. Then $ux=u_1x+...+u_kx=0$ is an identity of  $\V$ with homogeneous components $ux_i=0$. Since $\V$ is multihomogeneous, each $ux_i=0$ is an identity of $\V$. Similarly, $xu_i=0$ is an identity in $\V$. It follows, that for any $v\in F$ we have $uv=vu=0$, that is, $u\in I$.

The quotient-algebra $F'=F/I$ is then a free algebra in a multihomogeneous $(c-1)$-step nilpotent primitive class  $V'$. It remains to apply induction by $c$.    
\end{proof}

\section{Some linear algebra}\label{sLA}

In this section we gather some facts from Linear Algebra, which will be needed in the proofs of the main results in Section \ref{sMR}. We quickly recall one of the versions of the Structure Theorem of modules over the principal ideal domain (see e.g. \cite[Theorem 4 in Chapter XV]{SL}) in the case of the polynomial ring in one variable. Recall that a  vector space $V$ over a field $\F$ with the action of a linear operator $B:V\to V$ becomes a module over the polynomial ring $R=\F[t]$ via $tv=B(v)$, for any $v\in V$.  If $V$ is finite-dimensional then $V$ is a finitely generated $R$-module.

\begin{Proposition}\label{tFGMPID} Assume that $\dim_\F V<\infty$. Then, as an $R$-module, $V$ is the direct sum of cyclic submodules. Moreover, there exists a sequence of monic polynomials $f_1(t),f_2(t),\ldots,f_r(t)$ satisfying $f_1(t)\vert f_2(t)\vert\cdots\vert f_r(t)$ such that  
\begin{equation}\label{e1}
V= V_1\oplus V_2\oplus\cdots\oplus V_r, \mbox{ where }V_i\cong R/(f_i(t)).
\end{equation}
$\hfill\Box$
\end{Proposition}

In what follows, we will call $V$ as above, a $B$-module. If $\dim_\F V=m$ and $B=\lambda\,\id_V$ is a scalar operator then any minimal generating set of $V$ as a $B$-module is a basis of $V$ so that the  minimal number of generators needed to generate $V$ as a $B$-module equals $m$. If $V$ is a $B$-module as in (\ref{e1}), then the minimal number of generators equals $r$. Number $r$ is an invariant of a $B$-module $V$; it is called the \textit{rank} of $V$ and we write $\rank_B(V)=r$. We always have $\rank_B(V)\le \dim_\F V$.

\begin{Lemma}\label{28.1} Let $B$ be a  linear operator on a vector space $V$ over a field $\F$ such that $\dim_\F V=m$ and  $\rank_B(V)=r$. Then for any $\ell\in [\frac{2m+r}{3}, m]$ there is a pair of subspaces $U\subset W$ in $V$ such that $\dim_\F W=\ell$, $\dim_\F U=m-\ell$, $\dim_\F B(U)=\dim_\F U$,  and $V = W \oplus B(U)$.  
\end{Lemma}
\begin{proof} The case $m=1$  being obvious, we assume $m\ge 2$.

Fist assume that $V$ is cyclic as a $B$-module and $e$ is a generator of $V$ as a $B$-module. We want to show that in this case we can relax the interval for $\ell$ to $[\frac{m+1}{2}, m]$. 

Set $e_i=B^i(e)$, $i=0,\ld,m-1$. Then $\{e_0,e_1,\ld,e_{m-1}\}$ is a basis of $V$. 

If $m$ is even, we set 
\[
W^0=U^0=\Sp \{ e_0,e_2,\ld,e_{m-2}\},
\]
\[
B(U^0)=\Sp \{ e_1,e_3,\ld,e_{m-1}\}.
\]
 We have $V=W^0\oplus B(U^0)$, also $\frac{m}{2}=\ell^0=m-\ell^0=\dim_\F U^0=\dim_\F W^0$. 

If $m$ is odd, we set 
\[
W^0=\Sp \{ e_0,e_2,\ld,e_{m-1}\},
\]
\[
U^0=\Sp \{ e_0,e_2,\ld,e_{m-3}\}.
\]
  Then 
  \[
  B(U^0)=\Sp \{ e_1,e_3,\ld,e_{m-2}\}.
  \]
   Again we have $V=W^0\oplus B(U^0)$, also $\frac{m+1}{2}=\ell=\dim_\F W^0$, $\frac{m-1}{2}=m-\ell=\dim_\F U^0$. 

In both cases, if we want to have $U\subset W$ with a greater value for $\dim_\F W$, we simply remove, say, $e_0,\ld,e_{2k}$, from the basis of $U$ and add $e_1,\ld,e_{2k+1}$ to the basis of $W$. So in the case where $V$ is a cyclc $B$-module, we can find the desired pair $U\subset W$ with any $\ell\in[\frac{m+1}{2},m]$. 

Now if $V=V_1\oplus\cdots\oplus V_r$, $r>1$, each $V_i$ cyclic of dimension $m_i\ge 2$, we consider $W_i^0$ and $U_i^0$ in each $V_i$, built in the same way as in the case of a cyclic module, so that $\dim_\F W_i^0=\frac{m_i}{2}$ if $m_i$ is even and $\dim_\F W_i^0=\frac{m_i+1}{2}$ if $m_i$ is odd. Set $W^0=W_1^0\oplus\cdots\oplus W_r^0$, $U^0=U_1^0\oplus\cdots\oplus U_r^0$. We then have $V=W^0\oplus B(U^0)$. Set $\ell_0=\dim_\F W^0$. We have $\ell_0\le \frac{m+r}{2}$. Now let $\ell$ be any number, $\ell_0\le \ell\le m$. If $\ell\ne m$ then one of $\dim_\F W_i\le m_i$. Using the same method as in the case of a cyclic module, we can modify $W_i^0$ and $U_i^0$ so that the value of $\ell$ grows by 1.

In the general case, $V=V'\oplus V''$ where $V'$ is as just considered, with $m'=\dim_\F V'$, $r'=\rank_B(V')$ and $V''$ is the sum of one-dimensional $B$-invariant subspaces. We have $\dim_\F V''=m-m'$. The rank $r$ of $V$ as a $B$-module equals $r'+\dim_\F V''=r'+m-m'$. Since $\frac{m'+r'}{2}\le \frac{2m'+r'}{3}$, using our previous argument, for any $\ell'\in [\frac{2m'+r'}{3},m']$ we find $U'\subset W'\subset V'$ such that $\dim_\F W'=\ell'$, $\dim_\F U'=m'-\ell'$ and $V'=W'\oplus B(U')$. We set $U=U'$ and $W=W'\oplus V''$.

 Now we can conclude that since for $\ell'$ we can choose an arbitrary value in the interval $\left[\frac{2m'+r'}{3},m'\right]$, we can get for $\ell$ an arbitrary integral value in $\left[\frac{2m'+r'}{3}+(m-m'),m'+(m-m')\right]=\left[\frac{2m+r}{3},m\right]$, as claimed.

\end{proof}

\begin{Lemma}\label{28.24} Let $V$ be a vector space of dimension $m$ over a field $\F$ and $B:V\to V$ a linear operator. If $\rank_B(V)>\frac{m}{2}$ then there is $\lambda\in\F$ and a subspace $W$ of dimension $r=\rank_B(V)$ such that $(B-\lambda\, I)(W)=\{ 0\}$. 
\end{Lemma}
\begin{proof} Let us use (\ref{e1}): $V=V_1\oplus\cdots\oplus V_r$ where $V_i=\F[t]/(f_i(t))$ and $f_1(t)\vert\ld\vert f_r(t)$. If all of $V_i$ had dimension greater than 1 then we would have $\rank_B(V)\le \frac{m}{2}$. So one of $V_i$ is 1-dimensional, hence spanned by an eigenvector with eigenvalue $\lambda\in\F$. So $f_1(t)=t-\lambda$ and also each $f_i(t)$ is divisible by $t-\lambda$. As a result, in each $V_i$ there is an eigenvector $e_i$ with eigenvalue $\lambda$. So we can choose $W=\Sp\{e_1,\ld,e_r\}$ for the $r$-dimensional subspace annihilated by $B-\lambda I$.

\end{proof}

\begin{Lemma}\label{27.15}Let $V$ be a vector space over a field $\F$, written as $V=V_1\op\ld\op V_t$ for some subspaces $V_1,\ld,V_t$, and $U_1,\ld, U_n$ some subspaces in $V$. Assume that there exists an integer $c$ such that for any $j=1,\ld,t$ at most $c$ of the subspaces $U_1,\ld, U_n$ have nonzero projection to $V_j$. Then
\begin{equation}\label{eLB}
\dim_\F U_1+\cdots+\dim_\F U_n\le c\dim_\F(U_1+\cdots+ U_n).
\end{equation}
\end{Lemma}
\begin{proof} We start with few remarks.  Set $U=U_1+\cdots+U_n$. Since for any $j=1,\ld,t$ the projections of each $U_i$ to $V_j$ are in the projections of $U$ to $V_j$, the conditions of the Lemma will be satisfied if we replace each $V_j$ by the projection $\pr_{V_j}U$. Once $V_j=\pr_{V_j}U$ is assumed, setting $\widetilde{V_j}=V_1\op\cdots\op V_{j-1}\op V_{j+1}\op\cdots\op V_t$, for any $j=1,\ld,t$,  by the Homomorphism Theorem applied to the projection map of $U=U_1+\cdots+U_n$ onto $V_j$, we will have 
\begin{equation}\label{eLB2}
\dim_\F (U_1+\cdots+U_n)-\dim_\F((U_1+\cdots+U_n)\cap \widetilde{V_j})=\dim_\F V_j.
\end{equation}
Also note that since $\dim_\F U_i\le \dim_\F U$,
in the case $n\le c$, our lemma easily follows because
\[
\dim_\F U_1+\cdots+\dim_\F U_n \le n \dim_\F U\le c \dim_\F (U_1+\cdots+U_n).
\]

To proceed with the proof, we apply induction by $t$. If $t=1$ then $V_1=V$ and $U_i$ is its own projection to $V_1$, for all $k=1,\ld, n$. So only $d\le c$ of $U_i$, say, $U_1,\ld,U_d$ are nonzero. Removing zero subspaces from (\ref{eLB}) we may assume $n\le c$ and in this case our claim has been shown to be true.

Now we need to consider the case $t\ge 2$ and $n>c$. Let us choose $n-c$ subspaces, say, $U_1,...,U_{n-c}$ whose projections to $V_t$ are zero. 

 By induction, for $V'=V_1\op\cdots\op V_{t-1}$ and the subspaces $U_i'=U_i\cap V'$, we have
  \begin{equation}\label{eLB3}
\dim_\F U_1'+\cdots+\dim_\F U_n'\le c\dim_\F (U_1'+\cdots +U_n').  
  \end{equation}
  Note that $\dim_\F U_1'=\dim_\F U_1,\ld,\dim_\F U_{n-c}'=\dim_\F U_{n-c}$. Also, since 
  \[
  U_i/U_i'\cong U_i+V'/V'\subset V/V'\cong V_t,
  \]
   we have that 
  \begin{equation}\label{eAT}
 \dim_\F U_i\le \dim_\F U_i'+\dim_\F V_t\mbox{ for all }i=n-c+1,\ld,n. 
  \end{equation} 
  
Using (\ref{eLB3}), which is true by induction, and (\ref{eAT}), which is always true, we obtain
 \begin{eqnarray*}
 &&\dim_\F U_1+...+\dim_\F U_n\le \dim_\F U_1'+\cdots+\dim_\F U_n'+c\dim_\F V_t\\
 &&\le c(\dim_\F (U_1'+\cdots+U_n')+c\dim_\F V_t\\&&=c\dim_\F((U_1\cap V')+...+(U_n\cap V'))+c\dim_\F V_t\\
&&\le c\dim_\F((U_1+...+U_n)\cap V')+c\dim_\F V_t\\&&=c\dim_\F(U_1+...+U_n), 
\end{eqnarray*}
as claimed. Note that in the last equality we have used that $V_t=\mathrm{pr}_{V_t}U$.
\end{proof}

\section{Ideals in relatively free nilpotent algebras}\label{sGRFNA}
 
\begin{Lemma}\label{27.4} For any $c\ge 1$ there exist positive $\kappa_1$ and $\kappa_2$ with the following property. Consider a primitive class $\V\subset \fN_c$ such that $\V\not\subset \fN_{c-1}$. Let $F_n=F(\V,\{ x_1,\ld,x_n\})$ be a $\V$-free algebra of arbitrary rank $n\ge 1$.  
Then, as soon as $F_n^c\ne \{ 0\}$, the following hold
\begin{equation}\label{e01} 
\kappa_1 n^c < \dim_\F F_n^c\le \dim_\F I(F_n) < \kappa_2 n^c.
\end{equation} 
\end{Lemma}

\begin{proof} 

For the estimate from below, we first assume that $n>c$. It follows from Lemma \ref{lRelX} and  $\V$ being $c$-step nilpotent that  for each subset $S$ of $c$ different variables of $\{x_1,\ld,x_n\}$ there is a monomial $w_S$ of degree $c$ in the variables of this subset which is nonzero in $F_n$. Moreover, the set of such monomials, one for each choice of subsets, will be linearly independent in $F_n$. Indeed, if $\sum\lambda_Sw_S=0$ and  $\lambda_S\ne 0$, for some $S$ then the endomorphism of $F_n$ sending $x$ to $x$ is $x\in S$ and $x$ to $0$ otherwise, sends the above linear combination to $\lambda_{S}w_{S}=0$, which is a contradiction. 
Thus $\binom{n}{c}\le \dim_\F F_n^c$ and one can find a positive constant $\kappa_1'=\kappa_1'(c)$ such that $\kappa_1'(c)< \frac{\binom{n}{c}}{n^c}$. Thus $\kappa_1'(c)n^c<\dim_\F F_n^c$, for all $n\ge c$. Then we  choose a positive $\kappa_1$, such that $\kappa_1\le \kappa_1'$ and $\kappa_1n^c\le \dim_\F F_n^c$, for all $F_n^c$, $n=1,\ld,c-1$, which are different from zero. Then we will have our lower bound $\kappa_1 n^c<\dim_\F F_n^c$, for all $n\ge 1$ such that $F_n^c\ne\{ 0\}$. Finally, the same estimate works for $I(F_n)$ because $F_n^c\subset I(F_n)$.

To get the estimate from above, we notice that in any monomial $x_{i_1}\cdots x_{i_c}$ the number of possible placements of brackets is restricted by some $\kappa_2'=\kappa_2'(c)$. Since $F_n$ is spanned by (nonassociative) monomials of degree $c$, we have $\dim_\F F_n^c\le \kappa_2'(c)n^c$. It follows that $\dim_\F F_n$ is bounded from above by a polynomial in $n$ of degree $c$ and $\dim_\F F_n/F_n^c$ by a polynomial in $n$  of degree $c-1$. Since $I(F_n)\supset F_n^c$, $\dim_\F I(F_n)$ is bounded from above by $\dim_\F F_n/F_n^c+\dim_\F F_n^c$, which is a polynomial in $n$  of degree $c$. So there is $\kappa_2(c)\ge \kappa_2'(c)$ such that $\dim_\F I(F_n)< \kappa_2(c)n^c$ and (\ref{e01}) indeed holds.     
\end{proof}

A close result, which we will need in Section \ref{ssWGI} is the following.

\begin{Lemma}\label{L*}
Let $f\in \cF_{\infty}^c$ be a polynomial without monomials depending on one variable only. Assume that $f=0$ is not an identical relation in a free algebra $F_n=F(\V,X)$, $X=\{ x_1,\ld,x_n\}$, of a $c$-step nilpotent variety $\V$, $c\ge 2$.  Then the verbal ideal $f(F_n)$ generated by $f$ in $F_n$ has dimension  bounded from below by a quadratic function of $n$.

\end{Lemma}
\begin{proof} If $Y$ is a subset of $X$, we denote by $f_Y$ the sum of all monomials in $f$ each of which depend on all variables of $Y$ and only on them. The primitive subclass of $\V$ defined by $f$ is the same as the primitive subclass of $\V$ defined by the collection of all $f_Y$, $Y\subset X$, such that $|X|=t\le c$. It is enough to prove our claim in the case where $f=f_Y$, for some $Y$, as before, that is, $|Y|=t\le c$. Given a subset $Z\subset X$ with $|Z|=t$, we consider a bijection $\sigma_Z: Y\to Z$. The polynomials $\sigma_Z(f)$ are all in the verbal ideal $f$, they are linearly independent and hence there is $\kappa$ such that $\dim_\F f\ge\kappa \binom{n}{t}$  (cf. with the proof of Lemma \ref{27.4}. Since $t\ge 2$, we have a quadratic function bounding $\dim_\F f(F_n)$ from below.
\end{proof}

We further discuss some properties of a relatively free $c$-step nilpotent algebra $F_n$, which will be instrumental later in the proof of the fact that ``almost all'' ideals of $F_n$, with $n>n_0=n_0(c)$, are contained in the annihilator $I(F_n)$.

Given subspaces $P, Q$ of an algebra $R$ over a field $\F$, we denote by $PQ$ the linear span of all products $ab$, where $a\in P$ and $b\in Q$.
 
\begin{Lemma}\label{27.18} Let $\V$ be a $c$-step nilpotent multihomogneous primitive class of algebras over a field $\F$, $c\ge 2$. Then there exists $\kappa=\kappa(c)>0$ with the following property. Let $F_n=F(\V,\{ x_1,\ld,x_n\})$ be a $\V$-free algebra of rank $n\ge 2c$. If $U$ is a subspace in $I_2(F_n)$ of dimension $s>0$ such that $U\cap I_1(F_n)=\{ 0\}$ then either $\dim_\F UF_n > \kappa s n$ or $\dim_\F F_nU >  \kappa s n$.
\end{Lemma}
 
 \begin{proof}
 Assume that there are subsets $M,L\subset \{ x_1,\ld,x_n\}$ with $|M|=|L|=c+1$ such that $\dim_\F Ux < \frac{s}{2c+2}$ for all $x\in M$ and $\dim_\F xU < \frac{s}{2c+2}$ for all $x\in L$. Then the intersection of the kernels of $2(c+1)$ linear maps $U\to Ux$ ($x\in M$) and $U\to xU$ ($x\in L$) has codimension $<2(c+1)\frac{s}{2c+2}=s$. In this case, there exists a nonzero $u\in U$ which belongs to the intersection of the kernels of all these maps. By Lemma \ref{27.9}, part (a), we must have $u\in I_1(F_n)$. Since we assumed $U\cap I_1(F_n)=\{ 0\}$, this is not possible. Thus our assumption is not valid.
 
 It follows that there is an $(n-c)$-element subset $Y$, say $Y=\{x_1,...,x_{n-c}\}$ $\subset \{ x_1,\ld,x_n\}$, such that either for each $x\in Y$   
 \begin{equation}\label{eGrillr}
\dim_\F Ux\ge \frac{s}{2c+2} 
 \end{equation}
 or for each $x\in Y$
 \begin{equation}\label{eGrilll}
\dim_\F xU\ge \frac{s}{2c+2}. 
\end{equation}

Let us consider the case of (\ref{eGrillr}) and set $W=\sum_{x\in Y} Ux \le UF\le I_1(F_n)                                                                                                                                                                                                                                                                                                                                                                                                                                                                                       $.

By statement (b) in Lemma \ref{27.9}, the annihilator $I_1(F_n)$ is a verbal ideal. Since the base field is infinite, by Lemma \ref{L1.3}, $I_1(F_n)$ is the direct sum of its multihomogeneous components: $I_1(F_n)=V_1\op\cdots\op V_t$. Since $F_n^{c+1}=\{ 0\}$,  each component $V_i$ is a linear span of the elements of fixed degree in each free generator belonging to a fixed subset $Z_i$ with $|Z_i| \le c$. It follows that the projection of $Ux$ to such $V_i$ can be nonzero only if $x\in Z_i$. Hence at most $c$ subspaces among $Ux$, $x\in Y$, can have a nonzero projection on $V_i$. By Lemma \ref{27.15}, equation (\ref{eGrillr}), and using $n-c\ge c\ge 2$, hence $n-c\ge \frac{n}{2}$,  we obtain
\[
\dim_\F UF_n \ge \dim_\F \sum_{x\in Y} Ux \ge c^{-1} \sum_{x\in Y} \dim_\F Ux \ge \frac{(n-c) s}{2c(c+1)}\ge \frac{sn}{4c(c+1)}.
\]

So any positive $\kappa <\frac{1}{4c(c+1)}$ can be taken to satisfy the claim of our lemma.

\end{proof}

To state the next lemma, we introduce the following notation. Let $\varPi$ be the  set of words in the alphabet $\{\lambda,\rho\}$ and $U$ a subspace in an algebra $P$. If $\varpi\in \varPi$ is an empty word, we write $\varpi U=U$. If $\varpi=\varpi'\lambda$, we use induction by the length of $\varpi$ and write $\varpi U=F_n(\varpi' U)$. If $\varpi=\varpi'\rho$  we write $\varpi U=(\varpi' U)F_n$. For instance, if $\varpi=\lambda\rho\rho\lambda$ then $\varpi U=F_n(((F_nU)F_n)F_n)$. Our next lemma is the following.

\begin{Lemma}\label{27.24} Assume $c\ge 2$ and $n\ge 2c$. Then one can find a positive  $\gamma=\gamma(c)>0$ such that the following is true. Let $i\ge 1$ be such that $I_i$ is a proper subspace in $I_{i+1}$,  $U$ a subspace in $I_{i+1}$, such that $\dim_\F (U+I_i)/I_i =s_i$. Then there is $\varpi\in \varPi$ of length $i$ such that $\dim_\F \varpi U > \gamma s_i n^i$.  
\end{Lemma}

\begin{proof}
Since $I_1$ annihilates $U$, the conditions $\dim_\F U=s_i$ and $U\cap I_1 =\{ 0\}$ in Lemma \ref{27.18} can be relaxed to $\dim_\F (U+I_1)/I_1 =s_i$. If we apply Lemma \ref{27.18} in its relaxed form to the relatively free algebra $F_n/I_{i-1}$, we will obtain  $\dim_\F \tau U+I_{i-1}/I_{i-1}\ge\kappa s_i n$, for some $\kappa$ which is common for all $i$ in question, where $\tau=\rho$ or $\tau=\lambda$. Applying induction by $i$ with basis formed by Lemma \ref{27.18}, we will conclude that $\dim_\F \varpi'\tau U\ge \kappa^is_in^i$ for some bracket placement $\varpi'$ of length $i-1$ so that we can take $\varpi=\varpi'\tau$ and $\gamma=\kappa^i$. 
\end{proof}

In the next lemma we consider two types of linear operators $\vp$ on a relatively free $c$-step nilpotent algebra $F_n$ of rank $n$: the automorphisms $\vp$ and the derivations $D$. Let $A$ be the linear operator induced by one of $\vp$ or $D$ on $F_n/F_n^2$. By Proposition \ref{lGLFc}, if $A$ comes from $\vp$ then $A\in\GL(n)$. Otherwise, $A\in\gl(n)$. The restriction of both $\vp$ and $D$ to $F^c_n$ depends only on $A$ and denoted by $\wA$.

\begin{Lemma}\label{34.3} For any $c\ge 2$ there exists $\delta=\delta(c)>0$ such that the following holds. Let $F_n$ be a relatively free $c$-step nilpotent algebra of  rank $n\ge c$. For each of the operators $\vp\in\Aut{F_n}$ and $D\in\Der{F_n}$, denote by $A$ its action on $F_n/F_n^2$ and by $\wA$ its action on $F^c_n$. If $A$ is not scalar, then $\dim_\F F^c_n-\rank_{\wA}(F^c_n)\ge\delta n^{c-1}$.
\end{Lemma}

\begin{proof}

It is enough to find a subspace $U$ in $F_n^c$ such that $U\cap \wA(U)=\{ 0\}$ and $\dim_\F \wA(U)>\delta n^{c-1}$ for a positive $\delta$ depending only on $c$. In order to do this, we consider the decomposition  $F_n/F^2_n= V_1\op \cdots\op V_r$ as the direct sum of cyclic $\wA$-submodules $V_1,\ld,V_r$, as in Proposition \ref{tFGMPID},  such that at least one of them, say  $V_1$, has dimension $t\ge 2$. If $e$ is a generator of $V_1$ and $A$ is not scalar, then $\{ e_1=e, e_2=A(e),\ld, e_t=A^{t-1}(e)\}$ is a basis of $V_1$. Let us complement  this basis of $V_1$ to the basis $\{ e_1,\ld,e_n\}$ of $ F_n/F_n^2$ by the vectors from the subspaces $V_2,\ld,V_r$.

We denote by $y_1,\ld, y_n$ the preimages of $e_1,\ld,e_n$ in $F_n$ under the natural homomorphism $\ve: F_n\to F_n/F_n^2$. We know (see Corollary \ref{cGFS}) that these elements freely generate $F_n$.  Since $F_n$ is $c$-step nilpotent, one can find a nonzero monomial $w$, with some placement of brackets, of degree $c$ in $y_1,\ld,y_c$. If we replace in $w$ the variables $y_1,\ld,y_c$ by any other set of pairwise different variables $y_{\mu(1)},\ld,y_{\mu(c)}$ then we will obtain a linearly independent (thanks to Lemma \ref{L1.3}) set of monomials $\{ w_\mu\}$ in $F_n^c$, for pairwise different subsets of the form $\{\mu(1),\ld,\mu(c)\}$.  Let us consider the span $U$ of those $w_\mu$  which include $y_1$ but not $y_2$ or $y_t$. We have $\dim_\F U\ge \binom{n-3}{c-1}$. Clearly, for some $\delta=\delta(c)>0$ and all $n\ge c+3$ we have   $\binom{n-3}{c-1}>\delta n^{c-1}$.

Let us apply $\wA$ to $w_\mu$ in $U$. We first consider the case where $A$ and $\wA$ are induced by an automorphism. The following is well-defined: 
\[
\wA(w_\mu)=\wA(w(y_{\mu(1)},\ld,y_{\mu(c)}))= w(\ve^{-1}(A(e_{\mu(1)})),\ld,\ve^{-1}(A(e_{\mu(c)})).
\]
We have $A(e_i)=e_{i+1}$, $i=1,2,\ld,t-1$, and $A(e_j)\in U_2\op\cdots\op U_r$, if $j=t+1,\ld,n$. It then follows that $\ve^{-1}(A(e_{i}))=y_{i+1}+F_n^2$, $i=1,2,\ld,t-1$, and $\ve^{-1}(A(e_j))=v_j+F_n^2$ where $v_j$ does not depend on the variables $y_1,\ld,y_t$, if $j=t+1,\ld,n$.
As a result, the elements $\wA(w_\mu)$ are well-defined  linear combinations of the monomials with the same placement of brackets as in $w$, not depending on $y_1$. We have $\dim_\F \wA(U)=\dim_\F U\ge\delta n^{c-1}$. At the same time, in $U$, all spanning monomials depend on $y_1$. Moreover, $U\cap \vp(U)=\{ 0\}$ because the endomorphism of $F_n$ mapping $y_1$ into 0 and leaving other variables fixed annihilates $U$ but leaves $\vp(U)$ invariant. So $U\cap\wA(U)=\{ 0\}$, as needed.

Now assume that $A$ and $\wA$  are induced by a derivation $D$. Note that if $w_\nu=w(y_{\nu(1)},\ld, 
y_{\nu(c)})$ then 
\begin{eqnarray}\label{e34.6}
\wA(w_\nu)&=&\sum_{j=1}^c w(y_{\nu(1)},\ld,D(y_{\nu(j)}),\ld, y_{\nu(c)})\\&=&\sum_{j=1}^c w(\ve^{-1}(e_{\nu(1)}),\ld,\ve^{-1}(A(e_{\nu(j)})),\ld, \ve^{-1}(e_{\nu(c)})).\nonumber
\end{eqnarray}

There is only one monomial, let us denote this by $w_\nu'$, on the rightmost side of (\ref{e34.6}), depending on $y_2$. Since $y_1,\ld,y_c$ are free generators, $w_\nu'=w_\mu'$ if and only if $\mu=\nu$. The variables in $w'_\nu$ are the same as in $w_\nu$, except that $y_1$ is replaced by $y_2$.  It follows that $\vp(U)$ has the same dimension as $U$. As in the case of the automorphisms, $U\cap D(U)=\{ 0\}$ or $U\cap\wA(U)=\{ 0\}$. The proof is compete.
\end{proof}

We now proceed with an auxiliary section, placed for the reader's convenience. In the whole of Section \ref{ssSFAG} the base field is algebraically closed of arbitrary characteristic.

\section{Some Algebraic Geometry}\label{ssSFAG}

For the reader's convenience, we list some basic definitions and results from the first chapter of the book \cite{IRS}. 

\subsection{Basic notions}\label{ssBN} Let $V$ be  a finite-dimensional vector  space  over an algebraically closed field $\F$. Given $0\ne v\in V$, one denotes by $\PS(v)$ the one-dimensional subspace of $V$ containing $v$. The \textit{projective space} $\PS(V)$ built on $V$ is the set of all $\PS(v)$. One has $\PS(u)=\PS(v)$ iff there is $\lambda\in\F$ such that $v=\lambda u$. Thus, given a basis $\mathcal{B}=\{e_0,\ld,e_{n}\}$ in $V$, the point $\PS(v)$ is defined by  its \textit{projective coordinates}, which is the equivalence class of non-zero tuples $(\alpha_0:\ld:\alpha_{n})$ such that  $\alpha_0e_0+\cdots+\alpha_{n}e_{n}\in\PS(v)$. The tuple $(\beta_0:\ld:\beta_{n})$ defines the same point iff  there exists $\lambda\ne 0$ such that 
\begin{equation}\label{eP}\alpha_0=\lambda\beta_0,\ld,\alpha_{n}=\lambda\beta_{n}.
\end{equation}

If $\alpha:U\to V$ is an isomorphism of vector spaces then $\alpha$ establishes a bijection between one-dimensional subspaces of $U$ and $V$, which is denoted by $\PS(\alpha)$. If $U=V$ then the set of all $\PS(\alpha)$ with $\alpha\in\GL(V)$ is a group $\PGL(V)$, which is isomorphic to the factor-group of $\GL(V)$ by its center. 

The (standard) projective space $\PS^n$ is the set of equivalence classes of nonzero tuples $(\alpha_1:\ld:\alpha_{n+1})$ under the equivalence relation given by (\ref{eP}).

The projective space $\PS^n$ is a topological space under the \textit{Zariski topology}. A closed subset $\X\subset \PS^n$ in this topology is defined as the set of solutions of a system of homogeneous polynomial equations with respect to the projective coordinates of the points in $\textbf{P}^n$. One calls  $\X$ a \textit{projective variety}. The variety $\X$ is called \textit{irreducible} if $\X$ is not the union of its proper projective varieties. An open subset $\Y$ of a projective variety is called a \textit{quasi-projective variety}. One calls $\Y$ \textit{irreducible} if its closure is irreducible.

A finite union of quasi-projective varieties in a projective space $\textbf{P}^n$ is called a \textit{constructible} set, or simply a \textit{variety}; it does not need to be a quasi-projective variety.

A map $f$  of an irreducible quasi-projective variety $\X$ to a variety $\Y\subset \PS^n$  is a function given by homogeneous monomials of the same degree on the coordinates of the points of an open subset of $\X$. The image $f(\X)$ of a quasi-projective variety $\X$ under a  map $f$ does not need to be quasi-projective. By \cite[Exercises 3.18]{HRT}, $f(\X)$ is a constructible set.

According to \cite[1.6.2]{IRS}, Corollary 3, the dimension $\dim\X$ of a projective variety $\X$ can be defined as the 
maximal integer $n$ for which there exists a strictly decreasing chain 
\[
\Y_0\supset \Y_1\supset\cdots\supset 
\Y_n \supset \varnothing
\]
of length $n$ of irreducible subvarieties $\Y_i\subset\X$. The dimension of any subset can be defined as the dimension of its closure in the Zariski topology.  If a constructible variety $\X$ is a finite union of quasi-projective varieties $\X=\X_1\cup\ldots\cup\X_n$ then $\dim\X=\max\{\dim\X_1,\ld,\dim\X_n\}$. 

{\sc Note.} We reserve the notation $\dim\X$ for the dimension of a \textit{variety} $\X$; we keep using  $\dim_\F V$ for the dimension of a \textit{vector space} $V$ over a field $\F$.
 
\begin{Remark}\label{rCartProd}The \textit{Cartesian product} $\textbf{P}^n\times \textbf{P}^m$ becomes a quasi-projective variety via Segr\' e embedding $\vp:\textbf{P}^n\times \textbf{P}^m\to \textbf{P}^N$ where $N=(n+1)(m+1)-1$ (see \cite[1.5.1]{IRS}). If $x=(u_0:\ld:u_n)\in\textbf{P}^n$ and $y=(v_0:\ld:v_m)\in\textbf{P}^m$ then explicit formula is $\vp(x,y)= (w_{ij})\in \textbf{P}^N$, where $w_{ij}=u_iv_j$, for all $0\le i\le n$, $0\le j\le m$. Using Segr\' e embedding, one gives the structure of a quasi-projective subvariety in $\textbf{P}^N$ to $\X\times\Y$ for any quasi-projective varieties $\X\subset \textbf{P}^n$ and $\X\subset \textbf{P}^m$.
\end{Remark}  

If $\dim_\F V=n$ then $\dim\GL(V)=n^2$ while $\dim\PGL(V)=n^2-1$. One also writes $\GL(V)=\GL(n)$ and $\PGL(V)=\PGL(n)$.

Some properties of the dimension are listed in the following. They all can be found in \cite[Chapter 1, \S\S 4-6]{IRS}.

\begin{Lemma}\label{LSH}  Let $\X$ and $\Y$ be two varieties, $\X$ irreducible. Then the following are true.
\begin{enumerate}
\item[\rm(a)]  If $\Y$ is an open subset in $\X$ then $\dim\Y=\dim\X$. If $\Y$ is a proper closed subset of  $\X$ then $\dim\Y<\dim\X$. 

\item[\rm(b)] Let $f:\X\to\Y$ be a  map of irreducible varieties. Then the fiber $f^{-1}(y)$ over each $y=f(x)$, $x\in\X$, is always a variety. Suppose the dimensions of all these fibers are bounded from below by some $r$. Then $\dim f(\X)\le \dim\X-r$. 
\item[$\mathrm{(b^\prime)}$] If $f:\X\to \Y$ is a surjective  map of  irreducible varieties, $\dim\X=n, \dim \Y=m$. Then $n\ge m$ and there exists a nonempty open subset $U\subset \Y$ such that $\dim f^{-1}(y) = n-m$ for $y\in U$.
 
\item[\rm(c)] If $\X$ and $\Y$ are irreducible varieties then $\X\times \Y$ is irreducible, also $\dim \X\times \Y =\dim \X+\dim \Y$;$\hfill\Box$ 
\end{enumerate}
\end{Lemma}

Now we can introduce the following central

\begin{Definition}\label{dGeneric} Let $\cS$ be a subset of a projective space $\textbf{P}^n$. We say that certain property is \textit{generic} in $S$ if the following holds. Denote by $T$ the subset of those points of $S$ which have this property. Then $T$ contains a subset which is a quasi-projective variety of some dimension $k$ while $S\setminus T$ is contained in a projective variety of dimension $<k$.
\end{Definition}

\subsection{Grassmann varieties}\label{sssGV}
Our main type of varieties  will be \textit{Grassmann varieties}, their products and images under the maps of varieties. For the proofs of the facts given below see \cite[Chapter 1, \S 6, Example 5]{IRS}.

Given a vector space $V$ with $\dim_\F V=n$ and a number $r$, $0\le r\le n$, one considers the projective space $\textbf{P}(\Lambda^r(V))$ where $\Lambda^r(V)$ is the homogeneous component of degree $r$ in the Grassman algebra $\Lambda(V)$ of the vector space $V$. The nonzero elements of the form $u_1\wedge\ld\wedge u_r\in \Lambda^r(V)$, where $u_1,\ld, u_r\in V$, are called \textit{decomposable}. The points $\PS(u_1\wedge\ld\wedge u_r)$ form a Zariski closed subset $\mathrm{Grass}(r,V)$ in $\textbf{P}(\Lambda^r(V))$, hence a projective variety. The variety $\Grass(r,V)$ is irreducible and
\begin{equation}\label{dimGr}
\dim\Grass(r,V)=r(n - r),\mbox{ where }n=\dim_\F V.
\end{equation} 

The points of the projective variety $\mathrm{Grass}(r,V)$ are in one-one correspondence with $r$-dimensional subspaces of $V$. The correspondence is given by $\Sp\{u_1,\ld,u_r\}\mapsto \PS(u_1\wedge\ld\wedge u_r)$.  Thus the set of $r$-dimensional subspaces $U$ in an $n$-dimensional space $V$ acquires the structure of a projective algebraic variety. One keeps the notation $\Grass(r,V)$ for this variety and calls this the \textit{Grassmann variety of $r$-dimensional subspaces in an $n$-dimensional space}. 

If $\mathcal{B}=\{e_1,\ld,e_{n+1}\}$ is a basis of $V$ then the elements of the form $e_{i_1}\wedge\cdots\wedge e_{i_r}$ with $1\le i_1<\ld<i_r\le n$ form a basis of $\Lambda^r(V)$. The projective coordinates of $\PS(u_1\wedge\cdots\wedge u_r)$ with respects to this basis are thus the projective coordinates of $U=\Sp\{u_1,\ld,u_r\}$ in the Grassmann variety $\Grass(r,V)$. They are called the \textit{Pl\"ucker coordinates} of $U$.
 
In the proofs about the properties of the maps of various constructions (unions, Cartesian products, etc.) of Grassmann varieties we will be using the following property mentioned in \cite[Example 1 in Section 4.1]{IRS}. 

\begin{Lemma}\label{lCartPluck} Let $\mathcal{B}=\{e_1,\ld,e_{n}\}$ be a basis of $V$. Let the Pl\" ucker coordinates in $\Grass(r,V)$ be defined via $\mathcal{B}$, as above. Then there exist homogeneous rational functions of degree zero with the following property. In every $r$-dimensional subspace $U$, there is a basis $\{ u_1,\ld,u_r\}$ such that the coordinates of the vectors of this basis with respect to $\mathcal{B}$ are the values of the above functions in the Pl\" ucker coordinates of $U$.
$\hfill\Box$
\end{Lemma}

\begin{Lemma}\label{lproduct} Let $W$ be a vector space, $\dim_\F W=n$. Consider the Grassmannians $\X=\Grass(s,W)$ and $\Y=\Grass(t,W)$, where $s+t\le n$. Then there is an open subset $\cO$ of $\fZ=\X\times \Y$ consisting of the pairs $(U,V)$ with $U\cap V=\{ 0\}$ on which the map $(U,V)\to U\oplus V$ is a map of varieties from $\cO$ to $\Grass(s+t,W)$.
\end{Lemma}
\begin{proof}  We may assume that $\X$ and $\Y$ are embedded in projective spaces $\textbf{P}^k=\textbf{P}(\Lambda^s(W))$ and $\textbf{P}^\ell=\textbf{P}(\Lambda^t(W))$, respectively. Then $\fZ$ is canonically embedded in $P^N$, where $N=(k+1)(\ell+1)-1$. By Remark \ref{rCartProd},  the homogeneous coordinates $w_{ij}$ of the point $(U,V)$ are expressed as homogeneous polynomials of degree 2 in the (Pl\" ucker) coordinates of $U$ and $V$. From these formulas one can see that, conversely, the coordinates of the points $U$ and $V$ can be obtained as ratios of the coordinates $w_{ij}$ on an open subset, for example, the subset of the points where $w_{00}\ne 0$.  

If $e_1\ld,e_n$ is a basis in $W$ then by Lemma \ref{lCartPluck} the Cartesian coordinates of a certain basis of $U$ (also $V$) can be rationally expressed in terms of the Pl\" ucker coordinates of $U$ (also $V$).

If $U\cap V=\{0\}$, using the Cartesian bases for $U$ and $V$, one can form the Cartesian basis $b_1,\ld,b_{s+t}$ for $U+V$. The coordinates of the vectors of this basis are used to compute the coordinates of $b_1\wedge\ldots\wedge b_{s+t}$ in the basis $\{ e_{i_1}\wedge\ldots\wedge e_{i_{s+t}}\,|\,i_1<\ld< i_{s+t}\}$ of $\Lambda^{s+t}(W)$. These coordinates are the homogeneous (Pl\" ucker) coordinates of $U+V$ in the variety $\Grass(s+t,W)$.

The composition of the above maps is also given by homogeneous polynomials. Now we have to exclude the pairs $(U,V)$ where $U\cap V\ne\{ 0\}$. The intersection is nonzero if the span of the union of the bases for $U$ and $V$ has dimension $< s+t$. This condition is given by equating to zero all minors of size $s+t$ of the matrix composed of the Cartesian coordinates of the spanning vectors, in terms of the basis $e_1,\ld,e_n$. Using Lemma \ref{lCartPluck} once again, we obtain an equation in terms of Pl\"ucker coordinates which defines the open subset $\cO$. Note that $\cO$ is not empty because there do exist pairs $(U,V)$ with $U\cap V=\{ 0\}$.

\end{proof}

In what follows we will be dealing with the ideals of algebras, rather than subspaces. We denote by $\cJ_m(P)$ the set of $m$-dimensional ideals of a finite-dimensional algebra $P$ and by $\cJ(P)$ the disjoint union  $\cJ(P)=\bigcup_{m=1}^n\mathcal{J}_m(P)$ of all ideals of $P$. 

\begin{Lemma}\label{lIdeajs} Given an algebra $P$ with $\dim_\F P=n\le\infty$ and a natural number $m$, the set $\cJ_m(P)$ is  a closed subspace in $\Grass(m,P)$.
\end{Lemma} 

\begin{proof}
If $e_1,\ld,e_n$ is a basis of $P$ and $u_1,\ld,u_s$ a basis in an ideal $K$ of $P$ then the matrix of coordinates of the latter basis in terms of the former one has rank $m$. If we add to this matrix the rows of coordinates of $e_iu_j$ and $u_je_i$, the rank of the enlarged matrix cannot grow and so all the minors of order $m+1$ must be zero. This necessary and sufficient condition  is polynomial in terms of the Pl\" ucker coordinates. It selects the subvariety of $m$-dimensional ideals in $\Grass(m,P)$.
\end{proof}

\section{Main results}\label{sMR}

In this chapter, the base field $\F$ is always assumed algebraically closed.

\subsection{Generic ideals}\label{ssINA}

Let us fix a $c$-step nilpotent primitive class $\V$, $c\ge  2$, and a  $\V$-free $c$-step nilpotent algebra $F_n=F_n(\V)$ of rank $n$. In this section, we shorten the notation $I_k(F_n)$ for the $k$th term of the annihilator series of $F_n$ to $I_k$ and write $d_k = \dim_\F I_k$. We often write $I$ in place of $I_1$. As noted just before Section \ref{sNAV}, the length of the annihilator series can be shorter than $c$. We denote by $h$ the least index such that $I_{h}=F_n$. We now introduce $\X_{i,s_i}=\Grass(s_i, I_{i+1})$, the Grassmann variety of subspaces of some dimension $s_i>0$ in the annihilator $I_{i+1}$, for each $i= 1,\ld,h-1$. We choose $s_i\le \dim_\F (I_{i+1}/I_i)$, for all $i=1,\ld,h-1$.

For each $U\in \X_{i,s_i}$ and each placement of brackets $\varpi\in \varPi$, with $|\varpi|=i$, we have defined a subspace $\varpi U$ just before Lemma \ref{27.24}. Since $U\subset I_{i+1}$, we have  $\varpi U\subset I_1=I$. 

\begin{Lemma}\label{lXsi} The map $U\mapsto \varpi U$, depending on the placement of brackets $\varpi$, is given by polynomials in the homogeneous coordinates on an open subset of the Grassmann variety of $s_i$-dimensional subspaces of the space $I_{i+1}$. It takes values in a Grassman variety $\Grass(t,I)$, for some $t$.
\end{Lemma}
\begin{proof} Let  $e_1,\ld,e_{d_{i+1}}$ be a basis of $I_{i+1}$ and $U$ be a subspace in $I_{i+1}$. As mentioned earlier, $U$ is defined by an  $s_i$-vector $f_1\wedge\ldots\wedge f_{s_i}$, where $f_1\ld,f_{s_i}\in I_{i+1}$. This $s_i$-vector is a linear combination of $s_i$-vectors $e_{j_1}\wedge\ldots\wedge e_{j_{s_i}}$, $j_1<\ldots<j_{s_i}$, forming a basis in $\Lambda^{s_i}(I_{i+1})$, the coordinates of the linear combination being the Pl\"ucker coordinates of $U$ in the Grassmann variety $\Grass(s_i,I_{i+1})$ as a subvariety in the projective space $\textbf{P}(\Lambda^{s_i}(I_{i+1}))$. 

According to Lemma \ref{lCartPluck}, there is a basis $f'_1,\ld,f'_{s_i}$ in $U$ such that the coordinates of the vectors of this basis with respect to $e_1\ld,e_{d_{i+1}}$ are homogeneous rational functions of degree zero in the Pl\"ucker coordinates of $U$.

Let us first consider the case $|\varpi|=1$ say $\varpi=\rho$. 

Using the the Cartesian coordinates of $f'_1,\ld,f'_{s_i}$, one can find the coordinates of the vectors spanning $Uv$, where $v$ is any monomial in the free generators of $F_n$ (linear functions whose coordinates depend on the structure constants of $F_n$). Let $g_1,\ld,g_r$ be a basis of $UF$. Then the Pl\" ucker coordinates of $UF$ in $\Grass(r,I_i)$ are the coordinates of $g_1\wedge\ldots\wedge g_r$ in the standard basis of the $r$th exterior power of $I_i$ (we assume that the basis of $I_i$ is a part of a basis of $I_{i+1}$). These Pl\" ucker coordinates of $UF$ do not depend on the choice of Cartesian basis in $UF$. As a result, the Pl\" ucker coordinates of $UF$ are homogeneous rational functions of degree zero in terms of the Pl\" ucker coordinates of $U$.

Now let us choose $U\in \Grass(s_i,I_{i+1})$ for which the number $r$ appearing above is maximal possible and denote this number by $r_0$. If $U$ is such that $r<r_0$ then the rank of the matrix formed by the Cartesian coordinates of $f'_1v,\ld,f'_{s_i}v$, spanning $UF$, is less than $r_0$ and then the Cartesian, hence Pl\" ucker coordinates of $U$ satisfy a fixed system of homogeneous equation, hence $U$ belongs to a proper closed irreducible subvariety of  $\X_{i,s_i}$. By Lemma \ref{LSH}, Claim (a), the dimension of this subvariety is strictly less that $\dim \X_{i,s_i}$. So the map $U\to UF$ is a polynomial map defined on an open subset of $\X_{i,s_i}$, that is, of a quasi-projective variety. This completes the proof in the case where $\varpi=\rho$. 

The proof in the case of any other of finitely many different  placements of brackets $\varpi$ goes in the same way except that instead of considering vectors $f'_1v,\ld,f'_{s_i}v$ spanning $UF_n$, in the case of, say, $(F_nU)F_n$, one has to consider $(v_1 f'_1)v_2,\ld,(v_1f'_{s_i})v_2$, for any two monomials in the free generators of $F_n$.
\end{proof}

To proceed further, for all $1\le j\le d_1$, we introduce Grassmann varieties  $\fZ_j=\Grass(j,\,I)$ of $j$-dimensional subspaces in $I$. Then $\dim \fZ_j=j(d_1-j)$. Finally, for any tuple $\textbf{s}=(s_h,\ld,s_1)$,  $s_i\le\dim_\F(I_{i+1}/I_i)$, we define the product 
\[
\fZ_{\textbf{s},j}=\X_{h,s_{h}}\times\cdots\times \X_{1,s_{1}}\times \fZ_j. 
\]
The dimension of this variety is at most
\begin{equation}\label{eDP}
s(d_{h}+\cdots+d_2)+j(d_1-j)\mbox{ where }s=s_1+\cdots+s_h.
\end{equation}

The following property of the points $(U_h,\ld,U_1)$ in the variety $\X_{h,s_h}\times\cdots\times\X_{1,s_1}$ is generic:  the sum $U_h+\cdots+U_1$ is  direct. This follows because a nonzero intersection of  $U_i$ with $I_i$ is a property selected by an algebraic equation. In the same way, generic is the property of the points  $(U_h,\ld,U_1,V)\in \fZ_{\textbf{s},j}$ where  $V\in \fZ_j$, saying that  the sum of $U_h+\cdots+U_1+V$ is direct.
 
At this time, it is convenient to write $\fZ_{\textbf{s},j}$ as the union of quasi-projective varieties $\fZ_{\textbf{s},\varpi,i}$, each of which includes the points with $\dim_\F\varpi U_i=t_i$, where $t_i=t_i(\varpi)$ is chosen in Lemma \ref{lXsi}. 

For $j\le d_1-t_i$, and any placement of brackets $\varpi$, using Lemmas \ref{lproduct} and \ref{lXsi}, we have a  map of varieties $\vp = \vp(\textbf{s},\varpi, i,j)$ from an open subset of the variety $\fZ_{\textbf{s},\varpi,i}$ into the Grassmann variety $\Grass(p,F_n)$ where  $p=\dim \fZ_{\textbf{s},\varpi,i}+\dim_\F \varpi U_i$, which maps the point corresponding to the direct sum $U=U_h\op\cdots\op U_1\op V$ into the point corresponding to the direct sum 
\begin{eqnarray*}
\bar U = U_h\op\cdots\op U_1\op V \op \varpi U_i
\end{eqnarray*}
where $|\varpi|=i$.

Since $t_i+j \le d_1$, the sum above is almost always direct, meaning that such $V$  is generic in $\Grass(j,I)$. Now $\vp$ is defined on the open subset of a projective variety $\fZ_{\textbf{s},\varpi,i}$. 
\begin{Lemma}\label{lC1}
Each ideal $N$ of $F_n$, which is not in $I(F_n)$, belongs to the image of some map $\vp(\textbf{s},\varpi,i, j)$ where $i$ is such that $s_i>\frac{s}{c}$, $\varpi$ is chosen as in Lemma \ref{27.24} and $j$ is some number satisfying $j \le d_1-\gamma \frac{sn}{c}$.
\end{Lemma}
\begin{proof}
Indeed, $N= V+ (U_1\op U_2\op\cdots\op U_h)$ for some subspace $V\subset I$ in $F_n$, such that 
\[
U_i\subset I_{i+1}\mbox{ and }U_i\cap I_{i}=\{ 0\},\mbox{ so that } \dim_\F U_i = s_i \le d_{i+1}-d_{i}\; (i\le h),
\]
for some tuple $\textbf{s}$ such that $s=s_1+\ld+s_{h}>0$, following because $N\not\subset I$. In this case, we choose $i$,  $\varpi$ with $|\varpi|=i$, and $t_i=\dim_\F\varpi U_i$ as earlier, and then 
\[
\dim_\F \varpi U_i>\gamma \frac{sn}{c}.
\]

Clearly, this subspace is in the ideal $N$, since $N$ is an ideal of $F_n$. Therefore, we can complement $\varpi U_i$ by a direct summand $V$ of dimension $j<d_1-\gamma s \frac{n}{c}$ to $V$. It follows from this construction and the definition of $\vp(\textbf{s},\varpi,i, j)$ that the point corresponding to $N$ belongs to $\vp(\textbf{s},\varpi, i,j)(\fZ_{\textbf{s},\varpi,i})$.
\end{proof}

Note that the same image $\bar U$ under $\vp$ will be obtained if we replace $V$ by any subspace $V'\subset I(F_n)$, such that $V'\op \varpi U_i = V \op \varpi U_i$. 

The dimension of the quasi-projective variety of such subspaces $V'$ in $W=V\op \varpi U_i$ equals $j t_i=j(j+t_i-j)$, where $t_i = \dim_\F \varpi U_i$. It follows that the dimension of the preimage of every point in the image of $\vp(\textbf{s},\varpi, i,j)$ is at least $jt_i$. Now for $i$ and $\varpi$ as in Lemma \ref{lC1}, we have by Lemma \ref{27.24} that $t_i >\gamma s_i n^i \ge\gamma s\frac{n}{c}$.  This estimate from below of the dimension of the preimage of any point of $\bar U$ under $\vp$, taken together with (\ref{eDP}), implies that the dimension of $\fZ_{\textbf{s},\varpi,i}$ is at most 
\begin{equation}\label{eDPP}
s(d_{h+1}+\cdots+d_2)+j(d_1-j)-\gamma js \frac{n}{c} \le   s(d_{h+1}+\cdots+d_2)+j(d_1-j-\gamma s \frac{n}{c}).
\end{equation}
Here we have used Claim (b) in Lemma \ref{LSH}.

\smallskip 

We need one more lemma before we proceed to the main result of the section (Theorem \ref{tNSA}). The conclusion of this lemma is stronger than it is necessary for the main result of this section; it will be used in  the treatment of the generic properties of algebras, rather than ideals, in Section \ref{ssSFR}.

\begin{Lemma}\label{lC2} If $c\ge2$ then there exists $n_0=n_0(c)$ such that if $n\ge n_0$,  $i$ is chosen so that $s_i>\frac{s}{c}$ and $\varpi$ chosen the same as in Lemma \ref{lC1}, then the dimension of the image of the variety $\fZ_{\textbf{s},\varpi,i}$  satisfies the following inequality: 
\[
\dim\vp(\textbf{s},\varpi,i,j)(\fZ_{\textbf{s},\varpi,i}))<\frac{d_1^2}{4}-n^2.
\]
\end{Lemma}
\begin{proof}
Note that with  $i, \varpi$ fixed and any $j$, the second summand in (\ref{eDPP}) is at most 
\[
\frac{(d_1-\gamma s \frac{n}{c})^2}{4} \le d_1\frac{d_1-\gamma s \frac{n}{c}}{4}.
\] 
By adding to this the first summand in the right hand side of  (\ref{eDPP}), we will obtain 
\begin{equation}\label{eAmp}
\dim\vp(\textbf{s},i,\varpi,j)(\fZ_{\textbf{s},\varpi,i}))\le \left(\frac{d_1^2}{4} -n^2\right)- s\frac{\gamma d_1 n}{4c}  + s(d_{h}+...+d_2)+n^2.
\end{equation}

Note that the function $f(n) = \frac{\gamma d_1 n}{4c}$, with $n$ increasing, grows faster than $d_{h}+\cdots+d_2+n^2$. Indeed, by Lemma \ref{27.4}, $f(n)>\gamma \frac{\kappa_1}{4c} n^{c+1}$ and  $d_{h}+\cdots+d_2$ is bounded from above by a polynomial in $n$ of degree $c$. Considering that $s>0$, for the values of $n$ which are appropriately great, we have that $\frac{d_1^2}{4}-n^2$ is greater than the right hand side of (\ref{eAmp}), as needed.
\end{proof}

\begin{Theorem}\label{tAAI} For any $c\ge 2$ there is a  number $n_0=n_0(c)$ with the the following property. Suppose that $F_n$ is a relatively free $c$-step nilpotent algebra of rank $n$. Let $\fS\subset\cJ(F_n)$ be the set of ideals of $F_n$, which belong to $I(F_n)$. If $n\ge n_0$ then being in $\fS$ is a generic property for the ideals of $F_n$. Moreover, if $\T=\cJ(F_n)\setminus\fS$ then $\dim \T<\dim \mathcal{J}(F_n)-n^2$.
\end{Theorem}
\begin{proof}
All $N$ that are not in $I(F_n)$ are the points of the finite union of all images $\vp(\textbf{s},\varpi,i,j)(\fZ_{\textbf{s},j})$, for different values of $\textbf{s},\varpi,i,j$, in accordance with Lemma \ref{lC1}. Using Lemma \ref{lC2} and the fact that the dimension of the union is the maximum of the dimensions of the terms of the union, we can see that all these ideals are the points of a projective variety whose dimension is less than $\frac{d_1^2}{4}-n^2$. Incidentally, if $d_1=2k$ is even, $\frac{d_1^2}{4}=\dim\Grass(k, I(F_n))$,  the dimension of the Grassmannn variety of $k$-dimensional subspaces, equivalently, $k$-dimensional ideals of $F_n$ which are in $I(F_n)$. If $d_1=2k+1$ is odd, $\frac{d_1^2}{4}-n^2=k(k+1)+\frac{1}{4}-n^2<k(k+1)=(k+1)(2k+1-k)=\dim\Grass(k+1, I(F_n))$,  the dimension of the Grassmannn variety of $k+1$-dimensional subspaces, equivalently, $k+1$-dimensional ideals of $F_n$, which are in $I(F_n)$. The proof is now complete.
\end{proof}
\begin{Remark}\label{rd1} Using the restrictions of Theorem \ref{tAAI}, the following properties of the ideals of $F_n$ are generic. If $d_1=\dim_\F I(F_n)$ is even then a generic ideal of $F_n$ is contained in  $I(F_n)$ and has dimension $\frac{d_1}{2}$. If $d$ is odd then a generic ideal of $F_n$ is contained in  $I(F_n)$ and has  of one of two dimensions $\frac{d_1}{2}+\frac{1}{2}$ or $\frac{d_1}{2}-\frac{1}{2}$.
\end{Remark}

\subsection{Ideals and automorphisms of relatively free algebras}\label{ssARFNA}

Before we move with our proofs any further, we have to impose a restriction on the primitive classes in question, as described below.

\subsubsection{Central primitive classes} \label{sssCentral}

Given a relatively free algebra $F_n$ of a $c$-step nilpotent primitive class $\V$ of algebras, the action of $\GL(n)$ on $F_n^c$ introduced in Proposition \ref{lGLFc} does not need to be always extendible to the action  on the annihilator $I(F_n)$. One of the solutions to this problem is to restrict oneself to considering only the primitive classes $\V$ in which $F^c_n=I(F_n)$ for any $\V$-free algebra of any rank $n$. In many important particular cases, such as the primitive classes of all associative algebras, all nonassociative algebras, all commutative, all anticommutative and all Lie algebras  this condition is satisfied. Once we assume $F_n^c=I(F_n)$, for all $n$, a generic ideal of $F_n$ is contained in $F_n^c$ (Theorem \ref{tAAI}).

\begin{Definition}\label{dCentral} Let $c\ge 2$. A $c$-step nilpotent primitive class $\V$ of algebras over a field $\F$ is called \textit{central} if $I(F_n(\V))=F_n(\V)^c$, for all $n\ge 1$.  
\end{Definition}

\begin{Remark}\label{rFcnotI}
Probably the easiest counterexample of non-central primitive class is the $5$-step nilpotent primitive class $\V$ of Lie algebras with identity $((x_1x_2)(x_3x_4))x_5=0$ (``centrally metabelian'' Lie algebras). In the free algebra $F_5=F_5(\V)$ of rank $5$ one has $(x_1x_2)(x_3x_4)\in I(F_5)\setminus F_5^5$.
\end{Remark} 

\begin{Remark}\label{rIFc}
In the case where $\V$ is multihomogeneous, to check that $V$ is central, we only need to check that $I(F_n(\V))=F_n(\V)^c$ for the ranks $n\le c+1$. This is a direct consequence of  Lemma \ref{27.9}, Claim (a).
\end{Remark}

In the remainder of the paper we assume that $\V$ is a central primitive class, $F_n$ the $\V$-free algebra of rank $n>n_0$. We also recall the notation $\cJ$ for the variety of ideals in $F_n^c$.  Finally, we denote by $\cR\subset\cJ$ the set of  ideals $K$ of $F_n$ such that $\Stab(K)\subset\Scal(F_n)$(see Definition \ref{dScal}).

\begin{Theorem}\label{tNSA}  For each $c\ge 2$ there exists a natural number $n_0$ such that being in $\cR$ is a generic property for the ideals of $F_n$. If $\cU=\cJ(F_n)\setminus\cR$ then $\dim \cU<\dim \mathcal{J}(F_n)-n^2$. 
\end{Theorem}

\bigskip

Before we proceed to the proof, we make several remarks which will be useful also in the future.

We classify all the ideals of $F_n$ as Class 1, Class 2 and Class 3. Class 1 is formed by the ideals $K$ such that $K\not\subset F_n^c$.  
By Theorem \ref{tAAI}, there is a variety of ideals of $F_n$ of dimension strictly less than $\frac{d_1^2}{4}-n^2$ which contains all the ideals of Class 1. 

The ideals in $F_n^c$ fall into classes 2 and 3. The dimension of the variety of ideals (=subspaces) of dimension $m$ in $F_n^c$ is given by the quadratic function
\begin{equation}\label{eQuadratic}
q(m)=m(d-m)=\dim\Grass(m,d)=\dim\Grass(m,F_n^c).
\end{equation}  
The maximum value of $q(m)$ is a rational number $\frac{d^2}{4}$. Now to split the ideals in $F_n^c$ into two classes, we introduce an interval for the values of the dimension $m$:
\begin{equation}\label{eInterval}
\Omega=\left[\frac{d}{2}-n,\frac{d}{2}+n\right]
\end{equation}
An easy exercise shows that for the values of $m$ outside $\Omega$ we have $q(m)<\frac{d^2}{4}-n^2$. We say that the ideals of dimension $m$ in $F_n^c$ such that $m\not\in\Omega$ form Class 2. As a result,it is true that there is a variety of dimension strictly less than $\frac{d^2}{4}-n^2$ which contains all the ideals of Class 1 and Class 2.

Now  the ideals of Class 3 are those ideals of $F_n$ which are not in Class 1 or Class 2.

We now proceed to the proof of Theorem \ref{tNSA}.
  
\medskip

\begin{proof} By what we just mentioned, we only need to consider varieties of ideals of Class 3. We denote by $\M$ be the Zarisski closure of the set of all Class 3 ideals of $F_n$,  
with the following property. For each $V\in\M$ there is an $NS$-automorphism $\alpha$ such that $\alpha(V)=V$. We need to prove that $\dim \M<\frac{d^2}{4}-n^2$. Note that according to Proposition \ref{lGLFc}, $V$  is  invariant under the linear operator $\wA$ induced by $A\in GL(n)$, where $A:F_n/F_n^2\to F_n/F_n^2$ is induced by $\alpha$. 

Let us denote by $\cG$ be the set of  linear maps $\wA$ on $F_n^c$, which are obtained from all $A\in \GL(n)$, 
as in Proposition \ref{lGLFc}. The map $A\to\wA$ is a morphism of  varieties $\GL\!(n)\to \cG$, so that $\dim \cG \le n^2-1$. 

Using Lemma \ref{27.4}, we choose an integer $k = k(c)$ such that  
\begin{equation}\label{eq_lower}
kd>5n^2\mbox{ (we remember that }c\ge 2).
\end{equation}
If $n$ is sufficiently great, we can choose $k$ so that $k\le \frac{d}{4}$. 

We set $\Y=\Y_1\times\Y_2$ where $\Y_1=\Grass(k,F_n^c)$ and $\Y_2=\Grass(m-2k,F_n^c)$. 

The dimension of $\Y$ is given by
\begin{eqnarray}\label{28.10}
\dim \Y &=& k(d - k) + (m-2k)(d- m + 2k)\nonumber\\ &=& k(d-k) + \left(\frac{d}{2}  - 2k\right)\left(\frac{d}{2} + 2k\right)+\Theta.
\end{eqnarray}
where $\Theta$ is a number such that $-2n^2<\Theta\le 0$ (for sufficiently great  values of $n$).

Let us define the map $\vp:\fZ=\cG\times \Y\to \Grass(m, I)$, given by
\begin{equation}\label{eqphi}
\vp((\wA, (U, W)))= W * U * \wA(U).
\end{equation}
The right hand side is an $m$-dimensional subspace defined as follows. If $\{ w_1,\ld,w_{m-2k}\}$ is a basis of $W$ and $\{ u_1,\ld,u_k\}$ is a basis of $U$ then the $m$-dimensional space $W*U*\wA(U)$ is defined by an $m$-vector
\begin{equation}\label{*}
w_1\wedge\cdots\wedge w_{m-2k}\wedge u_1\wedge\cdots\wedge u_k \wedge \wA(u_1)\wedge\cdots\wedge \wA(u_k).
\end{equation}
Note that this map is not always well-defined because the $m$-vector (\ref{*}) can be zero in the case where the vectors $w_1,\ld,\wA(u_k)$ are linearly dependent. 

To handle this difficulty, we consider two cases.

{\sc Case 1}. In this case, we assume that $\rank_{\wA} V\le m-3k$.

\begin{Lemma}\label{lBUW}
Let $V$ be an $m$-dimensional $\wA$-invariant subspace in $F_n^c$. Assume that  $\rank_{\wA}\, V\le m-3k$. Then there exists a pair $(U,W)\in \Y$, such that $\vp((\wA, (U, W)))=V$.
\end{Lemma}

\begin{proof}
 Let us apply Lemma \ref{28.1}. Since $r\le m-3k$, we have that $\frac{2m+r}{3}\le m-k$, so that $\ell=m-k$ is within the limits for $\ell$ in the statement of that lemma. As a result, there is a subspace $U$ in $V$ such that $\dim_\F U = m-\ell =k$ and $U \cap \wA(U) = \{ 0\}$. Let $W$ be any complement of dimension $m-k$ to $U\oplus\wA(U)$ in $V$. Then $V=W\oplus U\oplus\wA(U)$.
\end{proof} 
 
Using the same argument as in the proof of Lemma \ref{lXsi}, the Pl\" ucker coordinates of $W\ast U\ast \wA(U)$ given by (\ref{*}) are homogeneous polynomials of the same degree in the projective coordinates of $\wA$, $U$ and $W$.  Thus $\vp$, given by (\ref{eqphi}), is a  map of an open subset of the irreducible projective variety $\fZ$.

To proceed, we first estimate $\dim \fZ$.

Since $\dim \cG = n^2-1$, it follows from (\ref{28.10}) that 
\begin{equation}\label{**}
\dim \fZ = k(d-k) + \left(\frac{d}{2}  - 2k\right)\left(\frac{d}{2} + k\right)+n^2-1+\Theta.
\end{equation}

In the statement of the next lemma the operators $\wA$ are induced by the operators $A\in\GL(n)$, as in Proposition \ref{lGLFc}. So $F_n^c$ is an $\wA$-module, for each $A\in\GL(n)$.

\begin{Lemma}\label{lAVM} Let $\cN$ be the set of ideals $V$ of dimension $m$ in $F_n^c$, such that $\rank_{\wA}\,V \le m-3k$. Then  $\dim\cN\le\frac{d^2}{4}-n^2$.
\end{Lemma}
 \begin{proof} 
To prove our claim, we first use Lemma \ref{lBUW} to choose $\wA\in \cG$, $U\subset F_n^c$ and	 $W\subset I$, with $\dim_\F W=m-2k$, $\dim_\F U=k$  such that $V=W\ast U\ast \wA(U)$. As a result, $\cN\subset\vp(\fZ)$ where $\vp$ is defined in (\ref{eqphi}).

In the triple $(\wA,U,W)$ just found if  we fix $\wA$ then  ``almost any'' pair of subspaces $(U,W)$ of dimensions $k$ and $m-2k$, respectively, can be chosen provided that $U,W\subset V$.  This follows because the linear dependence of $m$ vectors in (\ref{*}) is given by algebraic equations with  respect to the coordinates of $w_1,\ld,w_{m-2k}, u_1,\ld,u_k,\wA(u_1),\ld,\wA(u_k) $: one simply needs to equate to zero all $m\times m$ minors of  the matrix formed by the coordinates of these vectors. As a result, $\dim(\vp^{-1}(V))$ for such $V$ equals to the dimension of such variety of pairs $(U,W)$. Therefore,
\begin{eqnarray}\label{***}
&&\dim(\vp^{-1}(V))\ge k(s-k) + 2k(m-2k)\nonumber \\&=& k(3m-5k) > k\left(\frac{3d}{2}-3n -5k\right).
\end{eqnarray}

By Lemma \ref{LSH}, Claim (b), the estimate from above for $\dim \mathcal{N}$ is given by the difference of the dimensions provided by  the formulas (\ref{**}) and (\ref{***}), namely
\begin{eqnarray*}
&k(d-k) + \left(\frac{d}{2}  - 2k\right)\left(\frac{d}{2} + k\right) +n^2-1+\Theta - k\left(\frac{3d}{2}-3n -5k\right)\\
&< \frac{d^2}{4} - dk + 2k^2 +3kn+n^2+\Theta-1 \\
&< \frac{d^2}{4}-n^2, 
\end{eqnarray*} 

Let us explain the latter inequality.  We call $-dk$ the \textit{negative term}.

First, assume $c=2$ and use $kd >5 n^2$ from (\ref{eq_lower}) and $\Theta\le 2n^2$. Then we need to check $-5n^2+2k^2+3kn+3n^2-1<-n^2$ or $n^2-3kn-2k^2+1>0$, a true inequality because $k$ is fixed and $n$ is sufficiently great.

Now if $c\ge 3$, then the inequality is obvious because, with $n$ growing, the negative term grows as a polynomial of degree $c$ (see Lemma \ref{27.4}) while all the remaining terms grow as polynomials of degrees at most $c-1$. The choice of $k$ now does not matter, simply for different values of $k$ the number $n_0$ should be chosen appropriately.

As a result, $\dim\mathcal{\cN}< \frac{d^2}{4}-n^2$.
\end{proof} 

Thus we have proved Theorem \ref{tNSA} for the set of ideals $V$ admitting $NS$-automorphisms $\wA$, and such that $\rank_{\wA} V\le m-3k$ (Case 1).

To deal with the remaining case, we will use the number $k$, satisfying (\ref{eq_lower}) and the variety $\cG'$ that consists of the images of $NS$-automorphisms of $F_n$. We consider $m$-dimensional ideals $V$ in $F_n^c$, where $m\in\Omega$, each invariant under an element from $\cG'$. Notice the following relation: 
\begin{equation}\label{es2}
m-3k < d,
\end{equation}
following because by (\ref{eInterval}), $m-3k\le \frac{d}{2}+n-3k$ and $n-3k<\frac{d}{2}$ for great enough values of $n$. Also, we remember that by Lemma \ref{27.4} $d$ is bounded from below by a polynomial of degree $c\ge 2$.

{\sc Case 2}. Assume that  $\rank_{\wA}(V)> m-3k$.
Let $\Y_m=\Grass(m - 3k, F_n^c)$, the Grassmann variety of subspaces of dimension $m-3k$ in $F_n^c$. By (\ref{dimGr}), we have $\dim \Y_m=(m-3k)(d-m +3k)$. 

Let us choose a basis $\{e_j\,|\,j=1,\ld,d\}$ of $F_n^c$ and a standard basis  for the $n^2$ coordinates of $\wA$. We cover the set $\cG'$ by open subsets $\cG'(i)$ consisting of $\wA$ such $\wA\in\X'(i)$ if there is nonzero $b_{i1}$ in the matrix $B$ of $\wA$, with respect to the chosen basis of $F_n^c$ . Also denote by $\fZ_m(i)$ the subset  in $\cG'(i)\times \Y_m$ whose elements are the pairs $(\wA, U)$ such that for any $u\in U$ there is $\lambda\in\F$ that satisfies $\wA(u) = \lambda u$. Clearly, $\lambda$ must be the same for all vectors in $U$. The union of all $\fZ_m(i)$ is denoted by $\fZ_m$.
\begin{Lemma}\label{lZs} Each
$\fZ_m(i)$ is a quasi-projective variety.
\end{Lemma}
\begin{proof}
 Passing to the projective spaces, all points in $\textbf{P}(U)$ are fixed by the projective map $\textbf{P}(\wA)$ induced by $\wA$ in $\textbf{P}(F_n^c)$: $\textbf{P}(\wA)(\textbf{P}(u))=\textbf{P}(u)$, for any $u\in U$. 
  
 Choose a nonzero elements $b_{i1}$ in the $i$th column of the matrix $B$ of $\wA$ with respect to this basis. Then writing
\begin{equation}\label{e28.27}
\wA(u)=b_{i1}u
\end{equation} 
 makes the system of equations satisfied by $U$ homogeneous both with respect to the coordinates of $u$ and the coordinates of $\wA$. 

The coordinates of any basis $\{f_j\}$ of a subspace $U$ can be written as such linear combinations of the basis $\{e_i\}$, where according to Lemma \ref{lCartPluck}, the coefficients are ratios of homogeneous polynomials of the same degree in the Pl\" ucker coordinates of $\textbf{P}(U)$. Any projective map of an $(r-1)$-dimensional projective space  is determined by the images of $r+1$ generic points. So if we substitute the ratios of Pl\" ucker coordinates mentioned above into (\ref{e28.27}) for $u$ equal to each of $f_j$ and to their sum, we will obtain homogeneous algebraic equations in the Pl\" ucker coordinates of $U$ and the elements of $\wA$, which are equivalent to (\ref{e28.27}). Hence $\fZ_m(i)$ is a quasi-projective subvariety in $\cG'(i)\times \Y_m$.
\end{proof}

\begin{Lemma}\label{ldimZs} It is true that 
\[ 
\dim\fZ_m<\frac{d^2}{4} - d\frac{\delta n^{c-1}}{4} + n^2-1.
\]
\end{Lemma}
\begin{proof} We prove the claim for all $\fZ_m(i)$. Since $\fZ_m$ is covered by $\fZ_m(i)$, our cliam for $\fZ_m$ will follow.

 Let $B\in\X'(i)$. We consider the preimage $\pi^{-1}(B)\cap \fZ_m(i)$ in $\fZ_m(i)$ with respect to the projection $\pi: \cG'(i)\times \Y_m\to  \cG'(i)$. It consists of the pairs $(\wA,U)$ where each $U$ is a $(m-3k)$-dimensional subspace, a point in  $\Y_m$, consisting of eigenvectors for $B$. Such $U$ is contained in a (maximal) $B$-eigenspace $L$ of $F_n^c$.  Additionally, since $B$ is not scalar, it follows by Lemma \ref{34.3} that $\dim_\F L \le d-\delta n^{c-1}$. Hence all the second components of the pairs $(A,U)\in \pi^{-1}(B)\cap \fZ_m(i)$ belong to the variety of subspaces of dimension at most $m-3k$ of the space $L$ of dimension at most $d-\delta n^{c-1}$. (As a matter of fact, there are finitely many such ``large'' $B$-eigenspaces for different eigenvalues but their number is unimportant for our estimates!) By (\ref{dimGr}), this latter variety has dimension at most $(m-3k)(d-\delta n^{c-1}-m+3k)$.

Since $\cG'(i)$ is in the image of $\mathrm{PGL}\!(n)$, $\dim \cG'(i) \le n^2-1$, and it follows by Lemma \ref{LSH}, Claim (c), our choice of $m$, and the estimate (\ref{es2}) that 
\begin{eqnarray}\label{28.30}
\dim \fZ_m(i) &\le& (m-3k)(d-\delta n^{c-1}-m+3k) + n^2-1\nonumber\\ & \le& \frac{(d-\delta n^{c-1})^2}{4} +n^2+1\nonumber\\&\le& \frac{d^2}{4} - d\frac{\delta n^{c-1}}{4} + n^2-1
\end{eqnarray}
\end{proof}

 Now it follows by  Lemma \ref{27.4} that $m$ grows as a polynomial of degree $c$ in $n$. Since $k$ is a constant not depending on $n$, $m>6k$ for sufficiently great $n$ or $m-3k>\dfrac{m}{2}$. Then Lemma \ref{28.24} applies and $V$ has an eigenspace of dimension at least $m-3k$. 
  
  Let us choose in this eigenspace a subspace $U$ of dimension exactly $m-3k$; then we have obtained a pair $(B,U) \in \fZ_m(i)$ such that $U \subset V$. The variety $\fZ_m(i)$ has been defined before Lemma \ref{lZs}. 

Let us introduce one more Grassmann variety $\cP_m=\Grass(3k,F_n^c)$.
Since $U$ can be complemented in $V$ by a subspace $U'$ of dimension $3k$, we conclude that $V$ is a point in the image under the map $\vp_m(i):\fZ_m(i)\times\cP_m\to\cR_m$, which maps $(B,U,U')$ to $V=U+U'$. So $\dim\cR_m\le \dim(\fZ_m(i)\times\cP_m$. (See Lemma \ref{lBUW}). By (\ref{dimGr}), $\dim \cP_m = (3k)(d-3k)$, while $\dim\fZ_m(i)$ was estimated in (\ref{28.30}). Adding these two values, we  obtain
\begin{equation}\label{28.33}
\dim\cR_m\le\dim(\fZ_m(i) \times \cP_m)\le\frac{d^2}{4} - d\frac{\delta n^{c-1}}{4} + n^2-1+3k(d-3k).
\end{equation}
 
This number is less than $\frac{d^2}{4}-n^2$ if $n$ is sufficiently great because if $k$ is fixed as in (\ref{eq_lower}) then after subtracting $\frac{d^2}{4}$ from (\ref{28.33}) we will obtain a summand  $-d\frac{\delta n^{c-1}}{4}$ (negative term). Using Lemma \ref{27.4}, it follows, on the one hand, that the absolute value of $-(\frac{\delta}{2})d_1n^c$ grows at least as fast as a polynomial of degree $2c-1$, while on the other hand, that the growth of other terms is bounded by the polynomials in $n$ whose degree is at most $c$. We remember that $c\ge 2$.

As a result, the image under $\vp_m(i)$, which contains $\fZ_m(i)$, is a  variety of smaller dimension than  $\dim\Grass(m,F_n^c)-n^2$. The same is true for the variety $\fZ_m$, which is a finite  union of $\fZ_m(i)$ by all $i$, and finally for $\fZ$ which is a disjoint union of all $\fZ_m$.  Since for each $m$, $\fZ_m$ contains all the subspaces $V$ of Case 2, the proof of Theorem \ref{tNSA} is complete.
\end{proof}

\begin{Proposition}\label{tNSD} For each $c\ge 2$ there exists a natural number $n_0=n_0(c)$ with the following property. Suppose $F_n$ is a relatively free $c$-step nilpotent algebra of rank $n>n_0$. Let $\cU\subset\cJ$ be the set of ideals each being invariant only under the derivations of $F_n$, which are scalar modulo $F_n^2$. Then being in $\cU$ is a generic property for the ideals of $F_n$. Moreover, if $\cV=\cJ(F_n)\setminus\cU$ then $\dim \cV<\dim \mathcal{J}(F_n)-n^2$. 
\end{Proposition}

\begin{proof}
To prove this, we have to slightly modify the proof of Theorem \ref{tNSA}. First of all, we have to replace the variety $\cG$ of linear operators $\wA$ on $F_n^c$ induced by the operators $A\in\GL(n)$, acting on $F_n/F_n^2$, by the space $\cG'$ of operators $\wA$ on $F_n^c$ induced by the operators $A\in \mathfrak{gl}(n)$ (this is legitimate because of Proposition \ref{lGLFc}, considering that $\V$ is multihomogeneous if $\F$ is algebraically closed, hence infinite). Then $\dim\cG'\le n^2$, which causes only minor changes in the estimates of the proof. The key Lemma \ref{34.3} stating that $\dim_\F F_n^c-\rank_{\wA}(F_n^c)\ge\delta n^c$ is true both for $\wA$ induced from the automorphisms and derivations. The main obstacle in carrying out the proof of that Lemma could be the apparent requirement that $\dim_\F\wA(U)=\dim_\F U$ must be nonsingular in the definition (\ref{eqphi}) of an auxiliary map $\vp$. However the argument in Lemma \ref{lBUW} and in the  paragraph following that Lemma allows one to further decrease the open set in the variety $\fZ$ to include the requirement of non-degeneracy of $\wA$ on the space $U$ in question. No further changes in the proofs are needed.
\end{proof}

\begin{Remark}\label{rchar0} In the case where $\chr\F=0$, by Proposition \ref{pMultilin} and Corollary \ref{cDer}, we have an infinite set of derivations $D_\lambda$ extending the map $x\mapsto\lambda x$ on the free generating set $X$ of $F_n$. A subspace $V\subset F_n^c$ invariant under an $NS$-derivation $D$ will be invariant under any $D+D_\lambda$. Clearly, one of $D+D_\lambda$ is nonsingular, which eliminates any obstacles in the proof of Proposition \ref{tNSD}.
\end{Remark}

\subsection{Generic algebras}\label{ssSFR}
 
In this section we are dealing with the variety $\V_n$ of (isomorphism classes of) $n$-generated algebras in a $c$-step nilpotent central primitive class $\V$, $c\ge 2$. All algebras  in $\V_n$ have the form $F_n/K$ where $K$ is an ideal of $F_n$. By Malcev's Lemma \ref{tScrap8.4}, $F_n/K\cong F_n/L$ iff there is an automorphism $\delta\in\Aut{F_n}$ such that $\delta(K)=L$. So the study of the isomorphism classes of algebras in $\V_n$ is equivalent to the study of the variety of orbits under the action of the algebraic group $\Aut{F_n}$ on the variety $\cJ_n=\cJ(F_n)$ of ideals of $F_n$.

We start with some more facts from Algebraic Geometry. This can be found in \cite{HUM}.

\begin{Remark}\label{31.6}
Suppose $\cG$ is a  connected algebraic group. One says that $\cG$ acts on an algebraic variety $\fZ$ if there is everywhere defined regular map $\cG\times \fZ\to \fZ$, satisfying $(gh)(z)=g(hz)$ and $1z=z$, for all $g,h\in \cG$ and $z\in\fZ$. In this case, the set of orbits $\fZ/\cG$ is an algebraic variety and there is a rational map $\fZ\to\fZ/\cG$  which takes different values on different orbits and constant on each particular orbit.  The pre-image of any point is an orbit. Since the dimension of an orbit is at most the dimension of the acting group $\cG$, using Lemma \ref{LSH}, Claim (c), it follows that $\dim \fZ/\cG\ge\dim \fZ-\dim \cG$.
\end{Remark}

\medskip

\subsubsection{Key observations}\label{KO} 

Remark \ref{31.6}  allows us to identify the set of orbits $\cJ_n/\Aut{F_n}$ with the set $\V_n$ of $n$-generated algebras in $\V$, viewed up to isomorphism. The structure of an algebraic variety on $\cJ_n/\Aut{F_n}$ transfers to $\V_n$, making it an algebraic variety. Having in mind this structure of $\V_n$, we will be talking about the generic properties of algebras in $\V_n$, in the sense of Definition \ref{dGeneric}.   
 
In Section \ref{ssARFNA}, we already split the set of ideals of $F_n$ into three classes: Class 1 of ideals $K$ such that $K\not\subset F_n^c$, Class 2 of ideals of dimension $m\not\in\Omega$. All ideals of these two classes where shown to belong to a variety of dimension $<\frac{d^2}{4}-n^2$. When we switch to the varieties of algebras $F_n/K$, where $K$ belongs of one of these two classes, the dimensions of these varieties as well as their unions, is always bounded from above by $\frac{d^2}{4}-n^2$.

The remaining ideals belong to Class 3, that is, they belong to $F_n^c$ and have dimension $m\in\Omega$. We select in Class 3 a subclass, we call it Class 3a, consisting of the ideals each of which is stable under an $NS$-automorphism. By Theorem \ref{tNSA}, the ideals of Class 3a  also belong to a variety of dimension $<\frac{d^2}{4}-n^2$. By Remark \ref{31.6}, the varieties of orbits of these ideals under the action of $\Aut{F_n}$, hence the varieties of algebras in $\V_n$ of the form $F_n/K$ where $K$ is in Classes 1, 2 and Subclass 3a are contained in a variety of dimension $<\frac{d^2}{4}-n^2$.

The remaining ideals form Class 3b. This is the main class: the dimension of the variety of algebras $F_n/K$ where $K$ is an ideal of Class 3b is greater than $\frac{d^2}{4}-n^2$. To prove this, we denote by $\cS_m$ the set of ideals in Class 3b having dimension $m$ and consider the action of $\Aut{F_n}$ on $\cS_m$.  For each $K\in\cS_m$ we have $\Stab_{\Aut{F_n}}(K)\subset\Scal (F_n)$ (see Definition \ref{dScal}). It follows that
\[
\dim\cS_m/\Aut{F_n^c}=\dim\cS_m-\dim\PGL(n)=\dim\cS_m-(n^2-1).
\]
Thus the dimensions of the varieties of orbits of $\cS_m$ under the action of $\Aut{F_n}$ drop by the same value $n^2-1$ when compared with the dimensions of the Grassmann varieties of ideals themselves, uniformly for all $m\in\Omega$. 

So while looking at the generic algebras, one can restrict oneself to the algebras of the form $P=F_n/K$ where $K\subset F_n^c$, $\dim_\F K=m$, $m\in\Omega$ and $\Stab_{\Aut{F_n}} K\subset\Scal(F_n)$. We can also calculate the dimension of the variety of ideals $\cJ_n$ of $F_n$ as follows.  Then 
\begin{equation}\label{dimGenId}
\dim\cJ_n=\left\{\begin{array}{ll}\dfrac{d^2}{4}&\mbox{ if }d\mbox{ is even}\\\\ \dfrac{d^2}{4}-\dfrac{1}{4}&\mbox{ if }d\mbox{ is odd}.
\end{array}\right.
\end{equation}
The dimension of the variety of algebras in $\V_n$ is given by the following
\begin{equation}\label{dimGenAlg}
\dim\V_n=\left\{\begin{array}{ll}\dfrac{d^2}{4}-n^2+1&\mbox{ if }d\mbox{ is even}\\\\ \dfrac{d^2}{4}-n^2+\dfrac{3}{4}&\mbox{ if }d\mbox{ is odd}.
\end{array}\right.
\end{equation}

The remainder of this section is devoted to the derivation of generic properties of algebras in $\V_n$ where $\V$ is a $c$-step nilpotent central class of algebras over an algebraically closed field $\F$.

\begin{Theorem}\label{36} For any $c\ge 2$ there is natural $n_0$ such that the following is true. Let $\V$ be a central $c$-step nilpotent primitive class of  algebras over an algebraically closed field $\F$, $c\ge 2$. Let $F_n$ be a $\V$-free algebra of rank $n\ge n_0$. Set $\V'=\V\cap\fN_{c-1}$. Then the following are true for a \textbf{generic} $n$-generated algebra $P$ in $\V$.
\begin{enumerate}
\item[$\mathrm{(i)}$] $P/P^c\cong F_n/F_n^c$ is a $\V'$-free algebra of rank $n$; 
\item[$\mathrm{(ii)}$] The dimension of $P$ is given by the following. If $\dim_\F F_n^c$ is even then 
\[
\dim_\F P=\dim_\F F_n - \frac{1}{2} \dim_\F F_n^c.
\]
If  $\dim_\F F_n^c$ is odd then 
\[
\dim_\F P=\dim_\F F_n - \frac{1}{2} \left(\dim_\F F_n^c \pm 1\right).
\]

\item[$\mathrm{(iii)}$] Let $n>n_0$ and $0<r\le \sqrt{n}$. Then $P$ contains a subalgebra isomorphic to $F_r(\V)$; in particular, if $n\ge c^2$, then $\var P=\V$.
\item[$\mathrm{(iv)}$] $P$ is graded by the degrees with respect to any $n$-element generating set;
\item[$\mathrm{(v)}$] Any automorphism of $P$ is scalar modulo $P^2$ and for any generating set $a_1,\ld,a_n$ of $P$ and any nonzero $\lambda\in\F$ the map $a_i\to\lambda a_i$ extends to an automorphism of $P$;
\item[$\mathrm{(vi)}$] Any derivation of $P$ is scalar modulo $P^2$ and for any generating system $a_1,\ld,a_n$ of $P$ and any $\lambda\in\F$ the map $a_i\to\lambda a_i$ extends to a derivation of $P$;
\item[$\mathrm{(vii)}$] $\Aut{P}$ is a solvable algebraic group which is a semidirect product of a $(c-1)$-step nilpotent normal subgroup and the multiplicative group of the field $\F$; one has \[
\dim\Aut P=(\dim_\F P-n)n+1;
\]
\item[$\mathrm{(viii)}$] $\Der P$ is a solvable Lie algebra which is a semidirect sum of a $(c-1)$-step nilpotent ideal and the additive group of the field $\F$; one has 
\[
\dim_\F(\Der P)=(\dim_\F P-n)n+1;
\]
\item[$\mathrm{(ix)}$] If $c=2$ then $\Aut P=\Def\,(P)$, as in Example \ref{rtwostep}.
\item[$\mathrm{(x)}$] If $c=2$, then  $\Der P=\default\,(P)$, as in Example  \ref{rtwostep}.
\item[$\mathrm{(xi)}$] The minimal number of relations defining  $P$ in $\V$ is 
\[
\dim_\F F_n - \frac{1}{2} \dim_\F F_n^c\mbox{  if }\dim_\F F_n^c \mbox{  is even}
\]
and 
\[
\dim_\F F_n - \frac{1}{2} \left(\dim_\F F_n^c \pm 1\right)\mbox{  if }\dim_\F F_n^c\mbox{  is odd.}
\] 

\item[$\mathrm{(xii)}$] $P^c$ is a ``large ideal'' in the sense that there is a constant $K=K(c)$ such that 
\[
\frac{\dim_\F P^c}{\dim_\F P} > 1-\frac{K}{n}.
\]
Additionally, every subspace of $P^c$ is a characteristic ideal of $P$.
\end{enumerate}
\end{Theorem}

\medskip 

\noindent \textit{Proof} of (i).
As follows from Key observations, we only need to consider the case of algebras $F_n/K$ where $K\subset F_n^c$. Then $P/P^c\cong F_n/F_n^c$ is a $\V'$-free algebra, as claimed.$\hfill\Box$ 

\bigskip

\noindent \textit{Proof} of (ii). By Key observations, $P\cong F_n/K$  where $K\in\cS_m$ is an ideal  in Class 3b.  If $\T_m$ is the variety of orbits of $PGL(n)$ in $\cS_m$ then by Lemma \ref{LSH}, Claim (b'), $\dim \T_m=\dim\cS_m-(n^2-1)$. It then follows that the greatest dimension of $\T_m$ is achieved when $m=\frac{d}{2}$ if $d=\dim_\F F_n^c$ is even or $m=\frac{d}{2}\pm\frac{1}{2}$ is $d$ is odd. The proof of Claim (ii) is complete. $\hfill\Box$

\medskip

\noindent \textit{Proof} of (iii). We fix in $F_n$ a collection of $q=\left[\frac{n}{r}\right]$ subalgebras $F(i)$, generated by pairwise disjoint $r$-element subsets of the set $\{x_1,\ld,x_n\}$ of free generators of $F_n$. It is obvious (see similar argument in the first paragraph of the proof of Lemma \ref{27.4}) that the sum of all subalgebras $F(i)$ is direct.

We set $C=C(r)=\dim_\F F(i)^c$. By Lemma \ref{27.9}, $C < \kappa_2 r^c \le \kappa_2 n^{c/2}$. It follows that for $U=\sum_i F(i)^c$ we have $\dim_\F U = Cq\le \kappa_2 n r^{\frac{c}{2} - 1}$.

Let us estimate the dimension of the set $\fZ$ of ideals $K$ of $F_n$ such that $F_n/K$ has no subalgebras isomorphic to $F_r(\V)$. In particular, $K\cap F(i)\ne\{ 0\}$ for all $i$. Since $K$ contains a nonzero vector from every $F(i)$, it follows that $K$ has a subspace $W\subset U$ of dimension $q$.

Having in mind Key Observations, we only consider the case where $K\subset F_n^c$ and $\dim_\F K=m\in \Omega$ (see Definition \ref{eInterval}). In this case, there is a subspace $V\subset F_n^c$ such that $K=W\oplus V$. Consequently, $K$ is the image of the pair $(W,V)\in \Grass(q,U)\times\Grass (m-q, F_n^c)$ under the map $(W,V)\to W\oplus V$ defined in the customary way on an open subspace of $\Grass(q,U)\times\Grass (m-q, F_n^c)$, of the same dimension. Since $c\ge 2$, for the dimension in question we will have
\begin{eqnarray}\label{estimate}
&q(Cq-q) + (m-q)(d-(m-q))< Cn^2/r^2 + d^2/4\\&\le \kappa_2 r^{c-2}n^2+ d^2/4 \le \kappa_2 n^{(c-2)/2}n^2+ d^2/4\le \kappa_2 n^c + d^2/4.\nonumber 
\end{eqnarray}

Next we notice that the image $K$ of a pair $(W,V)$ does not change if we replace $V$ by virtually any subspace $V'$ of $W\oplus V$ of the same dimension as $V$. Therefore, the preimage of $K$ has dimension at least $\dim \Grass(q, m)> \theta n^{c+ 1/2}$, for some constant $\theta=\theta(c,r)>0$. This is true because if  $m\in\Omega$, then by Lemma \ref{27.4} it is bounded from below by a polynomial of degree $c$ in $n$ .

This estimate together with the estimate (\ref{estimate}) bounds from the above the dimension of the set of ideals $K$ under consideration by the value 
\[
d^2/4 +\kappa_2 n^c- \theta n^{c +1/2} < d^2/4 -n^2.
\]
 Hence the dimension of the image after the factorization by the action of $\PGL(n)$ can be estimated by the same value.  Using (\ref{dimGenAlg}), we conclude that a generic $n$-generated algebra does contain a free subalgebra of rank $r$ if $n>n_0$ for some $n_0=n_0(c,r)$.
 
Finally, if $n\ge c^2$ then $P$ contains a subalgebra isomorphic to $F_c(\V)$. So $\var P\supset \var{F_c(\V)}$. By Remark \ref{rvarFc}, $\var{F_c(\V)}=\V$, and the proof of (iii) is complete.$\hfill\Box$

 \medskip
 
 \noindent \textit{Proof} of (iv). Using (i), we may assume that $P=F_n/K$, where $K\subset F_n^c\subset F_n^2$. Hence the preimages $y_1,\ld,y_n$ in  $F_n$ of the vectors $a_1,\ld,a_n$ of any generating system of $n$ elements for $P$ are linearly independent  $\mod F_n^2$. By Corollary \ref{cGFS}, Claim 1, $\{ y_1,\ld,y_n\}$ is the set of free generators in $F_n$. Now using  equation (\ref{eRelGrad}) and Lemma \ref{L1.3} we obtain a grading of $F_n$ by $X=\{y_1,\ld,y_n\}$. The ideal $K$ is graded with respect to this grading since $K\subset F_n^c$. So the grading of $F_n$ induces the grading of $P=F_n/K$.$\hfill\Box$
 
 \medskip
 
\noindent \textit{Proof} of (v). Assume that $\alpha\in\Aut{P}$ is an $NS$-automorphism. We use Malcev's Lemma \ref{tScrap8.4}. By (i) $P=F_n/K$ where $K\subset F_n^2$. Then $\alpha$ is induced by an $NS$-automorphism $\delta$ of $F_n$ which leaves $K$ invariant. By Key observations, $K$ is in Classes 1, 2 or 3a. Hence $P=F_n/K$ is not generic, a contradiction.

To prove that any map $a_i\to\lambda a_i$ extends to an automorphism of $P$, we use Claim (iv). Then $P=P_1\oplus\cdots\oplus P_c$ is the grading by the degrees with respect to $a_1,\ld,a_n$ and so the map $w\to\lambda^m w$ where $w$ is any element of $P_m$, $m=1,\ld,c$, is an automorphism of $P$ providing the desired extension of the original map of the generators of $P$.  $\hfill\Box$ 

\medskip

\noindent \textit{Proof} of (vi). We can repeat the argument of (v). Since $\F$ is infinite, the variety $\V$ is multihomogeneous, and according to Proposition \ref{pID} any dervation of $F_n/K$ is induced by a derivation of $F_n$, preserving $K$. Additionally, by Proposition  \ref{tNSD} the ideals $K$ invariant under nonscalar derivations are negligible. $\hfill\Box$ 

\medskip

\noindent \textit{Proof} of (vii). By Malcev's Lemma \ref{tScrap8.4}, if  $P=F_n/K$ and $\ve:F_n\to P$ is a natural homomorphism with kernel $K$ then for any $\alpha\in\Aut{P}$ there is $\delta\in \Aut F_n$ such that $\alpha\ve=\ve\delta$. 

Now in $\Aut{F_n}$ there is a subgroup $P$ consisting of automorphisms $\delta$ for each of which there is $\lambda\in\F^\ast$ such that $\delta(x_i)=\lambda x_i$, for all $i=1,\ld,n$. Since $K\subset F_n^c$ is invariant under all these $\delta$ and since $F_n/F_n^2$ maps onto $P/P^2$, we have a subgroup $\wP$ in $\Aut{P}$ isomorphic to $P$. If $a_i=\ve(x_i)$, for all $i=1,\ld,n$, we have $\delta(a_i)=\lambda a_i$.

There is a normal subgroup $Q\subset\Aut{F_n}$ consisting of automorphisms $\delta_{w_1,\ld,w_n}$, where $w_i\in F_n^2$, mapping each $x_i$ to $x_i+w_i$, $i=1,\ld,n$ (see Corollary \ref{cGFS}, Claim 2). These automorphisms act trivially on all factors $F_n^i/F_n^{i+1}$, $i=1,2,\ld,c$, so  $Q$ is nilpotent, because it can be represented by unitriangular matrices. 

Now $Q$ acts trivially on  $F_n^c$ and $P$ acts by scalar automorphisms on $F_n^c$, hence both $P$ and $Q$ leave $K$ invariant. As a result, we have subgroups $\widetilde{P}$ and $\widetilde{Q}\subset\Aut{P}$ induced by $P$ and $Q$. By  Claim (iv) of this theorem, $\Aut{P}=\wP\widetilde{Q}$.  Now $\dim\widetilde{Q}=\dim_\F (P^2)^n=n(\dim_\F P-n)$. Since $\dim\wP=1$, we finally obtain 
\[
\dim\Aut{P}=n(\dim_\F P-n)+1.
\]
Since $\widetilde{Q}$ is nilpotent and $\wP$ abelian, the group $\Aut{P}$ is  solvable.$\hfill\Box$

\medskip 

\noindent \textit{Proof} of (viii). The proof of (viii) is analogous to the proof of (vii). 

\noindent \textit{Proof} of (ix) and (x).  Both claims follow by application of (v), in the case of automorphisms, (vi), in the case of derivations, and equations $\Scal(P)=\Def(P)$ and $\scal(P)=\default(P)$ from Example \ref{rtwostep}. $\hfill\Box$

\medskip

\noindent \textit{Proof} of (xi). If an ideal $K$ in $F_n$ is contained in $F_n^c$ then its dimension is precisely the minimum number of defining relations needed to define $P=F_n/K$ in $\V_n$. (This number is also the lower bound for the defining relations of $P$ in any greater primitive class!) It remains to refer to laim (ii) of t his Theorem.

\medskip

\noindent \textit{Proof} of (xii). The first claim follows because $\dim_\F P^c$ is bounded from below by a polynomial in $n$ of degree $c$ while $\dim_\F P/P^c$ is bounded from above by a polynomial of degree $c-1$.

The second claim follows from the description of $\Aut{P}$: every $\alpha\in\Aut{P}$ acts on $P^c$ as a scalar transformation.$\hfill\Box$

A remark in the spirit of Claim (iii) of the previous theorem is the following.

\begin{Remark}\label{rSubalgebras}\label{rGenTup} Let $\V$ be $c$-step nilpotent primitive class and $Q$ a finite - dimensional algebra in $\V$ that contains a $\V$-free subalgebra of rank $r$. Denote by $\cG$ the set of $r$-tuples $(g_1,\ld,g_r)$ in an affine space $\underbrace{Q\times\ldots\times Q}_r$ such that $g_1,\ld,g_r$ freely generates a $\V$-free subalgebra of $Q$. Then $\cG$ is a Zarisski open subset of $\underbrace{Q\times\ldots\times Q}_r$, of the same dimension $r\dim_F Q$.   

Indeed, let $D=\dim_\F F_r$. Then a subalgebra $R$ generated by $\{ g_1,\ld,g_r\}$ is $\V$-free, freely generated by this set, if and only if $\dim_\F H=D$. Let us fix a basis $\{e_i\}$ of $Q$. Given an $r$-tuple $g_1,\ld,g_r$, we can form a matrix $S$, whose rows are formed by the coordinates  of the elements $g_1,\ld,g_r$ and their product of length at most $c$, with respect to $\{e_i\}$. Then the fact that the $\dim_\F Q<D$ is equivalent to the fact that all minors of size $D$ in $S$ equal zero. Thus we get a system of polynomial equations, defining a Zarisski closed subset whose complement $\cO$ is an open set. Since $Q$ contains a $V$-free subalgebra of rank $r$, $\cO$ is not empty, which completes the proof of our claim.

Less formally, we can say that  if some $r$-tuple of elements of $Q$ freely generates a $\V$-free subalgebra then almost any $r$-tuple does the same. In particular, by (iii), almost any $r$-tuple of elements of a generic algebra $P$ freely generates a $\V$-free subalgebra.

\end{Remark} 

In Section \ref{ssWGI} we will need another result, close to (iii) from Theorem \ref{36}.

\begin{Lemma}\label{lFor Width}
Let $w(x_1,\ld, x_t)\in\cF_\infty$ be a polynomial without monomials depending only on one variable. Suppose  that $w(x_1,\ld, x_t)=0$ is not an identity in $\V$. Then $w(x_1,\ld, x_t)=0$ is not an identity in a generic algebra $P$. Moreover, there exists an increasing quadratic function $\beta(n)$ such that $\dim_\F w(A)\ge\beta(n)$.
\end{Lemma}
\begin{proof}One can restrict oneself to a consequence of $w$ where all monomials have degree $c$. Indeed, if $w\not\in F^c$ then for a new variable $y$ either $wy=0$ or $yw=0$ is not an identity in $\V$, because $I(F_n)=F_n^c$. Then one can use Lemma \ref{L*} to conclude that the dimension of the verbal ideal $w(F)$ is bounded from below by a polynomial $\alpha(n)$, which is quadratic in $n$.

Note that $w(A) = w(F_n/K)=(w(F_n)+K)/K\cong w(F_n)/(K\cap w(F_n))$. We want to show that the set of algebras for which $\dim_\F w(A)$ is not bounded from below by a quadratic function of $n$ is negligible. As follows from Basic observations, we only need to consider ideals $K$ inside $F^c$ and if in $F^c$ then of dimension $m\in\left[\frac{d}{2}-n,\frac{d}{2}+n\right]$ and invariant only those automorphisms of $F_n$ which are scalar modulo $F^2$. 

We first consider the variety $\fZ_m$ of such ideals $K\subset F^c$ of dimension $m$ from the above interval which additionally satisfy  $m\le 1.1\dim_\F w(F_n)$ and $\dim_\F (K\cap w(F_n))>0.9 \dim_\F w(F_n)$. Let us set $\ell=\left[0.9\dim_\F w(F_n)\right]$.

If $\fZ_m$ is nonempty then $m > 0.9\dim_\F w(F_n)$, so that
\[
\dim_\F w(F_n)< \frac{10}{9}m \le \frac{10}{9}
\left(\frac
{d}{2}+n\right) < 0.57 d
\]
 for all great enough $n$. In this case there is a basis in $K$ in which the first $\ell$ vectors are in $w(F_n)$. It follows that $K$ is the image of a pair $(U,V)$ of the product 
\[
\Grass\left(\ell, w(F_n)\right) \times \Grass\left(m-\ell, F_n^c\right)
\] 
under the map $\ast$ (see \ref{eqphi}). Hence by Lemma \ref{LSH}, Claims (b) and (c), we have the following.
\begin{eqnarray*}
\dim \fZ_m\le \ell\left(\dim_\F w(F_n)-\ell\right) + \left(m-\ell\right)d.
\end{eqnarray*}
Here the first summand is less than
\[
0.09\dim_\F w(F_n)^2+d <0.09(0.57d)^2+d < 0.03d^2+d,
\]
while the second is less than 
\begin{eqnarray*}
\left(1.1\dim_\F w(F_n) -\ell\right)d < 0.2(\dim_\F w(F_n) +1)d\le 0.2d^2 + d,
\end{eqnarray*}

The sum of these two is less than $0.23 d^2 +2d < \frac{d^2}{4} -n$ for all great enough $n$ because $d$ is bounded from below by a quadratic function of $n$.

We have shown that $\dim \fZ_m<\frac{d^2}{4} -n^2$. Since all these ideals have trivial stabilizers in $\PGL(n)$ the dimension of the set of $F_n/K$ with such $K$ is strictly than $\dim\V_n$.

It remains to consider the cases of factorization by an ideal $K\subset F^c$ whose dimension $m$ belongs to that interval and, besides, either $m>1.1\dim_\F w(F_n)$ or $\dim_\F (K\cap w(F))\le 0.9 \dim_\F w(F)$.

In this case, using $\dim_\F w(A)=\dim_\F w(F_n)/(K\cap w(F_n))$, it follows that in both cases the dimension is bounded from below by the value 
\[
\min\left\{0.1; \frac{1}{9}\right\}\dim_\F w(F).
\]
Hence the dimension in question is at least $0.1\alpha(n)$. We  conclude the proof of Claim (iii) by setting $\beta(n) =0.1\alpha(n)$.
\end{proof}

\medskip

\begin{Remark}\label{rScrap29.3}
Any $2$-nilpotent primitive class is central. Indeed, it is obvious that $F_n^2\subset I(F_n)$. Also, an element outside of $F_n^2$ can be included in a free generating set of $F_n$ (see Corollary \ref{cGFS}). Clearly, if a free generator is in $I(F_n)$ then $I(F_n)\subset F_n$. Thus $I(F_n)= F_n^2$ and $\V$ is central.  Now the normal subgroup $\widetilde{Q}$ from the proof of Theorem \ref{36}, Claim (vi), is simply the additive group of a vector space $(P^2)^n$, on which $\wP\cong\F^\ast$ acts by the scalar multiplication. 

Indeed, given $\psi\in\wP$, there is $\lambda$ such that $\psi(a_i)=\lambda a_i$. For $\theta\in\widetilde{Q}$ there are $u_1,\ld,u_n\in P^2$ such $\theta(a_i)=a_i+u_i$. Note that $\psi(u_i)=\lambda^2u_i$. In this case, $\psi\theta\psi^{-1}(a_i)=a_i+\lambda u_i$. So the tuple $(u_1,\ld,u_n)$ maps to $(\lambda u_1,\ld,\lambda u_n)$, as needed. 

Hence the automorphism group is almost always one and the same in the case, where $n$ is great enough and $d$ is even, or there are two such groups in the case where $n$ is great enough and $d$ is odd.
\end{Remark}

\subsection{Generic quotient-algebras of generic algebras}\label{sGFG}

As earlier, we consider generic properties of algebras in a fixed $c$-step nilpotent primitive class $\V$, where $c\ge 2$.
A finite-dimensional algebra $Q$ is called a \textit{generic quotient-algebra} for $\V$ if $Q$ is a homomorphic image of a generic algebra in $\V_n$, as soon as $n>n_0=n_0(Q)$. In particular, $n_0$ is greater that the minimal number of generators of $Q$.

In this section, we keep assuming that $\V$ is a central primitive class, that is, $I(F_n)=F_n^c$, for any $\V$-free algebra $F_n$, $n\ge 2$.

\begin{Theorem}\label{tTypical} An algebra
$Q\in\V$ is a generic quotient-algebra for $\V$ if and only if $Q^c=\{ 0\}$.
\end{Theorem}

\begin{proof}
First assume that $Q\in\V$ and $Q^c=\{ 0\}$. By Theorem \ref{36}, Claim (i),  a generic $n$-generated algebra $P$  admits a homomorphism $\vp$ onto $F_n/F_n^c$ (with kernel $P^c$). Also $F_n$ admits a homomorphism onto $Q$ whose kernel contains $F_n^c$. Thus we have a homomorphism $\psi:F_n/F_n^c\to Q$. Finally, $\psi\vp: P\to Q$ is a homomorphism of $P$ onto making $Q$ a quotient algebra of $P$. 

Now let us assume that $Q^c\ne\{ 0\}$. In this case, $Q$ possesses a quotient-algebra $Q'$, such that $(Q')^c\ne\{ 0\}$ but whose quotient-algebra by an ideal of dimension 1 is $(c-1)$-step nilpotent. If we prove that $Q'$ is not a generic quotient-algebra in $\V$ then obviously $Q$ is not a generic quotient-algebra in $\V$. So from now on we assume $\dim_\F Q^c=1$.

Let $Q$ have $r$ generators (our number $n$ will turn out to be much greater). We write $Q=F_r/K$ which leads to a presentation $Q=F_n/L$ where $L$ contains all the free generators $x_{r+1},\ld,x_n$. By our choice, the intersection $M=L\cap F_n^c$ has codimension $1$ in $F_n^c$.

Let $d=\dim_\F F_n^c$ and $m\le d-1$. We set $\Y_m=\Grass(m,M)$. We have $\dim \Y_m=m(d-1-m)$. Now the automorphism group of $F_n$ acts on $F_n^c$ as $\GL(n)$ and on $\Y_m$ as the $\cG=\PGL(n)$ (see Proposition \ref{lGLFc}). We now consider $\fZ_m=\cG\times \Y_m$; $\dim \fZ_m =  m (d-1-m) +n^2-1$. Given $(g, U)$, where $g\in \cG$ and $U\in \Y_m$, we set $\vp((g, U)) = g(U)$, which is an $m$-dimensional subspace in $F_n^c$. This defines a polynomial map $\vp:\fZ_m\to \Grass(m, F_n^c)$.

We are going to show that the subset $\M_m$ of ideals $J$ in $\Y_m$ such that $F_n/J$ admits a homomorphism onto $Q$ belongs to $\vp(\fZ_m)$.

Indeed, for such $J$, there is an ideal $L'$, $J\subset L'$, such that $F_n/L'\cong Q$. Therefore, $F_n^c/(L'\cap F_n^c)\cong Q^c$, meaning that $M'=L'\cap F_n^c$ has codimension 1 in $F_n^c$. 
As mentioned in Theorem \ref{tScrap8.4}, there is an automorphism $\alpha$ of $F_n$ such that $\alpha(L)=L'$. Since $F_n^c$ is invariant, this automorphism also maps $M=L\cap F_n^c$ onto $M' = L'\cap F_n^c$. It follows that $g(M)=M'$ for an appropriate $g\in \cG$. Since $J\subset L'\cap F_n^c=M'$, it follows that there is $U\in \Y_m$ such that $g(U)=J$ and $\vp((g,U))=J$, as claimed.

Now we need to estimate the dimension of the image of $\vp$. In order to do this, we will estimate  the dimension of the preimage of $J$. So assume $g(U)=J$. If $h\in\cG$ then  the preimage also contains all pairs $(gh,h^{-1}(U))$. If we only consider the second components of the pairs, it follows that the dimension of the preimage is at least the dimension of the stabilizer of $M$ in $\cG$. Since $M$ contains the generators $x_{r+1},\ld, x_n$, the stabilizer in question contains the subgroup $\PGL(n-r)$, whose dimension is $(n-r)^2-1$.

It now follows by Lemma \ref{LSH}, Claim (c), that the dimension of $\vp(\fZ_m)$ does not exceed the following number:
\begin{eqnarray}
\dim \fZ_m - (n-r)^2 +1&=& m(d-1-m) +n^2-1 - (n-r)^2+1\nonumber\\ &=& m (d-1-m) +2nr -r^2\label{egenfac}
\end{eqnarray}

All values of this function of $m$ do not exceed $\frac{(d-1)^2}{4}+2nr-r^2$. Now the dimension of the variety of generic ideals equals $\frac{d^2}{4}$, if $d$ is even and $\frac{d^2}{4}-\frac{1}{4}$ if $d$ is odd. The difference of $\frac{d^2}{4}-\frac{1}{4}$ and $\frac{(d-1)^2}{4}+2nr-r^2$ equals
\[
\frac{d}{2}-2nr+r^2-\frac{1}{2}>0,\mbox{ for all sufficiently great }n,
\]
because when $c\ge 2$, by Lemma \ref{27.9}, the growth of $\frac{d}{2}$  is at least quadratic (see Lemma \ref{27.9}).  Thus there is $n_0=n_0(Q)$ such that for all $n\ge n_0$ the dimension  of the variety of ideals mapping onto $Q$ is strictly less than the dimension of the  variety of all ideals in $F_n^c$. So we proved that if an algebra $P=F_n/J$ admits a homomorphism onto $Q$ then $J$ is not a generic ideal. Actually, we have proven that for any $m$, $\dim\M_m\le \dim(\vp(\fZ_m))<\frac{d^2}{4}-1$.

Now let us switch from the ideals to their respective quotient-algebras. We consider inside $\Y_m$ the quasi-projective subvariety $\cQ_m$  whose points are $m$-dimensional ideals stable only under the automorphisms which are scalar modulo $F_n^2$ (use Theorem \ref{tAAI}). The ideals $J$ in $\cQ_m$ such that $F_n/J$ maps onto $Q$ are a subvariety in $\cQ_m\cap\vp(\fZ_m)$ whose dimension is less then $\frac{d^2}{4}-1$.

According Key observations, the dimension of varieties of ideals in Classes 1, 2 and 3a is always less than $\frac{d^2}{4}-n^2$. The same can then be said about the varieties of quotient-algebras by these ideals.

Let us consider Class 3b and compare the dimensions of varieties of algebras $F_n/K$, where $K\in \M_m$ on one hand and $K$ is of Class 3b on the other. In both cases $\Stab_{\Aut{F_n}}K=\Scal(F_n)$. Therefore, the dimensions of varieties of algebras arizing after factorization by the action of $\Aut{F_n}$, which amounts the action of $\PGL(n)$ are equal to the dimensions of varieties of ideals minus $n^2-1$. 

From these estimates and (\ref{dimGenId}), it follows that a generic algebra does not admit a homomorphism onto $Q$.
\end{proof}

\begin{Remark}\label{rInfNumRanks} We have proved slightly more than had been stated. Namely, that there no algebras $Q$ which are generic quotient-algebra for \textit{infinitely many} values of $n$ (not for all $n>n_0$, as stated in the theorem).
\end{Remark}

\subsection{Width of verbal ideals in generic algebras}\label{ssWGI}

Given an algebra $P$ over a field $\F$ and any nonassociative polynomial $f=f(x_1,\ld,x_q)\in\cF_\infty$, one can consider a verbal ideal $f(P)$ of $P$ generated by $f$. If $\V$ is the primitive class defined by the identical relation $f(x_1,\ld,x_q)=0$ then $f(P)$ is the  verbal ideal of $P$ corresponding to $\V$. Let $\overline{f(P)}$ stands for the closure of the set of all values $f(a_1,\ld,a_q)$, $a_1,\ld,a_q\in P$, under all possible multiplications by the elements of $P$. Then  $f(P)$ is the  set of all linear combinations of the elements in $\overline{f(P)}$.

An element $a\in f(P)$ is said to have \textit{$f$-width} $m$ if $a$ is a linear combination of $m$ elements from $\overline{f(P)}$ and $m$ is minimal with this property. We write $m=w_f(a)$. 

Finally, if exists, the number 
\[
w_f(f(P))=\max\{ w_f(a)\,|\,a\in f(P)\} 
\]
is called the \textit{width} of $f(P)$. If not, we write $w_f(f(P))=\infty$.

For example, for any algebra $P$ the $f$-width of an element $a\in P^2$ with respect to $f=xy$ is equal to the least number $m$ of nonzero summands in the expressions of the form $a=\sum_{i=1}^m b_ic_i$, where  $b_i,c_i\in P$. In the case of Lie algebras the width with respect to $xy$ is called the \textit{commutator width}. It is known to be finite for any finitely generated Lie algebra. For many finitely generated solvable Lie algebras the commutator width of the commutator subalgebra has been found by V. A. Romankov \cite{Rom}.

The definition of the width of $f(P)$ naturaly extends to the width with respect to arbitrary sets $\f$ of nonassociative polynomials and verbal ideals $\f(P)$. We use the notation $w_\f(a)$ and $w_\f(\f(P))$. 

\begin{Lemma}\label{ltwowidth}
\begin{enumerate}
\item[$\mathrm{(a)}$] If two finite sets of elements $\f$ and $\f'$ define one and the same primitive class $\W$ (respectively, one and the same subclass  $\W$ of a primitive class $\V$) then there exists a positive constant $C$ such that $w_{\f'}(\f'(P))\le C\, w_{\f}(\f(P))$ for any algebra $P$ (respectively, any $P\in\V$). 
\item[$\mathrm{(b)}$] If the base field $\F$ is infinite and $f$ contains a nonzero linear part then there exists a positive $C$ such that $w_f(f(P))\le C$ for any algebra $P$. 
\item[$\mathrm{(c)}$] If a primitive class $\V$ is nilpotent and $f$ contains a nonzero linear part then for any field and any nonzero $P\in\V$ we have $\overline{f(P)}=P$ and  $w_f(f(P))=1$
\end{enumerate}
\end{Lemma}
\begin{proof} \begin{enumerate}
\item[(a)]Since the verbal ideals of $\cF_\infty$ (respectively, $F_\infty(\V)$) generated by $\f$ and $\f'$ coincide, each polynomial in $\f$ can be written as  $\sum_i\lambda_iv_i$ where each $v_i$ is obtained from a value of a polynomial in $\f'$ using repeated multiplications by the elements of $\cF_\infty$ (respectively, $F_\infty(\V)$). Let $C$ bound from above the number of summands for all polynomials in $\f$ written  through $\f'$. Choose $a\in P$ and write it as an element of $\f(P)$. Then replace in this expression each polynomial of $\f$ by its expression through the polynomials of $\f'$. Then the number of summands in the resulting expression of $a$ through $\f'$ will get increased by at most $C$ times. Thus we have proved (a).
\item[(b)] By Lemma \ref{L1.3}, $f$ defines the same primitive class $\V$ as the set of its multihomogeneous components, hence $x=0$ is satisfied in $\V$. The width of any algebra with respect to $x$ equals 1, so by (a) the width with respect to $f$ will be restricted by $C$ equal to the number of summands of $f$.
\item[(c)] Follows from \cite[Theorem 2.4]{BO}.
\end{enumerate}
\end{proof}

In the remainder of this section, we assume that the field $\F$ is algebraically closed. Given a $c$-step nilpotent variety $\V$, $c\ge 2$, we denote by $F_n$ the $\V$-free algebra of rank $n$ and  by $\V_n$ the variety of (isomorphism classes of) $n$-generated algebras in $\V$.  The main result of this section is the following.

\begin{Proposition}\label{pLabel}
Let $f=f(x_1,\ld,x_t)$ be a nonassociative polynomial which is not an identity in $F_n$. Suppose $f$ does not contain monomials depending on one variable only. Then there is a constant $\rho=\rho(f,c)$, such that $w_f(f(P))>\rho n$ for a generic algebra $P\in \V_{n}$. 
\end{Proposition}

\begin{proof} We first assume that $f$ is homogeneous of degree $c$. By Lemma \ref{ltwowidth}, Claim (a), the proof reduces to the case where all monomials in $f$ are written through the same set of variables, depending on all of $x_1,\ld,x_t$, $t\le c$.  It then follows that $f$ defines a map 
\[
\Phi: \underbrace{P\times\cdots\times P}_t\to \overline{f(P)}\subset f(P).
\]
This map factorizes through
\[
\overline{\Phi}: U\to f(P)\mbox{ where }U= \underbrace{P/P^2\times\cdots\times P/P^2}_t\to \overline{f(P)}\subset f(P).
\]
 
 The map $\overline{\Phi}$ is given by homogeneous polynomials with coefficients expressed in terms of the structure constants of $P$. Thus $\overline{\Phi}$ defines a polynomial map of the affine space $\mathbb{P}^{nt}$ into $f(P)$, whose image is $\overline{f(P)}$. It then follows that $\overline{f(P)}$ is an affine variety whose dimension  is bounded from above by a linear function $tn$ in $n$.
 
Now for every $k$, we can consider a  map $\Psi:\overline{f(P)}\times\cdots\times \overline{f(P)}\to f(P)$ mapping $(u_1,\ldots,u_k)$ to $u_1+\cdots+ u_k$. If $w_f(f(P))\le k$, then the image of  $\Psi$ is the whole of $f(P)$. Note that $\dim_\F\Psi(\overline{f(P)}\times\cdots\times \overline{f(P)})\le k\dim \overline{f(P)}\le ktn$.

Since by Lemma \ref{lFor Width}, $\dim_\F f(P)$ as a function of $n$  is bounded from below by a quadratic function, the width $k=k(n)$ of $f(A)$ must be at least a linear function $\rho n$, for some constant $\rho$. This completes the proof in the case where $f$ is homogeneous of degree $c$.

Also, since by our hypotheses all monomials in $f$ have degree at least 2,  the case $c=2$ our Proposition is automatically true for any polynomials, without assumption that they have the same degree. 

If $c>2$, we write $f=f_2+\cdots+f_{c-1}+f_c$, where $f_i$ is the homogeneous component of degree $i$, $i=2,\ld,c$.
 
{\sc Case 1}. If $\bar f=f_2+\cdots+f_{c-1}=0$ holds identically in $F_n/F_n^c$ then there is a homogeneous polynomial $f'_c$ of degree c such that $f_2+\cdots+f_{c-1} = f'_c$ is an identity in $F_n$. Since $f$ and a homogeneous polynomial $f'_c+f_c$ take the same values on algebras in $\V$, we can use the assumption at the very beginning of the proof to complete the proof in this case. 

{\sc Case 2}. $\bar f=f_2+\cdots+f_{c-1}=0$ is not an identity in $F_n/F_n^c$. Then $w_{\bar f}({\bar f}(F_n/F_n^c))>{\bar\rho}n$ and then $w_f(f(F_n))>{\bar\rho}n$. Finally, since by Theorem \ref{tTypical}, Claim (i), a generic algebra in $\V_n$ homomorphically maps onto $F_n/F_n^c$, the verbal $\bar f$-width of $P$ is at least the verbal $f$-width of $F_n/F_n^c$, proving our proposition.
\end{proof}

The following claim does not involve generic algebras.
\begin{Corollary}\label{cTidFreeAlg}
Let $\V$ be an arbitrary primitive class of algebras over an infinite field $\F$, containing algebras with nonzero product. Assume that $f$ is the linear combinations of monomials of degree $c\ge 2$ none depending on one variable only. If $f$  is not an identity in $\V$ then the verbal ideal of $F_\infty=\cF_\infty(\V)$ generated by $f$ has infinite $f$-width.
\end{Corollary}

\begin{proof}
If $\V$ is $c$-step nilpotent, $c\ge 2$, this follows from Proposition \ref{pLabel} since for any $n>n_0$, a generic algebra is a  is a homomorphic image of $F_n(\V)$. 

In a general case, since $\F$ is infinite, $\cF_\infty$ is graded by its homogeneous components (see (\ref{eGrading}), so that $\bigcap_{i=1}^{\infty} F_{\infty}^i = \{0\}$ (see (\ref{eApprox})). In this case, there is $c\ge 2$ such that $f$ is not an identity in $F_\infty/F_{\infty}^{c+1}$ but an identity in $F_{\infty}/F_{\infty}^{c}$. Since our claim is true for the relatively free algebra $F_{\infty}/F_{\infty}^{c+1}$, it is certainly true for $F_{\infty}$.  
\end{proof}

\begin{Remark}\label{rBreadth}
\begin{enumerate}
\item[(a)] Using Proposition  \ref{pLabel} and the fact that a generic algebra in $\V_n$ is a homomorphic image of a $\V$-free algebra  $F_n(\V)$, we conclude that  there is a constant $\sigma$ such that $w_f(f(F_n(\V)))>\sigma n$ for a  in $\V$. 

\item[(b)] The conclusion of Proposition \ref{pLabel} will not change (up to some possible constant factors) if we replace $f$ by a finite set $\f=\{f_1,\ld,f_t\}$ of polynomials with the same properties as $f$.
\end{enumerate}
\end{Remark}
\begin{Corollary}\label{cBreadth}
If a polynomial $f$ over a field of characteristic zero is not an identity in a $c$-step nilpotent primitive class $\V$, $c\ge 2$, and $f$ has no monomials of degree 1 then there exists an ascending linear function $\psi$ such that $w_f(f(P))\ge\psi(n)$ in a generic algebra in $\V_n$. Also, it is true that for all $n$,we have $w_f(f(F_n(\V)))\ge\psi(n)$.
\end{Corollary}
\begin{proof}
By Proposition \ref{pMultilin}, $f$ defines the same primitive class as a finite set $\f=\{f_1,\ld,f_t\}$ of multilinear polynomials of degrees $\ge 2$, hence without monomials depending on one variable. The remainder of the proof is along the lines of Proposition \ref{pLabel}, considering that $f(P)=f_1(P)+\cdots+f_t(P)$, which can decrease the coefficient of the estimate obtained in that proposition by at most $t$ times. 
\end{proof}

\begin{Remark}\label{r12}
In the general case, the linear lower estimate of $f$-width obtained in Corollary \ref{cBreadth} cannot be improved. For example, in the associative case, the $xy$-width of $P^2$ of any algebra $P$ with generators $a_1,\ld,a_n$ is always bounded by $n$ since every element in $P^2$ can be written as the sum $u_1a_1+\cdots+u_na_n$, for some $u_1,\ld,u_n \in P$. Similar equation can be derived in any Lie algebra, if we use the anticommutativity and the Jacobi identity.

At the same time, let $\chr\F=0$ and $f$ is a polynomial without terms of degree $\le i$. Then our arguments are sufficient to find a \textit{lower} bound by a polynomial of degree $i$ in $n$  for the $f$-width in all $c$-step nilpotent generic algebras. In some cases, like associative or Lie algebras and  $f=x_1x_2\cdots x_{i+1}$ one can find an \textit{upper} bound for the $f$-width.
\end{Remark}
\begin{Remark}\label{r13}
If $\chr \F=p>0$ then an easy counterexample to Corollary \ref{cBreadth} is given by $f=x^p$. If $P$ is a  commutative and associative algebra then any linear combination of $p$-powers is a $p$-power so that the $f$-width of $P^p=x^p(P)$ equals 1.
\end{Remark}

\section{Generic nilpotent algebras over finite fields}\label{sGAFF}

In this section, $\F$ is a finite field of $q$ elements. One more stipulation is that we only consider \textit{central} primitive classes.

\begin{Definition}\label{dGenericSets} Let $\{S_n\,|\,n=1,2,\ld\}$ be an infinite sequence of pairwise disjoint finite sets and $\{T_n\}$ a sequence of subsets $T_n\subset S_n$. Let $S=\bigcup_{n=1}^\infty S_n$ and $T=\bigcup_{n=1}^\infty T_n$. We say that $S\setminus T$ is \textit{negligible} (\textit{exponentially negligible}) if  
$\frac{|T_n|}{|S_n|}\to 1$ ($\frac{|T_n|}{|S_n|}\to 1$ faster than $1-q^n$ for some  $0\le q<1$), as $n\to\infty$. One could also say that a \textit{generic element} of $S$ is in $T$.
\end{Definition}

In this section the role of $S_n$ will be played by one of two objects. This could be the set $\V_n$ of the isomorphism classes of $n$-generated algebras in a $c$-step nilpotent primitive class $\V$ over a finite field $\F$. The other option is the set $\cJ_n$ of ideals of a relatively free nilpotent algebra $F_n$.

 All $n$-generated algebras in $\V$ have the form of $F_n/K$, where $K\in\cJ_n$.  The connection between the two cases is that by Proposition \ref{lGLFc} and Theorem \ref{tScrap8.4}, there is one-one correspondence between the algebras in $\V_n$ and the orbits of the natural action of $\Aut{F_n}$ on $\cJ_n$.  In the particular case of quotient-algebras modulo the ideals in $F_n^c$, the isomorphism classes of such algebras are in one-one correspondence with the orbits of the natural action of the groups $\GL(n)$ on the set of subspaces of $F_n^c$ (see Proposition \ref{lGLFc}). Therefore, the results about the generic (properties of) $n$-generated algebras in $\V$ are closely connected to the results about the generic (properties of) the ideals  of the sequence of algebras $F_n$.

In  Section \ref{ssENI} we give some estimates for the number of ideals in finite algebras. In Section \ref{prop1}  we examine the generic properties of the ideals of $F_n$. Finally, on the basis of this, in Section \ref{prop2}, we turn our attention to the generic properties of $n$-generated algebras in $\V$.

In the study of ideals of $F_n(\V)$, we  denote by $\cI_n$ the set of ideals inside the annihilator ideal  $I=I(F_n)$ of $F_n$. We are going to prove that  the sequence $\frac{|\cJ_n\setminus\cI_n|}{|\cI_n|}$ converges to 0 as an exponential function as $n\to\infty$. 

Then we consider the set $\cN_n$ of ideals each of which is invariant under an $NS$-automorphism. We will show that the sequence of numbers $|\cN_n|$ is exponentially negligible when compared with $|\cJ_n|$. So generic ideals are invariant only under the automorphisms which are scalar modulo $F_n^2$.

Our results about the structure of generic algebras and their automorphisms in this section are quite similar to the results of Section \ref{sMR}, where we have assumed that the base field was algebraically closed.

\subsection{Estimates of the number of ideals}\label{ssENI} Using the same notations as in the previous sections, we first estimate from below the number of $m$-dimensional ideals (=subspaces) belonging to $F_n^c$. Recalling $d=\dim_\F F_n^c$, and setting $\exp(m)=q^m$, we find the number of ideals of dimension $m\le d$ as the product below.
\[
\frac{(\exp(d)-1)\cdots(\exp(d)-\exp(m-1))}{(\exp(m) -1)\cdots(\exp(m)-\exp(m-1))}
\]
Now for all $i$, $0\le i\le m-1$, we have 
\[
\frac{(\exp(d)-i)}{(\exp(m) -i)}\ge \exp(d-m),
\]
for all $i=0,\ld,m-1$.
 So the number of ideals of dimension $m$ contained in $I$ is bounded from below by $\exp((d-m)m)$. If $m=\left[\frac{d}{2}\right]$ (also $m=\left[\frac{d}{2}\right]+1 $ in the case where $d$ is odd)  then the number of ideals of such dimension in $I$ is greater than 
\begin{equation}\label{eff*}
\exp\left(\left[\frac{d^2}{4}\right]\right).   
\end{equation}

For the estimate from above, we note that for any $i$, $0\le i\le m-1$,
\begin{eqnarray*}
\frac{\exp(d) -\exp(i)}{\exp(m) - \exp(i)}&\le& \exp(d-m)+ 2\exp(d-2m+i)\\&=& \exp(d-m)(1+2\exp(i-m)).
\end{eqnarray*}
So the product of all these fractions over all $i$ is bounded from above by 
\[
\exp((d-m)m)\prod_{i=0}^{m-1}(1+2\exp(i-m)).
\]
Note if $q\ge 2$, the product $(1+2q^{-1})(1+2q^{-2})\cdots$ converges and does not exceed 
\begin{equation}\label{effqq}
e^{\left(2q^{-1}+2q^{-2}+...\right)}\le e^2<8.
\end{equation} 
Considering the estimate (\ref{effqq}), we have for the number of $m$-dimensional ideals in $I$ an estimate from above by the number
 \begin{equation}\label{hash}
8\exp\left((d-m)m\right). 
 \end{equation}
Our further inequalities hold if $n$ is sufficiently great.  

\subsection{Ideals in finite relatively free nilpotent algebras}\label{prop1} 

Recall the annihilator series (\ref{eAS}) in an algebra $P$ over $\F$. We write $I_m=I_m(F_n)$ and following the first paragraph of Section \ref{ssINA}, we choose $h$ the least number such that $I_{h}=F_n$. In the proof below we keep using the annihilating series  to be closer to the proofs in the case where the base field is algebraically closed (Section \ref{ssTNA}). At the same time, because $\V$ is central, we have $d_1=\dim_F I(F_n)=\dim_F F_n^c=d$.

The main result in this section is the following

\begin{Theorem}\label{tfGI}
Let $F_n=F_n(\V)$ be a $\V$-free algebra, where $\V$ is a central $c$-nilpotent primitive class over a finite field $\F$. Then being in $F_n^c$ is a generic property for the ideals of the sequence of algebras $\{ F_n\,|\, n=1, 2,\ld\}$.
\end{Theorem} 

\begin{proof}

Let $N$ be an ideal of $F_n$ which is not in $I(F_n)=F_n^c$. If $U=N\cap I$ then $N$ is spanned as a vector space by $U$ and and a union of the following sets of its vectors. First, $s_1$ vectors, where $\mathbf{s} =(s_1,s_2,..., s_{h-1})$, in $N\cap I_2$ linearly independent modulo $I=I_1$, then $s_2$ vectors in $N\cap I_3$ linearly independent modulo $I=I_2$, and so on, ending by $s_{h-1}$ vectors in $N\cap I_{h}=N$ which  are linearly independent modulo $I=I_{h-1}$. We set $s=s_{h-1}+\cdots+s_2+s_1$. We have $s\ge 1$.

Note the the results of Section \ref{sGRFNA} concering the estimates of the dimensions of subspaces in relatively free algebras hold valid for any field $\F$. So we are free to apply Lemmas \ref{27.4}, \ref{27.18}, \ref{27.24} to our current situation where $|\F|=q<\infty$.

Using the argument based on Lemma \ref{27.24} just before formula (\ref{eDPP}), we select in $U$ the subspace $L$ with $\dim_\F L\ge \gamma s \frac{n}{c}$, which is contained in the ideal $I'$ generated by the set $\textbf{s}$ of $s$ vectors chosen above. If we set $j=\dim_\F N$ then for the dimension of $U$, we will obtain $\dim_\F U=j-s\le d$. As a result, to obtain a generating set of $N$, as an ideal of $F_n$, we have to add to the collection of $s$ vectors, as above, a collection of $\le j-s - \gamma s \frac{n}{c}$ vectors from $U$. For each subspace $L$ in $I$ we fix its direct complement $L^*$ in $I$. One can choose these vectors from the subspace $L^*$ of dimension $\le d-\gamma s \frac{n}{c}$. The number of ways to select one vector from $L^*$ is at most $\le \exp\left(d-\gamma s \frac{n}{c}\right)$. The number of choices of all necessary vectors is then at most 
\begin{eqnarray}\label{eff**}
&&\exp\left(\left(j-s-\gamma s \frac{n}{c}\right)\left(d-\gamma s \frac{n}{c}\right)\right)\nonumber\\&\le& \exp\left(\frac{\left(d-\gamma s\frac{n}{c}\right)^2}{4}\right) < \exp\left(\frac{d^2}{4} - sd\gamma \frac{n}{4c}\right).
\end{eqnarray}
Note that the upper bound is uniform for all $j$.

We conclude that the ideal $N$ is completely defined by the choice of generating set, as described earlier. $F_n=F_n(\V)$. Let us set $f(n)=\dim_\F F_n(\V)$. Then the choice of one vector in $F_n$ is bounded by $\exp(f(n))$, while the choice of $s$ vectors is bounded by $\exp (sf(n))$. Now the choice of additional vectors inside $I$ is estimated in (\ref{eff**}). As a result, the number of ideals of dimension $j$ not belonging to $I$ is at most 
\[
\exp\left(\frac{d^2}{4} - sd\gamma \frac{n}{4c} + sf(n)
\right).
\]
Taking summation over all $j$ from $1$ to $f(n)$, we will obtain that the number of ideals not belonging $I$ is bounded from above by

\begin{equation}\label{eff***}
\exp\left(\frac{d^2}{4} - sd_1\gamma \frac{n}{4c} + sf(n)+\log_q f(n)
\right).
\end{equation}  
It is important in this estimate that since $s\ge 1$, similar to our argument in the proof of Lemma \ref{lAVM}, we have a \textit{negative term} $sd_1 \gamma \frac{n}{4c}$. Using Lemma \ref{27.4} with the estimates for $d$ from below and $\dim_\F F_n$ from above, the value $d\gamma \frac{n}{4c}$ is bounded from below by a polynomial of degree $c+1$ in $n$ while $f(n)$ is bounded from above by a polynomial of degree $c$ in $n$. Thus the ratio of the value in (\ref{eff***}) to the value in (\ref{eff*}) tends to zero as an exponential function as $n$ tends to $\infty$. Thus the property of an ideal of $F_n$ being in $I$ is generic for algebras $F_n$, moreover, it is \textit{exponentially generic}.
\end{proof}

\begin{Remark}\label{rff1} Comparing with the argument in Section \ref{ssINA}, we note that, in the case of finite fields, there is no need to perform summation over all tuples $\mathbf{s} =(s_{h-1},\ld,s_1)$ because the choice of $s$ vectors  uniquely defines the tuple $\textbf{s}$. No need also to take summation over all $\varpi U_i$ because we may assume that $L$ contains all these subspaces.
\end{Remark}

\subsection{Automorphisms of finite generic nilpotent algebras}\label{prop2}

We first estimate the number of pairwise nonisomorphic  $n$-generated algebras in $\V$. Since $I=F_n^c$, the group $\GL(n)$ acts on $I$ (see Proposition \ref{lGLFc} ). Further, $F_n/K\cong F_n/L$ where $K$ and $L$ are ideals of $F_n$, if and only if $K$ is moved to $L$ under the action of $\Aut{F_n}$ (see Lemma \ref{tScrap8.4}). If $K,L$ are in $I$ then $F_n/K\cong F_n/L$ iff $K,L$ belongs to the same orbit of the action of $\GL(n)$ on the set of subspaces of $I$. It follows that to estimate from below the number of algebras  of the form $F_n/K$, $K$ is in $I$ one has to divide the estimate of the number of all ideals in $I$ by the order of $\GL(n)$. So by (\ref{eff*}) the estimate will be at least  
\begin{equation}\label{effAmodGL}
\exp\left(\left[\frac{d^2}{4}\right]-n^2\right).
\end{equation}
Since the absolute value of the negative term in (\ref{eff***}) is bounded from below by a polynomial of degree $c+1$ in $n$, the total value in (\ref{eff***}) is exponentially negligible if compared with the value (\ref{effAmodGL}). It follows that a generic algebra is obtained when $K\subset I$. As a result, the algebras $P$ such that $P/P^c$ is free in the primitive class $\V'=\V\cap\fN_{c-1}$ of $(c-1)$-step nilpotent algebras of $\V$ are exponentially generic in $\bigcup_{n=1}^{\infty}\V_n$.
\begin{Theorem}\label{tffIdeal} Let $\V$ be a $c$-step nilpotent, $c\ge 2$, central primitive class of algebras over a finite field $\F$, $\V'=\V\cap\fN_{c-1}$ a primitive sublass of algebras $B$ with $B^c=\{ 0\}$. Then for an $n$-generated generic algebra $P\in\V$ it is true that $P/P^c$ is a free $\V'$-algebra of rank $n$.

\end{Theorem}

We now switch to the ideals  invariant under some $NS$ automorphisms of $F_n=F_n(\V)$. Following the notation introduced just before Lemma \ref{34.3}, based on Proposition \ref{lGLFc}, with each automorphism $\vp$ of $F_n$, we can associate a matrix $P\in\GL(n)$ of the action of $\vp$ on $F_n/F_n^2$. Since we deal with central primitve classes, the restriction of $\vp$ to $I$ depends only on $P$ and is denoted by $\wA$.
\setcounter{Case}{0}
\begin{Case}\label{st1}
The number of all ideals outside $I$, not necessarily invariant under the $NS$-automorphisms have been determined  bounded from above by (\ref{eff***}), where $s=1$. 
\end{Case}

\begin{Case}\label{st2}
Let $N$ be an ideal belonging to $I$ such that $\dim_\F N=m$, where 
\begin{equation}\label{effm}
\frac{d}{2}-2n<m<\frac{d}{2}+2n.
\end{equation} 
Suppose that $N$ is invariant under an automorphism $\vp$ whose action on $F_n/F_n^2$ is an non-scalar matrix $P$. Assume that $\rank_{\wA}(N)\le m-n^{\frac{3}{2}}$. (We will use later that the growth of $n^{\frac{3}{2}}$ is slower than $n^2$ but the growth of its square is faster than  $n^2$.) Then, with $P$ fixed, $N$ is defined, as an ideal, by a subspace of dimension at most $m-n^{\frac{3}{2}}$. Since $n$ is big, if we use (\ref{effm}), we will obtain $m-n^{\frac{3}{2}}<\frac{d}{2}-\frac{n^{\frac{3}{2}}}{2}$. Now using formula (\ref{hash}), we derive that for all great enough $n$, the number of  subspaces in our case does not exceed
\begin{eqnarray*}
&&8\exp\left(\left(m- n^{\frac{3}{2}}\right)\left(d-\left(m - n^{\frac{3}{2}}\right)\right)\right)\\ &<& 
8\exp\left(\left(\frac{d}{2}-\frac{n^{\frac{3}{2}}}{2}\right)\left(\frac{d}{2}+\frac{n^{\frac{3}{2}}}{2}\right)\right)\\
&<&\exp \left(\frac{d^2}{4} - \frac{n^3}{5}\right).
\end{eqnarray*}
(Note that we get rid of 8 here and few times later again  because  multiplying by 8 only adds a constant to the exponent, which is negligible as $n\to\infty$.)
After summation over all $A\in\GL(n)$, we will find that the number of such $N$ is bounded from above by 
\begin{equation}\label{hh}
\exp \left(\frac{d^2}{4} - \frac{n^3}{5}+n^2\right).
\end{equation}
\end{Case}

\begin{Case}\label{st3}
As in Case \ref{st1}, let $N$ be an ideal contained in $I$ and such that $\dim_\F N=m$, where $m$ as in (\ref{effm}). Further assume that $N$ is invariant under an automorphism $\vp$ whose action on $F_n/F_n^2$ is an  matrix $A$. This time we assume that $\rank_{\wA}(N)> m-n^{\frac{3}{2}}$.  Also, since  $m$ is bounded from below by a quadratic function (see the lower bound for $d$ in Lemma \ref{27.4}, it is true that for all great enough $n_0$, we have $m-n^{\frac{3}{2}}>\frac{m}{2}$. So Lemma \ref{28.24} applies and $N$ contains a subspace $W$ which is an eigenspace for $\wA$ with eigenvalue $\lambda$ and such that $\dim_\F W>m- n^{\frac{3}{2}}$. On the other hand, by Lemma \ref{34.3}, there is a proper subspace $L$ of $I$ such that $W\subset L$ and $\dim_{\F} L \le d-\delta n^{c-1}$.

By formula (\ref{hash}), the number of ways to choose the subspace $W$ in $L$ is less than
\begin{eqnarray*}
&&8 \exp \left(\left(m- n^{\frac{3}{2}}\right)\left(d-\delta n^{c-1}-\left(m- n^{\frac{3}{2}}\right)\right)\right)\\ &\le& \exp\frac{(d-\delta n^{c-1})^2}{4}\le \exp\left(\frac{d^2}{4} - d\delta \frac{n^{c-1}}{4}\right).
\end{eqnarray*}

Once $W$ has been selected, to complete the selection of $N$, 
one has to choose at most $n^{\frac{3}{2}}$ vectors from $I$, which can be done by less than $\exp \left(n^{\frac{3}{2}}d\right)$ ways. As a result, when $A$ is fixed, the number of subspaces $N$ in question, is less than
\[
\exp\left(\frac{d^2}{4} - d\delta \frac{n^{c-1}}{4} + d n^{\frac{3}{2}}\right).
\]
After summation by all $A\in\GL(n)$ and $m$ in (\ref{effm}), we find that  the number of subspaces $N$ in this case is less than  
\begin{equation}\label{hhh}
\exp\left(\frac{d^2}{4} - d\delta\frac{n^{c-1}}{4} + d n^{\frac{3}{2}}+n^2+\log_q 4n\right).
\end{equation}

\end{Case}

\begin{Case}\label{st4}
Finally, it follows by formula (\ref{hash}) that the number of ideals $N$ belonging to $I$ of dimension $m$,where $m$ is outside of (the interval \ref{effm}), is less than $8\exp\left( \frac{d^2}{4} - (2n)^2\right)$. The total number, for all such $m$, is less than 
\begin{equation}\label{hhhh}
8\exp\left( \frac{d^2}{4} - (2n)^2+ \log_q d\right).
\end{equation}
\end{Case}

In all four cases, the numbers of ideals, given by the formulas (\ref{eff***}), (\ref{hh}), (\ref{hhh}) and (\ref{hhhh}) are exponentially negligible with $n\to\infty$ when compared to the total number of ideals estimated from below in (\ref{eff*}). It follows that the number of ideals which are invariant only under the automorphisms which are scalar modulo $F_n^2$ are exponentially generic.

As several times earlier, we can use Theorem \ref{tScrap8.4} to switch from the ideals to the quotient-algebras by these ideals. We need to identify the ideals which are conjugate under the action of $\Aut{F_n}$. As a result, the total number of pairwise nonisomorphic $n$-generated algebras in our primitive class $\V$ is bounded from below by the ratio of the value in (\ref{eff*}) over  the number $\exp(n^2)$. So for the total number of pairwise nonisomorphic algebras is estimated from below by the same number $\exp\left(\left[\frac{d^2}{4}\right]-n^2\right)$, as in (\ref{effAmodGL}).   

This number is still exponentially greater than the numbers of ideals obtained in all four cases. So we conclude that the number of algebras in $\V_n$ whose automorphism group contains an $NS$-automorphism is exponentially negligible when compared with the total number of algebras in $\V_n$, as $n\to\infty$.

\begin{Theorem}\label{tffAut}Let $\V$ be a $c$-step nilpotent central primitive class of algebras over a finite field $\F$, $c\ge 2$. Then the automorphisms of a finitely generated generic algebra $P\in\V$ are scalar modulo $P^2$.
$\hfill\Box$
\end{Theorem}

Similar to the way we did this in Section \ref{ssARFNA}, see Propostion \ref{tNSD}, we can derive the following.
\begin{Proposition}\label{pffDer} Let $\V$ be a $c$-step nilpotent central primitive class of algebras over a finite field $\F$, $c\ge 2$. Then the derivations of a finitely generated generic algebra $P\in\V$ are scalar modulo $P^2$.
$\hfill\Box$
\end{Proposition}

\subsection{The automorphism group of a generic nilpotent algebra}\label{ssOAGFGA}

Given a relatively free algebra $F_n=F_\V(x_1,\ld,x_n)$ over a field $\F$, $|\F|=q$, we call an element $f\in F_n$ \textit{quasihomogeneous} if, when written as a polynomial in $x_1,\ld,x_n$, all its monomials have the same degree modulo $q-1$. For each $\bar k\in\Z_{q-1}$, we have a subspace $F_n^{\bar k}$ spanned by the monomials of all degrees $\ell$ such that $\ell\in\bar k$. We have $F_n=\sum_{\bar k\in\Z_{q-1}}F_{\bar k}$. We can write $f(x_1,\ld,x_n)=\sum_{\bar k\in\Z_{q-1}}f_{\bar k}(x_1,\ld,x_n)$, where $f_{\bar k}\in F_{\bar k}$. Let us prove that this sum is direct. By contradiction, assume 
\begin{equation}\label{effquasi}
f(x_1,\ld,x_n)=\sum_{\bar k\in\Z_{q-1}}f_{\bar k}(x_1,\ld,x_n)=0
\end{equation}  
so that the number of nonzero summands is the least possible, different from zero. Suppose that $\lambda\in\F$ is the generator of the (cyclic) multiplicative group $\F^*$. If $\bar k\ne\bar\ell$ then $\lambda^k\ne\lambda^\ell$.

Now (\ref{effquasi}) is an identity in $\V$ (see Lemma \ref{lRelX}). If we apply an automorphism $\psi$ of $F_n$ sending each $x_i$ to $\lambda x_i$ to (\ref{effquasi}), it remains to be identity. We will have  
\[
f(\lambda x_1,\ld,\lambda x_n)=\sum_{\bar k\in\Z_{q-1}}\lambda^kf_{\bar k}(x_1,\ld,x_n)=0.
\]

If for $k\le\ell$ we have $\bar k\ne\bar\ell$ we have $f_{\bar k}(x_1,\ld,x_n)\ne 0$ and $f_{\bar \ell}(x_1,\ld,x_n)\ne 0$ subtracting from the latter relation (\ref{effquasi}) multiplied by $\lambda^k$, we will have a nontrivial relation with less number of summands. Thus $F_n=\oplus_{\bar k\in\Z_{q-1}}F_{\bar k}$ is indeed the direct sum of quasi homogeneous components. Actually, this is a grading of $F_n$ by the group $\Z_{q-1}$.

It follows that the multiplication of the free generators by a scalar $\lambda\ne 0$ extends to an automorphism of $F_n$ such that  each quasi homogeneous component $F_{\bar k}$ is multiplied by $\lambda^k$. Even if the $c$-step nilpotent algebra $F_n$ is not graded, every element of $F_n^c$ is multiplied by $\lambda^c$ which is true because $F_n^c$ is contained in the graded component $F_{\bar c}$. It now  follows that any ideal $N$ inside $I=F_n^c$ is invariant under such automorphism. Since by Theorem \ref{tffIdeal} a generic algebra $P$ is isomorphic to a quotient-algebra $F_n/N$, where $N\subset I$, we conclude that a generic algebra admits a  well defined faithful action by an automorphism of order $q-1$.   
 
Finally, since the automorphism which are identity modulo $F_n^2$ act trivially on $N$, it follows that in a generic algebra $P$ the map adding to any of $n$ generators an element of $P^2$ extends to an automorphism of $P$. Since by Remark \ref{rScrap29.1}, $I\subset F_n^2$, it follows that $|P/P^2|=|F_n/F_n^2|=q^n$, we obtain a formula for the order of the automorphism group of a generic algebra $P$ in the following.

\begin{Proposition}\label{pffaut} Let $P$ be a generic $n$-generated algebra in a central $c$-step nilpotent, $c\ge 2$,  primitive class of algebras over a finite field $\F$ with $q$ elements. Then
\[
|\Aut{P}|=(q-1)\left(\frac{|P|}{q^n}\right)^n.
\]
The group $\Aut{P}$ is a solvable group which is a semidirect product of a $(c-1)$-step nilpotent normal subgroup and the multiplicative group of $\F$. $\hfill\Box$
\end{Proposition}  
   
In conclusion, we invite the reader to provide the analogues of other claims of Theorem \ref{36}, in the case of finite fields.

\end{document}